\documentclass{ws-ccm_mod}

\usepackage{hyperref}
\usepackage{mathtools}
\usepackage{pdfsync}
\usepackage{mathrsfs}
\usepackage{color}

\newcommand{\eps}{\varepsilon}

\newcommand{\R}{\mathbb R}

\newcommand{\NN}{\mathbb{N}}
\newcommand{\ZZ}{\mathbb{Z}}
\newcommand{\var}{\varepsilon}

\newcommand{\beq}{\begin{equation}}
\newcommand{\eeq}{\end{equation}}

\newcommand{\ds}{\displaystyle}
\newcommand{\med}{- \hskip -1em \int}

\newcommand{\Per}{{\rm Per}}
\newcommand{\ts}{\textstyle}
\newcommand{\ol}[1]{\overline{#1}}

\newcommand{\grad}{\nabla}
\newcommand{\Leb}[1]{\mathscr{L}^{#1}}
\newcommand{\restr}[1]{\lfloor_{#1}}

\newcommand{\dist}{{\rm dist}}
\def\geq{\geqslant}
\def\leq{\leqslant}

\begin{document}

\markboth{B. Galv\~ao-Sousa \& V. Millot} {A two-gradient approach for phase transitions in thin films}

\title{A TWO-GRADIENT APPROACH FOR\\ PHASE TRANSITIONS IN THIN FILMS}

\author{\sc Bernardo Galv\~ao-Sousa}

\address{Department of Mathematics\\ University of Toronto\\
Toronto, ON, Canada
\\
\emph{\texttt{beni@math.toronto.edu}}}

\author{\sc Vincent Millot}

\address{Universit\'e Paris Diderot\\ 
Laboratoire J.L. Lions, UMR 7598, CNRS\\ 
F-75205 Paris, France\\
\emph{\texttt{millot@math.jussieu.fr}}}

\maketitle

\begin{abstract}
{\bf Abstract.} 
Motivated by solid-solid phase transitions in elastic thin films, we perform a $\Gamma$-convergence analysis for a singularly perturbed energy describing second order phase transitions in a domain of vanishing thickness. 
Under a two-wells assumption, we derive a sharp interface model with an interfacial energy depending on the asymptotic ratio between the characteristic length scale of the phase transition and the thickness of the film.
In each case, the interfacial energy is determined by an explicit optimal profile problem. This asymptotic problem entails a nontrivial dependance on the thickness direction when the phase transition is created at the same rate 
as the thin film, while it shows a separation of scales if the thin film is created at a faster rate than the phase transition. 
The last regime, when the phase transition is created at a faster rate than the thin film, is more involved. Depending on 
growth conditions of the potential and the compatibility of the two phases, we either obtain a sharp interface model with scale separation, or a trivial situation driven by rigidity effects. 
\end{abstract}



%
%

\section{Introduction}

In the last few years, many mathematical efforts have been devoted to variational problems arising in the modelling of 
phase transitions in solids, see {\it e.g.} \cite{BJ,CFL,CS1,CS2}. These problems often involve singularly perturbed 
functionals of the form 
\begin{equation}\label{mod1}
\mathbf{E}_\var(\mathbf{u},\mathscr{B})=\int_{\mathscr{B}} \frac{1}{\var}\, W(\nabla \mathbf{u})+\var|\nabla^2\mathbf{u}|^2\,d\mathbf{x}\,,
\end{equation}
where $\mathbf{u}:\mathscr{B}\subset\R^3\to\R^3$ represents the displacement of an elastic body $\mathscr{B}$, $\var>0$ is  a small parameter, and 
$W$ is a (nonnegative) free energy density with multiple minima  corresponding to martensitic materials. Due to the multiple well structure,  
nucleation of phases in a given configuration may occur without increasing  $\int W(\nabla \mathbf{u})$, so that 
the free energy may admit many (eventually constrained) minimizers. 
In order to select preferred configurations, the Van der Waals-Cahn-Hilliard theory adds  higher order terms leading to functionals of the form \eqref{mod1}.  
In such functionals a competition occurs between  
the two terms: 
the free energy favors gradients 
close to a minimum value of $W$, while  $|\nabla^2\mathbf{u}|^2$ penalizes transitions from one minima to another. 

The $\Gamma$-convergence method provides a suitable framework to study the asymptotic behavior of singularly perturbed energies like $\mathbf{E}_\var$ 
(see {\it e.g.} \cite{BD,DalMaso} for a more detailed overview of this subject). One of the first applications of 
$\Gamma$-convergence was actually obtained in \cite{M, MM,Stern1} in the context of fluid-fluid phase transitions (see {\it e.g.} \cite{G}). Here the authors deal with energy functionals of the form 
$\int \frac{1}{\var}\, W(v)+\var|\nabla v|^2\,dx$ where the potential $W$ has a 
double well structure, {\it i.e.}, $\{W=0\}=\{\alpha,\beta\}$. It is shown that such  family of energies  $\Gamma$-converges (in a suitable topology) as $\eps\to0$ to a functional which calculates the area of the interface   
between the two phases $\alpha$ and $\beta$, for limiting $BV$-functions $v$ with values in $\{\alpha,\beta\}$. Since then this result has been generalized in many different ways   
(see {\it e.g.}~\cite{Amb,Bal,BF,Bou,FT,OS}), in particular in~\cite{FM} for an intermediate situation where the singular perturbation $|\nabla v|^2$ is replaced by the higher order term~ $|\nabla^2v|^2$. 
The first $\Gamma$-convergence result for functionals acting on gradient vector fields has been obtained in~\cite{CFL}. Assuming that $\{W=0\}=\{A,B\}$ for some rank-one connected 
matrices $A$ and $B$ (and some additional constitutive conditions on $W$), the authors prove the $\Gamma$-convergence of $\mathbf{E}_\var$ as $\eps\to 0$. Once again the effective functional returns  
the total area of the interfaces separating the phases $A$ and $B$, for limiting functions ${\bf u}$ satisfying $\nabla {\bf u} \in \{A,B\}$ a.e. and $\nabla {\bf u}\in BV$. 
Here the rank-one connection between the wells $A$ and $B$ turns out to be necessary for the existence of non-affine ${\bf u}$'s  satisfying $\nabla {\bf u} \in \{A,B\}$, and the interfaces must be planar and 
oriented according to the connection, see~\cite{BJ}. 
We also mention recent developments on weakening the condition on the wells of $W$ to allow for frame indifference, {\it i.e.}, assuming the zero level set of $W$ of the form $SO(3)A\cup SO(3)B$ (see \cite{CS1,CS2}). 
\vskip5pt

%

Another topic of increasing interest related to solid mechanics concerns thin elastic films.
It is well known that thin films may have different mechanical properties from bulk materials, specifically for martensitic ones.
Those properties are important for many physical applications (see \cite{KJ}).
In this context, the $\Gamma$-convergence point of view is again suitable
to rigorously derive limiting models starting from 3D nonlinear elasticity.
This has been shown in \cite{LDR,BFF}  for the membrane theory, and more recently in~\cite{FJM,FJM2} for nonlinear plate models.   
In the regime of membranes, several studies have focused their 
attention on the impact of a higher order perturbation on the behavior of thin films. The first variational approach has been addressed 
in \cite{KJ} where the authors add the singular perturbation $\eps^2\int |\nabla^2\mathbf{u}|^2$ to the free energy $\int W(\nabla \mathbf{u})$ for a 
domain of small thickness $h$. They obtain in the limit $h\to 0$ a 2D energy 
density which depends on the deformation gradient of the mid-surface, and the Cosserat vector $b$ which gives an asymptotic description of the out-of-plane deformation. 
An important consequence of the results of  \cite{KJ} is that for many interesting materials the low energy states in the thin film limit
are indeed different from the ones in three dimensional samples. 
However  \cite{KJ} does not treat possible correlations between the thickness $h$ and the parameter~$\varepsilon$. This issue was first conducted  
in \cite{Sh}, and more intensively in \cite{FFL} to keep track of  the Cosserat vector. It is shown in  \cite{FFL}  that the limiting model is determined by 
the asymptotic ratio $h/\varepsilon$ as $h\to 0$ and $\eps\to0$, and it depends whether $h/\eps\sim 0$,  $h/\eps\sim \infty$, or $h/\eps\sim 1$. 
\vskip5pt

The general idea of this paper is to study a simple class of singularly perturbed functionals describing  
phase transitions in thin films. In this direction, some models have been recently analyzed, see \cite{BL}, and also \cite{CM,H} 
for models without  singular perturbation leading to sharp interfaces. Here we want to carry out an analysis in the spirit of \cite{CFL} focusing on possible correlations between the strength 
of an interfacial energy and the thickness of the film. As in \cite{KJ,FFL,Sh} we consider a {\it ``membrane scaling"}, and we introduce the normalized functional ${\bf F}_{\eps}^{h}$ defined for ${\bf u} \in H^{2}(\Omega_{h};\R^3)$ by
\begin{equation*}
{\bf F}_{\eps}^{h}({\bf u}) := \frac{1}{h} \int_{\Omega_{h}} W\left(\nabla {\bf u}\right) + \eps^2|\nabla^2  {\bf u}|^2  \,d{\bf x}\,, 
\end{equation*}
where $\Omega_{h}:= \omega \times h I \subset \mathbb{R}^3$, $I:=(-\frac{1}{2},\frac{1}{2})$, and the mid-surface $\omega\subset \mathbb{R}^2$ is a  bounded {\it convex} open set (here convexity is assumed for simplicity, and we refer to \cite{CFL} for more general geometries). Considering configurations 
$\bf u$  with energy of order $\varepsilon$, that is ${\bf F}_{\eps}^{h}({\bf u})\leq O(\varepsilon)$, and renormalizing by $1/\eps$ we are led to the energy $\frac{1}{h}\mathbf{E}_\var$ in the 
thin domain $\Omega_h$. We are interested in the variational convergence of  the family $\{\frac{1}{h}\mathbf{E}_\var(\cdot,\Omega_h)\}$ as $h\to0$ and $\eps \to 0$. 
To this aim, we introduce  the standard rescaling 
$$u(x) = {\bf u}({\bf x})\,\text{ with }\, 
(x_1,x_2,x_3) = \left({\bf x}_1, {\bf x}_2, \frac{{\bf x}_3}{h}\right)\,,$$ 
which yields functionals 
$\{F_{\eps}^{h}\}$ defined for $u \in H^{2}(\Omega;\R^3)$  by
$$F_{\eps}^{h}(u) := \int_{\Omega} \frac{1}{\eps}\,W(\nabla_h u) + \eps |\nabla_h^2 u|^2 \, dx\,,$$
where $\Omega := \Omega_1$,  and $\nabla_h:= \left(\partial_1, \partial_2,\frac{1}{h}\partial_3\right)$ is the rescaled gradient operator. 
\vskip5pt

Our main goal is to perform the $\Gamma$-convergence as $\eps\to 0$ and $h\to 0$ of the family $\{F_{\eps}^{h}\}$ in the simplest context where we can illustrate a difference between the behavior of  thin films and  bulk materials.
The class of models we have in mind involves double-well potentials $W$ of the type considered in \cite{CFL}. In other words, $W$ should be structurally similar to 
\begin{equation}\label{eq:potential}
W(\xi) 
	\approx \dist\big(\xi,\{A,B\}\big)^p
	= \min\big\{ |\xi-A|^p,|\xi-B|^p\big\}.
\end{equation}
%
A situation where a qualitatively different behavior from \cite{CFL} is expected is when $A$ and $B$ are not rank-one connected, but $A'$ and $B'$ are. Here we denote by $A'$ and $B'$ the $3\times 2$ matrices extracted by taking the first two columns from $A$ and $B$, respectively.
In this case, sequences with bounded ${\bf E}_\eps$ energy, as in \cite{CFL}, converge to affine maps by the results in \cite{BJ}. But, as it will be made precise below, this rigidity effect might not occur for sequences with bounded $F_\eps^h$ energy since $A'$ and $B'$ are compatible on the mid-surface.
Accordingly, 
we assume that $W: \R^{3\times 3} \to [0,\infty)$ is continuous and satisfies the following first set of assumptions:
\begin{itemize}
\item[$(H_1)$] $\{W= 0\} = \{A,B\}$ where $A=(A',A_3)\in\R^{3\times 2} \times \R^3$ and $B=(B',B_3)\in\R^{3\times 2} \times \R^3$ are distinct matrices satisfying 
$A'-B' = 2a \otimes \bar \nu$ for some $a \in \R^3$  and $\bar \nu \in \mathbb{S}^{1}$;
\vskip5pt
\item[$(H_2)$] $\frac{1}{C_1}\, |\xi|^p - C_1 \leq W(\xi) \leq C_1\, |\xi|^p + C_1$,
for  $p\geq 2$ and some constant $C_1 > 0\,$.
\end{itemize}

%
%

Under the conditions $(H_1)$ and $(H_2)$, we shall derive compactness properties for sequences with uniformly bounded energy. In our setting the limiting configurations space turns out to be 
\begin{equation}\label{eq:V}
\mathscr{C}:= \left\{ (u,b) \in W^{1,\infty}(\Omega; \R^3) \times L^{\infty}(\Omega; \R^3) : (\nabla' u, b) \in BV\big(\Omega;\{A,B\}\big)\,,\, \partial_3 u =\partial_3 b = 0 \right\}\,,
\end{equation}
where we write $\nabla':=(\partial_1,\partial_2)$. Throughout the paper we identify pairs $(u,b) \in \mathscr{C}$ with functions defined on the mid-surface $\omega$, that is  
$u(x)=u(x')$, $b(x)=b(x')$ with $x':=(x_1,x_2)$. In particular, for any $(u,b) \in \mathscr{C}$, we can write 
\begin{equation}\label{eq:E}
 (\nabla' u,b)(x') = \bigl(1-\chi_E(x')\bigr)A + \chi_E(x') B \quad\text{for $\mathcal{L}^{2}$-a.e. $x' \in \omega$}\,,
\end{equation}
where $E\subset \omega$ is a set  of finite perimeter in $\omega$, and $\chi_E$ denotes its characteristic function. 
For $A'\neq B'$ the (reduced) boundary of $E$ consists of countably many planar interfaces with 
normal $\bar\nu$, while $E$ is an arbitrary set of finite perimeter in $\omega$ if $A'=B'$ (see Theorem \ref{thm:BJ}, and \cite{BJ}). 
Let us recall that, in our setting, $A$ and $B$ might {\it not} be rank-one connected, so that we shall have to construct recovery sequences substantially different from \cite{CFL}. 
We also emphasize that for $A'=B'$, the arbitrary geometry of the interface is again in sharp contrast with \cite{CFL}, where interfaces must be made by hyperplanes.

\vskip5pt

The general compactness result is formulated in Theorem~\ref{thm:compactness} below. As a matter of fact this theorem  
does not provide optimal compactness (only) in the case $\eps\ll h$. Indeed,  in this regime we may expect  a separation of scales to hold and the film to behave 
like a three dimensional sample. Thus, if the wells are not compatible in the bulk, {\it i.e.}, 
${\rm rank}(A-B)>1$, it is reasonable to believe that sequences with bounded energy should converge to trivial limits. 
This question will be addressed in the last section with {\it positive} results for some particular cases (see Theorems~\ref{rigidrank} and \ref{rigidrank2}). 

\begin{theorem}[Compactness]\label{thm:compactness}
Assume that $(H_1)-(H_2)$ hold. Let $h_n\to0^+$ and $\eps_n \to 0^+$ be arbitrary sequences, 
and let $\{u_{n}\} \subset H^{2}(\Omega;\R^3)$ be such that 
$\sup_{n} F_{\eps_n}^{h_n}(u_n) < \infty$. 
Then there exist a subsequence (not relabeled) and $(u,b) \in \mathscr{C}$ such that 
$ u_n - \med_{\Omega} u_n \, dx \to u$ in $W^{1,p}(\Omega;\R^3)$ and
 $\frac{1}{h_n} \partial_3 u_n \to~b$ in $L^p(\Omega;\R^3)$. 
\end{theorem}

To describe our  $\Gamma$-convergence results we need some additional assumptions on the potential $W$ 
(these assumptions will be used only in the construction of recovery sequences). First 
of all we assume, without loss of generality, that $A=-B $ and $\bar \nu = e'_1:=(1,0)$, so that 
\begin{equation}\label{orientationwells}
A'=-B'= a \otimes e'_1\ ,\quad A_2=B_2=0\ , \quad\text{and}\quad A_3=-B_3\,.
\end{equation}
Indeed, the general case can be reduced to \eqref{orientationwells} by considering a modified bulk energy density ${W}_{\rm mod}$ defined by 
$W_{\rm mod}(\xi):= W(\xi R +C)$ 
where $C=1/2(A+B)$ and $R={\rm diag}(R',1)$ with $R'\in SO(2)$ satisfying $R'\bar \nu= e'_1$.
This new potential $W_{\rm mod}$ obviously satisfies hypotheses $(H_1)$ and $(H_2)$ with \eqref{orientationwells}. 
Our second set of assumptions requires $W$ to share some structural properties of the prototypical function defined in \eqref{eq:potential}.
More precisely, given \eqref{orientationwells}, we assume that 
\begin{itemize}
\item[$(H_3)$] there exist constants $\varrho>0$ and $C_2>0$ such that
$$
\frac{1}{C_2}\, {\rm dist}\big(\xi,\{A,B\}\big)^p  
  \leq W(\xi) 
  \leq C_2\, {\rm dist}\big(\xi,\{A,B\}\bigr)^p
\; \text{ if } \; {\rm dist}\bigl(\xi,\{A,B\}\big) \leq \varrho\,;
$$
\item[$(H_4)$] $W(\xi_1,0,\xi_3) \leq W(\xi)$ for all $\xi=(\xi_1,\xi_2,\xi_3)\in \R^{3\times 3}$;
\vskip8pt

\item[$(H_5)$]  if $A'=B'=0$, then $W(\xi',\xi_3) = V\bigl( |\xi'|,\xi_3\bigr)$ for some  $V:[0,+\infty)\times \R^3\to[0,\infty)$. 
\end{itemize}
Here $|\cdot|$ stands for the usual Euclidean norm, $\xi':=(\xi_1,\xi_2)\in\mathbb{R}^{3\times 2}$ and $|\xi'|^2=|\xi_1|^2+|\xi_2|^2$.  
Observe that in the case $A'=B'$ (so that \eqref{orientationwells} yields $A'=B'=0$),  assumptions $(H_4)$ and $(H_5)$ require  
 the function $r\mapsto V(r ,z)$ to be nondecreasing for every $z\in\R^3$. 
Taking \eqref{orientationwells} into account, we also notice that these assumptions are clearly satisfied for 
$W(\xi)=\dist\big(\xi,\{A,B\}\big)^p$ with $
V(r,z) := \left( r^2 + \dist^2\big(z,\{A_3,B_3\}\big)\right)^{p/2}$.
Let us finally mention that similar assumptions are already present in~\cite{CFL}.
Condition $(H_3)$ is a standard non-degeneracy condition on $W$ near the wells, while $(H_4)$ allows one to construct lower dimensional optimal profiles connecting the two phases $A$ and $B$.
Hypothesis $(H_5)$ is a more technical isotropy condition, that we assume for simplicity.
\vskip8pt

Let us now consider the family of functionals $\mathcal{F}_\eps^h:[L^1(\Omega;\R^3)]^2\to [0,\infty]$ 
defined by 
$$\mathcal{F}_\eps^h(u,b):=\begin{cases}
F_{\eps}^{h}(u)  & \text{if $u\in H^2(\Omega;\R^3)$ and $b=\frac{1}{h}\partial_3u$}\,,\\
+\infty & \text{otherwise}\,.
\end{cases}$$
We will prove that the behavior of $\mathcal{F}_\eps^h$ depends, as expected, on the asymptotic ratio $ \frac{h}{\eps}\to\gamma\in[0,\infty]$ 
as $h$ and~$\eps$ tend to~$0$, and that   the family  
$\{\mathcal{F}_\eps^h\}$ $\Gamma$-converges to a  functional  $
\mathscr{F}_\gamma:[L^1(\Omega;\R^3)]^2\to [0,\infty]$ given~by 
\begin{equation}\label{defgamlim}
\mathscr{F}_\gamma(u,b):=
\begin{cases}
K_{\gamma}\, \Per_{\omega}(E) & \text{ if } (u,b) \in \mathscr{C}\,, \\
+\infty & \text{ otherwise}\,,
\end{cases}
\end{equation}
where $(\nabla'u,b)(x)=(1-\chi_E(x'))A+\chi_E(x')B$ as in \eqref{eq:E}, and $\Per_{\omega}(E):=\mathcal{H}^1(\partial^*E\cap\omega)$ denotes 
the (measure theoretic) perimeter of $E$ in $\omega$. 
Here the constant $K_\gamma>0$ is determined by an optimal profile problem for connecting phase $A$ to phase $B$.
By assumption $(H_4)$, we will be able to describe $K_\gamma$  through a lower dimensional variational problem. To simplify the notation, we introduce the 2D energy density $\mathcal{W}:\R^{3\times2}\to [0,\infty)$ given by 
$$
\mathcal{W}(\zeta_1,\zeta_2):=W(\zeta_1,0,\zeta_2)\,,
$$
which is a double-well potential with zero level set $\big\{(A_1,A_3),(B_1,B_3)\big\} 
$.
\vskip5pt

Our first convergence result deals with the critical regime where the thickness of the film and the strength  
of the interfacial energy are of the same order, that is $\gamma\in(0,\infty)$.

\begin{theorem}[Critical Regime]\label{thm:gammalim_gamma1} 
Assume that $(H_1) - (H_5)$ hold with \eqref{orientationwells}. Let $h_n\to0^+$ and $\eps_n\to0^+$ be arbitrary sequences such that 
$h_n/\eps_n\to \gamma$ for some $\gamma\in(0,\infty)$. Then the functionals $\{\mathcal{F}_{\eps_n}^{h_n}\}$ $\Gamma$-converge 
for the strong $L^1$-topology to the functional $\mathscr{F}_\gamma$ given by \eqref{defgamlim} with
\begin{multline}\label{formulacrit}
K_{\gamma} :=\inf \bigg\{  \frac{1}{\gamma}\int_{\ell I\times \gamma I}\mathcal{W}(\nabla v)+|\nabla^2 v|^2\,dy \,:\, \ell>0\,,\, v\in C^{2}(\ell  I\times \gamma  I;\R^3)\,,\\
\nabla  v(y)= (A_1,A_3)  \text{ nearby } \big\{y_1= \ell/2\big\} \text{ and } \nabla  v(y)= (B_1,B_3)  \text{ nearby } \big\{y_1= -\ell/2\big\} \bigg\}\,.
\end{multline}
\end{theorem}

We observe that the formula for $K_\gamma$ (with $\gamma\in(0,\infty)$) entails a highly nontrivial dependence on the vertical direction in the asymptotic problem. In fact, in the case $A_3=B_3$, one can find potentials $W$ for which a nontrivial dependance on $x_3$ still occurs, see \cite[Section 8]{CFL}.  
Note that in many second order phase transitions problems, optimal profiles usually have an oscillatory behavior along the limiting interface, see \cite{CFL,JK} and references therein 
(see also Theorem~\ref{thm:gammalim_gammasup} below).  
\vskip5pt

In contrast with the critical regime, one may expect the case $\gamma=0$ ({\it i.e.}, $h\ll \eps$) to lead to a simpler behavior with respect to the $x_3$-variable by separation of scales.  
Indeed, the energies formally behave like two dimensional ones,  and optimal transition layers should only  depend on the distance to the interface by assumptions $(H_4)-(H_5)$. 
We will illustrate this fact with more details in Section~\ref{Sectsub} (see Remark \ref{sepscales}).  Our results for this regime give a positive answer to our formal discussion, and they can be summarized in the following theorem.

\begin{theorem}[Subcritical Regime]\label{thm:gammalim_gammasub} 
Assume that  $(H_1)-(H_5)$ hold with \eqref{orientationwells}. Let $h_n\to0^+$ and $\eps_n\to0^+$ be arbitrary sequences such that 
$h_n/\eps_n\to 0$.  Then the functionals $\{\mathcal{F}_{\eps_n}^{h_n}\}$ $\Gamma$-converge 
for the strong $L^1$-topology to the functional $\mathscr{F}_0$ given by \eqref{defgamlim} with
\begin{multline}\label{formulasubcrit}
K_{0}:=  \inf \bigg\{ \int_{-\ell}^\ell \mathcal{W}\big(\phi(t)\big)+|\phi_1^\prime(t)|^2+2|\phi_2^\prime(t)|^2\,dt \, :\, \ell>0\,,\\
\phi=(\phi_1,\phi_2)\in C^{1}\big([-\ell,\ell\,];\R^{3\times2}\big)\,,\,
\phi(\ell)= (A_1,A_3) \text{ and } \phi(-\ell)=(B_1,B_3)\bigg\}\,.
\end{multline}
\end{theorem}
\vskip5pt

In the supercritical case $\gamma=+\infty$ ({\it i.e.}, $\eps\ll h$), one may again expect a separation of scales to hold. In other words, we should be able to recover the limiting functional by 
taking  first the limit $\eps\to 0$, and then  the thin film limit $h\to 0$. Hence, to obtain a nontrivial $\Gamma$-limit, it is natural to ask for $A$ and $B$ to be compatible in the bulk with a 
non vertical connection (see \eqref{orientationwellssc} below).  As already mentioned, we have exhibited rigidity effects  in the other cases, at least for some particular potentials (see Theorems~\ref{rigidrank} and \ref{rigidrank2}). 
For this reason we assume in the supercritical regime that $A-B$ is a rank-one matrix, and that $A'\not=B'$.
Under the structure \eqref{orientationwells}, this assumption is equivalent to the existence of  $\lambda\in\R$ such that $A_3=-B_3=\lambda a$. 
Then the wells $A$ and $B$ can be written as
\begin{equation}\label{orientationwellssc}
A=-B=a\otimes (e_1+\lambda e_3)\qquad \text{($a\not=0$)}\,.
\end{equation}

We have obtained partial results for this regime through lower and upper bounds for the $\Gamma$-$\liminf$ and $\Gamma$-$\limsup$, respectively. Fortunately, our estimates turn out to be nearly optimal in the sense that upper and lower bounds agree whenever $\lambda=0$, $p=2$, and $W$ is symmetric with respect to $\xi_3$ (which is the case for the potential \eqref{eq:potential} assuming \eqref{orientationwellssc}). In this latter case, it follows that the separation of scales is indeed true by \cite[Theorem~1.4]{CFL} (see Remark~\ref{sepscalessc}). 

\begin{theorem}[Supercritical Regime]\label{thm:gammalim_gammasup} 
Assume that  $(H_1)-(H_4)$ and \eqref{orientationwellssc} hold for some $\lambda\in\R$.   
Let $h_n\to0^+$ and $\eps_n\to0^+$ be arbitrary sequences such that 
$h_n/\eps_n\to \infty$.  Then, 
$$ c\, \mathscr{F}_\infty \leq  \Gamma\big(L^1\big)-\liminf_{n\to+\infty}\, \mathcal{F}_{\eps_n}^{h_n} 
\quad \text{and}\quad 
\Gamma\big(L^1\big)-\limsup_{n\to+\infty}\, \mathcal{F}_{\eps_n}^{h_n} \leq \mathscr{F}_\infty 
\,,$$
for a constant $c>0$, where the functional  $\mathscr{F}_\infty$ is given by \eqref{defgamlim} with
\begin{multline*}
K_{\infty}:= (1+\lambda^2)^{\frac{1}{2}} \inf \bigg\{\int_{Q'_{\lambda}} \ell\,\mathcal{W}(\nabla v)+\frac{1}{\ell}|\nabla^2 v|^2\,dy \,:\, \ell>0\,,\,
v\in C^{2}(\R^2;\R^3)\,, \\ 
\nabla  v(y)= (A_1,A_3)  \text{ nearby } \big\{y\cdot \nu_\lambda= \ell/2\big\} \text{, } \nabla  v(y)= (B_1,B_3)  \text{ nearby } \big\{y\cdot \nu_\lambda= -\ell/2\big\},\\
\text{and  $v$ is 1-periodic in the direction orthogonal to  $\nu_\lambda$}\bigg\}\,,
\end{multline*}
where $\nu_\lambda:=\frac{1}{\sqrt{1+\lambda^2}}(1,\lambda)\in\mathbb{S}^1$, and $Q^\prime_{\lambda}$ denotes the unit cube of $\R^2$ centered at the origin 
with two faces orthogonal to $\nu_\lambda$. 
Moreover, if $p=2$ in $(H_1)$, $\lambda=0$ in \eqref{orientationwellssc}, and $W$ satisfies $W(\xi',\xi_3)=W(\xi',-\xi_3)$ for all $\xi\in \mathbb{R}^{3\times 3}$, 
then the functionals $\{\mathcal{F}_{\eps_n}^{h_n}\}$ $\Gamma$-converge for the strong 
$L^1$-topology to $\mathscr{F}_\infty$. 
\end{theorem}
\vskip5pt

The paper is organized as follows. We start in Section 2 with a structure result for the class $\mathscr{C}$ of limiting maps, and the compactness 
theorem is proved in Section 3. The proofs of Theorems~\ref{thm:gammalim_gamma1}, \ref{thm:gammalim_gammasub}, and \ref{thm:gammalim_gammasup} are 
given in Sections~4, 5, and 6 respectively. We complete Section~6 with the aforementioned rigidity results in the supercritical case.

%
%

\section{Preliminaries}

Throughout the paper, $Q$ and $Q'$ denote the standard open unit cubes centered at the origin of $\R^3$ and~$\R^2$ respectively, 
while $I:=(-\frac{1}{2},\frac{1}{2})$. For simplicity, the differential operators $\frac{\partial}{\partial x_i}$ and  $\frac{\partial^2}{\partial x_i\partial x_j}$ are written  
$\partial_i$ and~$\partial^2_{ij}$ respectively. The rescaled gradient $\nabla_h$ and rescaled Hessian $\nabla_h^2$ operators are given~by 
$$\nabla_h u =\left(\partial_1u,\partial_2u,\frac{1}{h}\partial_3 u\right) \quad \text{and}\quad
\nabla_h^2 u=\left(\begin{array}{ccc}
\partial^2_{11} u\; &  \partial^2_{12} u\; & \frac{1}{h}\partial^2_{13} u\\[4pt]
\partial^2_{12} u\; & \partial^2_{22} u\; & \frac{1}{h}\partial^2_{23} u\\[4pt] 
\frac{1}{h}\partial^2_{13} u \;& \frac{1}{h}\partial^2_{23} u\;& \frac{1}{h^2}\partial^2_{33} u 
\end{array}\right)\,.$$
For a Borel set $B\subset \R^3$ and an admissible map $u$ we write 
$$F^h_\eps(u,B):= \int_{B} \frac{1}{\eps}\,W(\nabla_h u) + \eps |\nabla_h^2 u|^2 \, dx\,,$$
In the sequel, it will be useful to consider the two reference maps, $u_0$ and $b_0$, defined for $x\in\R^3$ by
\begin{equation}\label{defiu0b0}
u_0(x):=\bar u_0(x_1)\quad\text{and}\quad b_0(x):= \bar b_0(x_1)\,,
\end{equation}
where $\bar u_0$ and $\bar b_0$ are given by 
\begin{equation}\label{defu0}
\bar u_0(t) := |t|a
\quad \text{and}\quad
\bar b_0(t) := \begin{cases}
  A_3  &  \text{ if } t \geq 0\,, \\
  B_3  &  \text{ if } t < 0\,.
\end{cases}
\end{equation}
\vskip8pt 

We shall follow \cite{AFP} for the standard results and notations on functions of bounded variation. We only recall that, given an open set $\mathscr{O}\subset \R^N$, a 
Borel set $E\subset \mathscr{O}$ is said to be of finite perimeter in $\mathscr{O}$ if its characteristic function $\chi_E$ belongs to $BV(\mathscr{O})$. In such a case, 
the perimeter of $E$ in $\mathscr{O}$, that we write $\Per_{\mathscr{O}}(E)$, is the total variation~$|D\chi_E|(\mathscr{O})$, and it is equal 
to $\mathcal{H}^{N-1}(\partial^*E\cap \mathscr{O})$ where $\partial^*E$ denotes the reduced boundary of $E$, and $\mathcal{H}^{N-1}$ is the $(N-1)$-dimensional Hausdorff measure. 
\vskip5pt

We now state a structure result for the class $\mathscr{C}$ of limiting configurations (see \eqref{eq:V}). To this purpose, let us~define
\begin{equation}\label{defalpha}
\alpha_{\rm min}:=\inf\big\{x_1\in\R : x=(x_1,x_2)\in\omega\big\}\quad\text{and}\quad\alpha_{\rm max}:=\sup\big\{x_1\in\R : x=(x_1,x_2)\in\omega\big\}\,.
\end{equation}
We have the following theorem as a consequence of \cite{BJ,CFL}. 

\begin{theorem}\label{thm:BJ}
Assume that \eqref{orientationwells} holds. Then for every pair $(u,b) \in \mathscr{C}$, $(\nabla'u,b)$ is of the form 
$$\big(\nabla'u(x),b(x)\big)=(1-\chi_E(x'))A+\chi_E(x')B \,,$$
where  $E \subset \omega $ is a set of finite perimeter in $\omega$. Moreover, 
if $A'\not= B'$ then  $u$ is of the form 
\begin{equation}\label{fctionofx1}
u(x)= c_0 + x_1a - 2\psi(x_1)a \,,
\end{equation}
where $c_0 \in \R^3$, $c_0 \cdot a = 0$, $\psi\in W^{1,\infty}((\alpha_{\rm min},\alpha_{\rm max});\R)$ and $\psi'\in BV_{\rm loc}((\alpha_{\rm min},\alpha_{\rm max});\{0,1\})$. 
In particular, if $A'\not= B'$ then 
 $E$ is layered perpendicularly to $e'_1$, i.e.,
\begin{equation}\label{structredbdAdiffB}
\partial^{\ast} E \cap \omega = \bigcup_{i\in\mathscr{I}} \{\alpha_i\} \times J_i\,, 
\end{equation}
where $\mathscr{I}\subset\ZZ$ is made by successive integers, $\{\alpha_i\}\subset (\alpha_{\rm min},\alpha_{\rm max})$ is  
locally finite in $(\alpha_{\rm min},\alpha_{\rm max})$, $\alpha_i<\alpha_{i+1}$, and the sets $J_i :=\{t\in\R:(\alpha_i,t)\in\omega\}$ are open bounded intervals. 
\end{theorem}

\begin{proof} {\it Step 1.} We start with the case where $A'\not =B'$. Given $(u,b) \in \mathscr{C}$, we can write 
\begin{equation}\label{presetF}
(\nabla'u,b)=(1-\chi_F) A +\chi_F B\,,
\end{equation}
for some set $F\subset \Omega$ of finite perimeter in $\Omega$. Since $\partial_3 u=0$, we have $\nabla u \in BV(\Omega;\{(A',0),(B',0)\})$. Then we observe that \eqref{orientationwells} 
yields $(A',0)-(B',0)=a\otimes e_1$. Thanks to the convexity of $\omega$,  we can apply \cite[Theorem 3.3]{CFL} to deduce that $u$ is of the form \eqref{fctionofx1}. From \eqref{fctionofx1}, 
\eqref{presetF}, and the convexity of~$\omega$, it readily follows that $\chi_{F}=\chi_{E\times I}$ $\mathcal{L}^3$-a.e. in $\Omega$ for some set $E\subset \omega$ of finite perimeter in $\omega$ satisfying~\eqref{structredbdAdiffB}. 
\vskip5pt

\noindent{\it Step 2.} We now consider the case $A'=B'$.  Given $(u,b) \in \mathscr{C}$, we have $b\in BV(\Omega;\{A_3,B_3\})$, so that $b=(1- \chi_F)A _3+\chi_F B_3$ 
for some set $F\subset \Omega$ of finite perimeter in $\Omega$. By standard slicing results (see\cite[Section 3.11]{AFP}), $b^{x'}:= b(x',\cdot)$ belongs to $BV(I;\R^3)$ for $\mathcal{L}^{2}$-a.e. $x'\in\omega$, 
and $\mathcal{L}^2 \restr{}\omega \otimes D b^{x'} = \partial_3 b$. Since $\partial_3 b=0$, we deduce that $D b^{x'} =0$ for $\mathcal{L}^2$-a.e. $x'\in\omega$. On the other hand, 
we can find a representative $b^*$ of $b$ such that  $(b^*)^{x'}:=b^*(x',\cdot)$ is a good representative of  $b^{x'}$ for $\mathcal{L}^{2}$-a.e. $x'\in\omega$. Since $D b^{x'} =0$, we conclude 
that $(b^*)^{x'}$ is constant for $\mathcal{L}^2$-a.e. $x'\in\omega$, that is $b^*(x)=b^*(x')$. Then it  follows that $\chi_{F}=\chi_{E\times I}$  $\mathcal{L}^3$-a.e. in $\Omega$ for some set $E\subset \omega$ of finite perimeter in $\omega$. 

\end{proof}

%
%

\section{Compactness}

This section is devoted to the proof of Theorem \ref{thm:compactness}, and we assume that $(H_1)$ and $(H_2)$ hold. 
We consider arbitrary sequences  $h_n\to0^+$, $\eps_n\to0^+$ as $n\to\infty$,  and $\{u_n\}_{n\in\NN}\subset H^{2}(\Omega;\R^3)$ such that 
 $\sup_n F_{\eps_n}^{h_n}(u_n) < \infty$. Throughout this section we write $b_n:=\frac{1}{h_n}\partial_3 u_n$. 
\vskip5pt 

\noindent{\bf Proof of Theorem  \ref{thm:compactness}.} {\it Step 1.} We claim that there exist a subsequence $\{\eps_n\}$ (not relabeled), 
a pair $(u,b)\in W^{1,\infty}(\Omega;\R^3)\times L^\infty(\Omega;\R^3)$ satisfying $\partial_3 u=0$, and 
$\theta\in L^\infty(\Omega;[0,1])$ such that 
\begin{equation}\label{weakconv}
\text{$u_n-\,\med_{\Omega}u_n\,\rightharpoonup u$ weakly in $W^{1,p}(\Omega;\R^3)\,$},\;\text{$b_n\rightharpoonup b$ weakly in $L^p(\Omega;\R^3)$ as $n\to\infty$}\,, 
\end{equation}
and 
\begin{equation}\label{eq:gradu_YM}
(\nabla' u, b)(x) = \bigl( 1- \theta(x)\bigr) A + \theta(x) B\quad\text{for $\mathcal{L}^3$-a.e.  $x\in\Omega$}\,.
\end{equation}
Indeed, we first deduce from the growth assumption $(H_2)$ that
$$
\int_{\Omega} \big(|\nabla' u_n|^p + |b_n|^p \big) \, dx \leq C \bigg(\int_{\Omega} W(\nabla' u_n, b_n) \, dx + 1\bigg)\leq C\big(\eps_n F_{\eps_n}^{h_n}(u_n) +1\big)\leq C\,.
$$
Therefore $\{b_n\}$ is uniformly bounded in $L^p(\Omega;\R^3)$, and $\{u_n-\,\med_{\Omega}u_n\}$ is uniformly bounded in $W^{1,p}(\Omega;\R^3)$, thanks to the 
Poincar\'e-Wirtinger Inequality. Hence we may extract a subsequence such that \eqref{weakconv} holds for some pair $(u,b)\in W^{1,p}(\Omega;\R^3)\times L^p(\Omega;\R^3)$.  
Since  $\|\partial_3 u_n\|_{L^p(\Omega)}\leq Ch_n$, we  deduce that $\partial_3 u \equiv 0$.

Next we observe that the sequence $\{(\nabla' u_{n}, b_{n})\}$ also  generates a Young measure $\{ \nu_x \}_{x \in \Omega}$.  
From the fundamental theorem on Young measures (see {\it e.g.} \cite[Theorem 6.11]{Pedregal}), we derive 
$$
\int_{\Omega} \int_{\R^{3\times 3}} W(\xi) \, d\nu_x(\xi) \, dx \leq  \lim_{n\to+\infty} \int_{\Omega} W\big( \nabla' u_{n} , b_{n} \big) \, dx =0 \,,
$$
so that ${\rm supp}\, \nu_x \subset \{A, B\}$ for $\mathcal{L}^3$-a.e. $x \in \Omega$. 
Hence there exists $\theta \in L^1\bigl(\Omega;[0,1]\bigr)$ such that
$$
\nu_x = \bigl( 1-\theta(x)\bigr) \delta_{\xi=A} + \theta(x) \delta_{\xi=B}
\quad \text{ for $\mathcal{L}^3$-a.e. } x \in \Omega\,.
$$
Multiplying this last equality by $\xi$ and integrating with respect to $\xi$ yields \eqref{eq:gradu_YM}, 
which completes the proof of the claim. 
\vskip5pt

\noindent{\it Step 2.}   We claim that 
$(u,b) \in \mathscr{C}$. We shall distinguish two distinct cases.
\vskip5pt

\noindent\emph{Case a).}  We first assume that $A' \neq B'$.  For $M>0$ and $\xi' \in \R^{3\times 2}$, we define
\begin{multline*}
\varphi(\xi') := \inf \bigg\{ \int_0^1 \min\Big(\sqrt{W_0\bigl(g(s)\bigr)},M \Big) | g'(s)| \, ds \,:\,
 g\in W^{1,\infty}([0,1];\R^{3\times 2})\,,\\
  g(0) =A' \text{ and } g(1) =  \xi'  \bigg\}\,,
\end{multline*}
where $\ds W_0(\xi') := \min\{W(\xi' , z) : z \in \R^3\}$
is a continuous function of $\xi'$. One may easily check that 
$\varphi$ is Lipchitz continuous, $\varphi(\xi')=0$ if and only if $\xi' = A'$, and that 
\begin{equation}\label{eq:phi}
|\nabla \varphi(\xi')|  \leq \min\big\{\sqrt{W_0(\xi')}, M\big\} 
  \quad \text{for $\mathcal{L}^{3\times 2}$-a.e. $\xi' \in \R^{3 \times 2}\,$.}
\end{equation}
We claim that $\big\{ \varphi(\nabla'u_n)\big\}$ is uniformly bounded in $W^{1,1}(\Omega;\R)$. Indeed, estimate first 
$$
\int_{\Omega} \big|\nabla \big(\varphi(\nabla' u_n)\big)\big| \, dx
\leq \int_{\Omega} \sqrt{W_0\bigl(\nabla' u_n(x)\bigr)}\, \big| \nabla (\nabla'u_n) \big| \, dx \leq \frac{1}{2}\, F_{\eps_n}^{h_n}(u_n) \leq C\,,
$$
and by \eqref{eq:phi},
$$\int_{\Omega} \left| \varphi\bigl(\nabla' u_n(x)\bigr) \right| \, dx
  \leq M \int_{\Omega}  \left| \nabla'u_n(x) \right| \, dx + \varphi(0) \mathcal{L}^3(\Omega)\,.
$$
Hence, up to a further subsequence (not relabeled), 
\begin{equation}\label{eq:convH}
\varphi(\nabla'u_n) \to H\quad 
\text{ in $L^1(\Omega)$ as $n\to\infty$}\,, 
\end{equation}
for some $H\in BV(\Omega)$. On the other hand, the Young measure $\{\mu_x\}_{x\in\Omega}$ generated by $\big\{\varphi(\nabla' u_n) \big\}$ is given~by
$$
\mu_x = \bigl( 1-\theta(x)\bigr) \delta_{t=\varphi(A')} + \theta(x) \delta_{t=\varphi(B')}\,.
$$
Then the strong convergence in \eqref{eq:convH} yields $\mu_x = \delta_{t=H(x)}$, so that 
$$
\delta_{t=H(x)} = \bigl( 1-\theta(x)\bigr) \delta_{t=\varphi(A')} + \theta(x) \delta_{t=\varphi(B')}\,.
$$
As a consequence $\theta(x) \in \{0,1\}$ for $\mathcal{L}^3$-a.e. $x \in \Omega$, and 
$$
H(x) = \bigl( 1- \theta(x) \bigr) \varphi(A') + \theta(x) \varphi(B') =  \theta(x) \varphi(B')\quad\text{for $\mathcal{L}^3$-a.e. $x\in\Omega\,$}\,.
$$
Since $\varphi(B')\not=0$, it yields  $\theta \in BV\bigl(\Omega; \{0,1\}\bigr)$, and we may now write 
$\theta= \chi_F$ for some set  $F\subset \Omega\,$ of finite perimeter in $\Omega$. 
In view of \eqref{eq:gradu_YM}, we obtain $(\nabla'u,b) = (1-\chi_F) A + \chi_F B$  $\mathcal{L}^3$-a.e. in $\Omega$, and thus $(\nabla' u, b)$ belongs to  $BV\bigl(\Omega;\{A,B\}\bigr)$. 
Since $\partial_3 u=0$, we have $\nabla u\in BV\bigl(\Omega;\{(A',0),(B',0)\}\bigr)$ and it follows from \cite[Theorem 3.3]{CFL} that  $F= E \times I$ for some set $E \subset \omega$ 
of finite perimeter in $\omega$, which in turn implies that $\partial_3 b=0$, and thus~$(u,b)\in  \mathscr{C}$. 
\vskip5pt

\noindent\emph{Case b).}  Let us now assume that  $A' = B'$ and $A_3 \neq B_3$. For $M>0$ and $z\in\R^3$, we define
\begin{multline*}
\psi(z):= \inf \bigg\{ \int_0^1 \min\Bigl(\sqrt{W_1\big(g(s)\big)},M \Bigr) | g'(s) | \, ds\,:\, g\in W^{1,\infty}([0,1];\R^{3})\,,\,
 g(0) = A_3 \text{ and } g(1) = z \bigg\}\,,
\end{multline*}
where $ W_1(z) := \min \{W(\xi' , z): \xi' \in \R^{3\times 2}\}$  
is a continuous function of $z$. 
As previously $\psi$ is Lipschitz continuous, and $\psi(z)=0$ if and only if $z = A_3$. 
Arguing as in Case a),  we obtain that $\{\psi(b_n)\}$ is uniformly bounded in $W^{1,1}(\Omega;\R)$, and 
$$\frac{1}{h_n}\int_{\Omega} \big|\partial_3 \big(\psi(b_n)\big) \big| \, dx
 \leq \frac{1}{h_n}\int_{\Omega} \sqrt{W_1(b_n)}\, | \partial_3 b_n | \, dx 
   \leq \frac{1}{2}\, F^{h_n}_{\eps_n}(u_n) \leq C\,. $$
Therefore, up to a subsequence,  
\begin{equation}\label{eq:convH_case3}
\psi(b_n) \to G \quad\text{in  $L^1(\Omega)$ as $n\to\infty\,$,} 
\end{equation}
for some $G \in BV(\Omega)$ satisfying $\partial_3 G = 0$. 
Arguing again as in Case a), we deduce from 
\eqref{eq:convH_case3} that 
$$
\delta_{t=G(x)} = ( 1-\theta(x)) \delta_{t=\psi(A_3)} + \theta(x) \delta_{t=\psi(B_3)}\,,
$$
which yields $\theta \in BV(\Omega; \{0,1\})$. 
Since $\partial_3 G=0$  we can argue as in  the proof of Theorem~\ref{thm:BJ}, Step~2,  
to deduce that $\theta(x) = \chi_E(x')$ for some set $E \subset \omega$ of finite perimeter in $\omega$.  
In view of \eqref{eq:gradu_YM}, we conclude that $(\nabla' u, b)\in BV\bigl(\Omega;\{A,B\}\bigr)$ 
and $\partial_3 b=0$, and thus $(u,b)\in  \mathscr{C}$. 
\vskip5pt

\noindent{\it Step  3.} In view of the previous steps, we know that  $ (\nabla u_n,b_n) \rightharpoonup (\nabla u,b)$ weakly in $L^p(\Omega)$, and that  
$ \{ (\nabla' u_n, b_n)\}_{n\in\NN}$ generates the Young measure $\{\nu_x\}_{x\in\Omega}$ given by 
$$\nu_x(\xi) = \big(1-\chi_E(x')\big) \delta_{\xi=A} + \chi_E(x') \delta_{\xi=B} = \delta_{\xi = (\nabla' u, b)}\,.$$
By standard results on Young measures (see {\it e.g.} \cite[Proposition 6.12]{Pedregal}), it follows that 
$(\nabla' u_n, b_n) \to (\nabla' u, b)$ strongly in $L^p(\Omega;\R^{3\times 3})$, and the proof Theorem \ref{thm:compactness} is complete.\prbox


%
%


\section{$\Gamma$-convergence in the critical regime}

This section is devoted to the proof of Theorem \ref{thm:gammalim_gamma1}. The $\Gamma$-liminf and $\Gamma$-limsup   
inequalities  are derived in Theorem \ref{thm:gammaliminf} and Theorem \ref{thm:RS_A'neqB'} respectively, while Corollary \ref{prop:charK_per} shows that 
the lower and upper inequalities actually coincide. In proving lower inequalities, we partially adopt the approach of \cite{CFL}, once 
adapted to the dimension reduction setting. Throughout this section  the parameter $\gamma\in(0,\infty)$ is~given.

%
%

\subsection[]{The $\Gamma$-$\liminf$ inequality}\label{gamliminfcrit}

We  introduce the constant  
\begin{multline}\label{eq:defKstar1}
K_\gamma^\star:=\inf \biggl\{\liminf_{n\to\infty}\, F^{h_n}_{\eps_n}(u_n,Q)\,:\, h_n\to0^+ \text{ and } \eps_n \to 0^+  \text{ with } h_n/\eps_n\to \gamma\,,\\ 
\{u_n\} \subset H^{2}(Q;\R^3),\,
\text{$(u_n,\frac{1}{h_n}\partial_3 u_n) \to  (u_0,  b_0)$  in  $[L^1(Q;\R^3)]^2$} \biggr\}\,,
\end{multline}
where  the functions $u_0$ and $b_0$ are given by  \eqref{defiu0b0}. The constant $K^\star_\gamma$ turns out to be finite, as one may 
check by considering an admissible sequence $\{u_n\}$ made of suitable regularizations of $u_0$ and $b_0$ (see also the proof of Theorem~\ref{thm:RS_A'neqB'}).  
In this subsection we shall prove that under assumptions $(H_1)-(H_2)$ and $(H_5)$, the lower $\Gamma$-limit evaluated at any $(u,b)\in\mathscr{C}$ is bounded 
from below by $K^\star_\gamma$ times the length of the  jump set of $(\nabla'u,b)$.  
We first prove this statement in the case of an elementary  jump~set.

\begin{proposition}\label{gliminfref}
Assume that  $(H_1)$, $(H_2)$ and \eqref{orientationwells} hold. Let $h_n\to0^+$ and $\eps_n\to0^+$ be arbitrary sequences such that $h_n/\eps_n \to \gamma$.  
Let $\rho>0$ and $\alpha \in\R$, let $J\subset\R$ be a bounded open interval, and consider the cylinder 
$U:=(\alpha-\rho,\alpha+\rho)\times J\times I$.  
Let $(u,b)\in W^{1,\infty}(U;\R^3)\times L^\infty(U;\R^3)$ satisfying $\partial_3 u=\partial_3 b=0$, and  
\begin{equation}\label{eq:u01} 
 (\nabla'u,b) = \chi_{\{x_1<\alpha\}}\,B + \chi_{\{x_1\geq\alpha\}}\,A \quad\text{or}\quad 
 (\nabla'u,b) = \chi_{\{x_1\leq \alpha\}}\,A + \chi_{\{x_1>\alpha\}}\,B\,.
\end{equation}
Then for any sequence $\{u_n\}\subset H^2(U;\R^3)$ satisfying $(u_n,\frac{1}{h_n}\partial_3 u_n)\to (u,b)$ 
in $[L^1(U;\R^3)]^2$, we have 
$$\liminf_{n\to\infty}\,F^{h_n}_{\eps_n}(u_n,U)\geq K^\star_\gamma\mathcal{H}^1(J)\,.$$
\end{proposition}
\vskip5pt

The proof of Proposition \ref{gliminfref}  relies on  scaling properties and the translation invariance 
of the energy functional $F^{h}_{\eps}$. To determine the corresponding properties in the limit 
 we introduce the following set function. For  an open bounded set $J\subset \R$ and $\rho>0$, we write $J_{\rho} := \rho I \times J \times I$, and we define 
\begin{multline*}
\mathcal{E}_\gamma(J,\rho):=\inf \biggl\{\liminf_{n\to\infty}\, F^{h_n}_{\eps_n}(u_n,J_\rho)\,:\, h_n\to0^+ \text{ and } \eps_n \to 0^+  \text{ with } h_n/\eps_n\to \gamma,\\ 
\{u_n\} \subset H^{2}(J_\rho;\R^3),\,
\text{$(u_n,\frac{1}{h_n}\partial_3 u_n) \to  (u_0,b_0)$   in  $[L^1(J_\rho;\R^3)]^2$} \biggr\}\,.
\end{multline*}
Noticing that $K_\gamma^\star=\mathcal{E}_\gamma(I,1)$ we now state the following lemma.

\begin{lemma}\label{lem:scales}
Assume that $(H_1)$, $(H_2)$ and \eqref{orientationwells} hold.
Then
\begin{itemize}
\item[(i)] $\mathcal{E}_\gamma(t + J,\rho) = \mathcal{E}_\gamma(J,\rho)$ for all $t \in \R\,$;
\vskip5pt
\item[(ii)] if $J_1 \subset J_2$ and $\rho_1\leq \rho_2$, then $\mathcal{E}_\gamma(J_1,\rho_1) \leq \mathcal{E}_\gamma(J_2,\rho_2)\,$;
\vskip5pt
\item[(iii)] if $J_1 \cap J_2 = \emptyset$, then $\mathcal{E}_\gamma(J_1 \cup J_2,\rho) \geq \mathcal{E}_\gamma(J_1,\rho) + \mathcal{E}_\gamma(J_2,\rho)\,$;
\vskip5pt
\item[(iv)] if $\alpha>0$, then $\mathcal{E}_\gamma(\alpha J, \alpha \rho) = \alpha \mathcal{E}_\gamma(J,\rho)\,$; 
\vskip5pt

\item[(v)] if $0<\alpha < 1$, then  $\mathcal{E}_\gamma(\alpha J, \rho) \geq \alpha \mathcal{E}_\gamma(J,\rho)\,$;
\vskip5pt
\item[(vi)] if $J$ is an interval then $\mathcal{E}_\gamma(J, \rho) = \mathcal{H}^{1}(J)\, \mathcal{E}_\gamma(I,\rho)\,$;
\vskip5pt
\item[(vii)] if $J$ is an interval then $\mathcal{E}_\gamma(J,\rho) = \mathcal{E}_\gamma(J, \delta)$ for all $\delta >0\,$.
\end{itemize}
\end{lemma}

\begin{proof} We first observe that {\it (i)}  follows from the translation invariance of the functional $F^h_{\eps}$. Then {\it (ii)} is due to  
the fact that any admissible sequence for $\mathcal{E}_\gamma(J_2,\rho_2)$ yields an 
admissible sequence for $\mathcal{E}_\gamma(J_1,\rho_1)$ whenever $J_1 \subset J_2$ and $\rho_1\leq \rho_2$. 
The proof  of claim {\it (iii)} 
follows  a similar argument.  
\vskip2pt

\noindent{\it Proof of (iv).} Let $h_n\to 0^+$ and $\eps_n\to 0^+$ be such that $h_n/\eps_n\to\gamma$, and let $\{u_n\}$ be  an 
admissible sequence for $\mathcal{E}_\gamma(\alpha J, \alpha \rho)$. 
We define for $x\in J_\rho$, $v_n(x) := \frac{1}{\alpha} u_n(\alpha x',x_3)$, $\tilde h_n:=h_n/\alpha$ and $\tilde \eps_n:=\eps_n/\alpha\,$. 
By homogeneity of $u_0$ and $b_0$, we infer that $v_n(x) \to \frac{1}{\alpha}  u_0(\alpha x',x_3) =  u_0(x)$ and $\frac{1}{\tilde h_n} \partial_3 v_n \to  b_0$ 
in  $L^1(J_\rho;\R^3)$ as $n\to\infty$. 
In particular, $\{v_n\}$ with $\{(\tilde h_n,\tilde \eps_n)\}$ is admissible for $\mathcal{E}_\gamma(J,\rho)$. 
Changing variables then yields
$F^{\tilde h_n}_{\tilde \eps_n}(v_n,J_{\rho})= \frac{1}{\alpha} F^{h_n}_{\eps_n}\big(u_n, (\alpha J)_{\alpha \rho}\big)$. 
Letting $n\to\infty$ we deduce that
$\liminf_{n} \,F^{h_n}_{\eps_n}\big(u_n,(\alpha J)_{\alpha \rho}\big) \geq \alpha\mathcal{E}_\gamma(J,\rho)$. 
In view of the arbitrariness of $\{(h_n,\eps_n)\}$ and $\{u_n\}$, we conclude that $\mathcal{E}_\gamma(\alpha J, \alpha \rho)\geq \alpha\mathcal{E}_\gamma(J,\rho)$. 
The reverse inequality follows from the arbitrariness of $\alpha>0$.  
\vskip2pt

\noindent{\it Proof of (v).} If $0< \alpha< 1$ 
 then $(\alpha J)_{\alpha \rho} \subset (\alpha J)_{\rho}$, and we derive from {\it (ii)} and {\it (iv)} that $
\mathcal{E}_\gamma(\alpha J, \rho) \geq \mathcal{E}_\gamma(\alpha J, \alpha \rho) = \alpha \mathcal{E}_\gamma(J,\rho)$. 
\vskip2pt

\noindent{\it Proof of (vi).}  Write $J = Z\cup(\bigcup_{k =1}^N J_k)$ 
for some finite set $Z$ and some family $\{J_k\}_{k=1}^N$ of mutually disjoint open intervals of the form 
$J_k = a_k + \alpha_k I$ with $0<\alpha_k<1$. In particular, 
$\mathcal{H}^{1}(J) = \sum_{k=1}^N \alpha_k$. We infer from {\it (i)}, {\it (iii)} and {\it (v)} that
$\mathcal{E}_\gamma(J,\rho)  \geq (\sum_{k}\alpha_k ) \mathcal{E}_\gamma(I,\rho)$, 
leading to $\mathcal{E}_\gamma(J,\rho) 
  \geq \mathcal{H}^{1}(J) \mathcal{E}_\gamma(I,\rho)$. The reverse inequality can be  obtained in the same way inverting the roles of $I$ and $J$. 
\vskip2pt

\noindent{\it Proof of (vii).}  Claim {\it (vii)} is a straightforward consequence of {\it (iv)} and {\it (vi)}. 
\end{proof}

\begin{remark}\label{compegam}
As a consequence of {\it (vi)} and {\it (vii)} in Lemma~\ref{lem:scales}, we have 
$\mathcal{E}_\gamma(J,\rho)=K_\gamma^\star   \mathcal{H}^{1}(J)$ 
for every $\rho>0$ and every  bounded open interval $J\subset \R$. 
\end{remark}

An important consequence of Lemma \ref{lem:scales} is that the energy 
of optimal sequences for $\mathcal{E}_\gamma(J,\rho)$ is 
concentrated near the limiting interface. We shall make use of Corollary \ref{cor:concentration} 
in the next subsection in order to compare the constants $K_\gamma^\star$ and $K_\gamma$. 

\begin{corollary}\label{cor:concentration}
Assume that $(H_1)$, $(H_2)$ and \eqref{orientationwells} hold. 
Let $0<\delta<\rho$ and let $J\subset\R$ be a bounded open interval.  For any sequences $h_n\to0^+$, $\eps_n\to 0^+$ and $\{u_n\}\subset H^{2}(J_\rho;\R^3)$ such that 
$h_n/\eps_n\to \gamma$, $(u_n,\frac{1}{h_n}\partial_3 u_n)\to  (u_0,  b_0)$ 
in $[L^1(J_\rho;\R^3)]^2$, and $ \lim_{n} F^{h_n}_{\eps_n}(u_n,J_\rho)=K_\gamma^\star   \mathcal{H}^{1}(J)$, we have 
$$\lim_{n\to\infty} F^{h_n}_{\eps_n}(u_n,J_\rho\setminus J_{\delta}) = 0\,.$$
\end{corollary}

\begin{proof}
Since $J_{\delta} \subset J_\rho$, we deduce from Remark~\ref{compegam}  that
$\limsup_{n}\, F^{h_n}_{\eps_n}(u_n,J_{\delta}) 
  \leq   \mathcal{E}_\gamma(J,\rho)<\infty$.
On the other hand the sequence $\{u_n\}$ is admissible for $\mathcal{E}_\gamma(J,\delta)$. In view of Lemma \ref{lem:scales}, we infer that 
$$
\liminf_{n\to\infty} F^{h_n}_{\eps_n}(u_n,J_{\delta}) 
  \geq \mathcal{E}_\gamma(J,\delta) 
  = \mathcal{E}_\gamma(J,\rho)   = \lim_{n\to\infty} F^{h_n}_{\eps_n}(u_n,J_\rho)\,.
$$
Therefore $\lim_{n} F^{h_n}_{\eps_n}(u_n,J_{\delta}) = \lim_{n} F^{h_n}_{\eps_n}(u_n,J_\rho)$ 
which clearly implies the announced result. 
\end{proof}

\noindent{\bf Proof of Proposition \ref{gliminfref}.}
By the translation invariance of $F^{h_n}_{\eps_n}$, we have
$
\liminf_{n} F^{h_n}_{\eps_n}(u_n,U) 
  = \liminf_{n} F^{h_n}_{\eps_n}(\tau u_n,J_\rho)$, 
where $\tau u_n(x) := u_n(x_1+\alpha,x_2,x_3)$. Obviously $(\tau u_n,\frac{1}{h_n}\partial_3\tau u_n) \to (\tau u, \tau b)$ in $[L^1(J_\rho;\R^3)]^2$ 
with $\tau u(x): = u(x_1+\alpha,x_2,x_3)$ and $\tau b(x) := b(x_1+\alpha,x_2,x_3)$. 
If the first case in \eqref{eq:u01} holds, then $(\tau u, \tau b) = (u_0+c_0, b_0)$ for some constant $c_0 \in \R^3$. Subtracting the constant $c_0$, 
we derive from the  definition of $\mathcal{E}_\gamma$ and Remark~\ref{compegam} that
$$\liminf_{n\to \infty} F^{h_n}_{\eps_n}(u_n,U)=\liminf_{n\to \infty} F^{h_n}_{\eps_n}(\tau u_n-c_0,J_\rho) \geq  \mathcal{E}_\gamma(J,\rho)
  = K^{\star}_\gamma \,\mathcal{H}^1(J) \,.
$$
If the alternative case in \eqref{eq:u01} holds, then $(-\tau u, \tau b)(-x) = (u_0+c_0, b_0)(x)$ for some constant $c_0 \in \R^3$.
Observe that $F^{h_n}_{\eps_n}(\tau u_n,J_\rho) = F^{h_n}_{\eps_n}(v_n,J_\rho)$ with $v_n(x) = -\tau u_n(-x)-c_0$. Then $(v_n,\frac{1}{h_n}\partial_3v_n) \to (u_0, b_0)$ 
in $[L^1(J_\rho;\R^3)]^2$. Hence, 
$$\liminf_{n\to \infty} F^{h_n}_{\eps_n}(u_n,J_\rho)
  = \liminf_{n\to \infty} F^{h_n}_{\eps_n}(v_n,J_\rho)\geq  \mathcal{E}_\gamma(J,\rho)
  = K^{\star}_\gamma \,\mathcal{H}^1(J) \,,
$$
and the proof is complete. \prbox

\begin{remark}\label{validKinfty}
Setting $K_\infty^\star$ to be the constant defined by \eqref{eq:defKstar1} with $\gamma=+\infty$ (see \eqref{defKstarinf}), and assuming that $K_\infty^\star<\infty$, one may check that 
Proposition~\ref{gliminfref} and Corollary \ref{cor:concentration} still hold in the case $\gamma=+\infty$ with the same proofs. 
\end{remark}

We are now ready to prove the main result of this subsection which extends Proposition~\ref{gliminfref} to the general case. The proof for $A'\not=B'$ will be a 
direct consequence of Proposition~\ref{gliminfref}, while the case $A'=B'$ will require an additional analysis based on a blow-up argument.

\begin{theorem}\label{thm:gammaliminf}
Assume that  $(H_1)-(H_2)$,  $(H_5)$ and  \eqref{orientationwells} hold.   
Let $h_n\to0^+$ and $\eps_n\to0^+$ be arbitrary sequences such that $h_n/\eps_n \to \gamma$.  
Then, for any $(u,b)\in \mathscr{C}$ and any sequences $\{u_n\}\subset H^2(\Omega;\R^3)$ such that $(u_n,\frac{1}{h_n}\partial_3 u_n)\to (u, b)$ in $[L^1(\Omega;\R^3)]^2$, we have  
$$
\liminf_{n\to\infty}\, F^{h_n}_{\eps_n}(u_n) \geq K_\gamma^{\star}\, {\rm Per}_{\omega}(E)\,,
$$
where $(\nabla' u,b) (x)= \bigl(1-\chi_E(x')\bigr)A + \chi_E(x') B\,$.
\end{theorem}

\noindent{\bf Proof.} {\it Step 1.} We first assume that $A^\prime\not= B^\prime$. 
Assuming that $E$ is non trivial, by Theorem \ref{thm:BJ} we can write $\partial^{\ast}E \cap\omega$ as in \eqref{structredbdAdiffB}. 
Then we have  
$\mathcal{H}^{1}(\partial^{\ast} E\cap \omega )=\sum_{i\in\mathscr{I}} \mathcal{H}^{1}( J_i) < \infty$. 
Consider an arbitrarily small $\delta>0$ and choose $k^- = k^-(\delta) \in \mathscr{I}$ and  $k^+ = k^+(\delta) \in \mathscr{I}$ such that $k^-\leq k^+$, and 
$\mathcal{H}^{1}(\partial^{\ast} E\cap\omega ) 
  \leq \sum_{i=k^-}^{k^+} \mathcal{H}^{1}(J_i) + \delta$. 
For each  $i=k^-,\ldots,k^+$, let $J_i' \subset\!\subset J_i$ be an open interval satisfying
$\mathcal{H}^{1}(J_i) 
  \leq \mathcal{H}^{1}(J_i') + \frac{\delta}{k^+-k^-+1}$. 
Since $\{\alpha_i\} \times J_i' \subset\!\subset \omega$ and $\alpha_i<\alpha_{i+1}$, 
we may find $\rho>0$ small in such a way that the sets $(\alpha_i-\rho,\alpha_i+\rho) \times J_i'$ are still compactly contained in $\omega$, and $\alpha_{i}+\rho<\alpha_{i+1}-\rho$ for $i=k^-,\ldots,k^+$. 
Then we set for each $i=k^-,\ldots,k^+$, $U_i:=(\alpha_i-\rho, \alpha_i+\rho) \times J_i' \times I\subset \Omega$. 
Observe that the $U_i$'s are pairwise disjoint, and that $(\nabla'u,b)$ is of the form \eqref{eq:u01} in each $U_i$.  

Let us now fix an arbitrary sequence 
$\{u_n\} \subset H^{2}\bigl(\Omega; \R^3\bigr)$ such that 
$(u_n,\frac{1}{h_{n}}\partial_3 u_n) \to (u, b)$ in~$[L^1(\Omega;\R^3)]^2$.
Using Proposition \ref{gliminfref},  we estimate
$$\liminf_{n \to \infty} F^{h_{n}}_{\eps_n}(u_n)    \geq \sum_{i=k^-}^{k^+} \liminf_{n \to \infty} F^{h_{n}}_{\eps_n}\bigl(u_n,U_i\bigr) 
   \geq \sum_{i=k^-}^{k^+} K^{\star}_\gamma \mathcal{H}^{1}( J_i') 
   \geq K^{\star}_\gamma \mathcal{H}^{1}(\partial^{\ast}E\cap\omega) - 2\delta K^{\star}_\gamma\,,$$
 and the conclusion follows letting $\delta \to 0$.
\vskip5pt

\noindent{\it Step 2.} We now consider the case $A^\prime=B^\prime$ ($=0$ by \eqref{orientationwells}). Then $u_0= 0$ and $u$ is constant. 
Without loss of generality  we may assume that   $u\equiv 0$. 
In the remaining of this proof, we shall identify any $b \in L^1(Q';\R^3)$ with its extension to $Q$ given by $b(x) = b(x')$. 
With this convention, we introduce for $b \in L^1(Q';\R^3)$, 
\begin{multline*}
{\cal G}_\gamma(b) :=  \inf \Bigl\{ \liminf_{n \to \infty} F^{h_n}_{\eps_n}\bigl(u_n,Q \bigr)\,:\;
  h_n\to 0 \text{ and } \eps_n \to 0 \text{ with }h_n/\eps_n\to \gamma,\\
  \{u_n\} \subset H^{2}\bigr(Q;\R^3\bigr), \,    \text{$(u_n, \frac{1}{h_n} \partial_3 u_n) \to (0, b)$  in  $[L^1\bigl(Q;\R^3\bigr)]^2$} \Bigr\}.
\end{multline*}
Notice that $K^\star_\gamma=\mathcal{G}_\gamma(b_0)$. 
We shall require the sequential $L^1$-lower semicontinuity of the functional $\mathcal{G}_\gamma$ stated below. 
The proof of Lemma \ref{lem:lsc} only involves a standard diagonalization argument, and we shall omit it.

\begin{lemma}\label{lem:lsc}
${\cal G}_\gamma$ is sequentially lower semicontinuous for the strong $L^1$-topology. 
\end{lemma}

Let $\{u_n\} \subset H^{2}(\Omega;\R^3)$ be an arbitrary sequence  satisfying $(u_n,\frac{1}{h_n} \partial_3 u_n) \to (0,b)$ in~$[L^1(\Omega;\R^3)]^2$.
Without loss of generality, we may assume that $\liminf_{n} F^{h_n}_{\eps_n}(u_n)  = \lim_{n} F^{h_n}_{\eps_n}(u_n)< \infty$. 
By Theorem~\ref{thm:BJ} we have $b(x)=b(x')=\bigl(1-\chi_E(x')\bigr)A_3 + \chi_E(x') B_3$ for a set $E\subset \omega$ 
of finite perimeter in~$\omega$. 

Using Fubini's theorem, we define a finite nonnegative Radon measure $\mu_n$ on $\omega$ by setting 
$$
\mu_n := \left( \int_{-\frac{1}{2}}^{\frac{1}{2}} \frac{1}{\eps_n} W\left(\nabla_{h_n}u_n\right)
+ \eps_n \left|\nabla_{h_n}^2 u_n \right|^2  dx_3 \right) \mathcal{L}^{2} \restr{} \omega\,, 
$$
and $\mu_n(\omega)=F^{h_n}_{\eps_n}(u_n) $.  In particular  $\sup_n \mu_n(\omega) < \infty$, and thus  there is a subsequence (not relabeled) such that
$\mu_n \rightharpoonup \mu $ weakly* in the sense of measures for some  finite nonnegative Radon measure $\mu$ on~$\omega$. 
By lower semicontinuity we have $\mu(\omega) \leq \lim_{n} \mu_n(\omega)= \lim_{n} F^{h_n}_{\eps_n}(u_n)$. It then suffices to prove that $\mu(\omega) \geq K^{\star}_\gamma \mathcal{H}^{1}(\partial^\ast E\cap\omega)$.
By the Radon-Nikod\'ym Theorem, we can decompose $\mu$ as $\mu = \mu_0 + \mu_s$, where $\mu_0$ and  $\mu_s$ are mutually singular nonnegative Radon measures on $\omega$, 
and $\mu_0 \ll \mathcal{H}^{1}\restr{}\partial^\ast E\cap\omega$.
It is enough to show that $\mu_0(\omega) \geq K^{\star}_\gamma \mathcal{H}^{1}(\partial^\ast E\cap\omega)$ which can be obtained by proving that
$$
\frac{d \mu_0}{d \mathcal{H}^{1}\restr{}\partial^\ast E\cap\omega}(x_0') \geq K^{\star}_\gamma \quad
\text{ for } \mathcal{H}^{1}\text{-a.e. } x_0' \in\partial^\ast E\cap\omega\,.
$$

For $\nu \in \mathbb{S}^1$ and $\delta>0$, we denote by $Q'_\nu\subset \R^2$ the unit 
cube centered  at the origin with two sides orthogonal to $\nu$, and $Q'_\nu(x'_0,\delta):=x'_0+\delta Q'_\nu$. By a generalization of the Besicovitch Differentiation 
Theorem (see \cite[Proposition 2.2]{AD}), there exists a Borel  set $Z\subset \omega$ such that $\mathcal{H}^1(Z)=0$, the Radon-Nikod\'ym derivative of $\mu_0$  with respect to 
$\mathcal{H}^{1}\restr{}\partial^\ast E\cap\omega$ exists and is finite at every $x_0'\in(\partial^\ast E\cap\omega)\setminus Z$, and 
$$\frac{d \mu_0}{d \mathcal{H}^{1}\restr{}\partial^\ast E\cap\omega}(x_0')
=\lim_{\delta\to0^+}\frac{\mu_0\big(Q'_\nu(x'_0,\delta)\big)}{\mathcal{H}^1(\partial^\ast E\cap Q'_\nu(x'_0,\delta))}\quad 
\text{for every $x_0'\in(\partial^\ast E\cap\omega)\setminus Z$ and all $\nu\in\mathbb{S}^1$}\,.$$
For $ x_0' \in\partial^\ast E\cap\omega$, let us denote by $\nu_{0}\in \mathbb{S}^1$ the generalized outer normal to $E$ at $x'_0$.  By Theorem~3.59 and Remark 2.82 in \cite{AFP}, we have
\begin{equation}\label{Tgpt}
\lim_{\delta\to0^+}\frac{\mathcal{H}^1(\partial^\ast E\cap Q'_{\nu_0}(x'_0,\delta))}{\delta}=1\quad\text{ for }\mathcal{H}^{1}\text{-a.e. } x_0' \in\partial^\ast E\cap\omega\,.
\end{equation}
Moreover, it is well known (see {\it e.g.} \cite[Example 3.68]{AFP}) that  
\begin{equation}\label{sideslimits}
\lim_{\delta\to 0^+}\frac{1}{\delta^2}\int_{Q'_{\nu_0}(x'_0,\delta)} |b(x')-\ol{b}_{x_0}(x')|\,dx'=0\quad\text{ for }\mathcal{H}^{1}\text{-a.e. } x_0' \in\partial^\ast E\cap\omega\,,
\end{equation}
where 
$\ol{b}_{x_0}(x') := \chi_{\{(x'-x_0')\cdot \nu_0 > 0\}}(x')A_3 + \chi_{\{(x'-x_0')\cdot \nu_0 < 0\}}(x')B_3$. 

Let us now fix a point $x_0' \in(\partial^\ast E\cap\omega)\setminus Z$ satisfying \eqref{Tgpt}-\eqref{sideslimits}. We 
choose a sequence $\delta_k \to 0^+$ such that $Q'_{\nu_0}(x_0',\delta_k)\subset \omega$ and $\mu\bigl( \partial Q'_{\nu_0}(x_0',\delta_k)\bigr) = 0$  for all $k \in \NN$.
Then 
\begin{multline*}
\frac{d \mu_0}{d \mathcal{H}^{1}\restr{}\partial^\ast E\cap\omega}(x_0')
  = \lim_{k\to \infty} \frac{1}{\delta_k} \mu\bigl(Q'_{\nu_0}(x_0',\delta_k)\bigr) 
    = \lim_{k\to \infty} \lim_{n\to \infty} \frac{1}{\delta_k} \mu_n\bigl(Q'_{\nu_0}(x_0',\delta_k)\bigr) \\
    = \lim_{k\to \infty} \lim_{n\to \infty} \frac{1}{\delta_k} F^{h_n}_{\eps_n}\bigl(u_n,Q'_{\nu_0}(x_0',\delta_k) \times I\bigr)\,,
\end{multline*}
where we have used Fubini's theorem in the last equality. Let $R \in SO(2)$ be such that $Re'_1 = \nu_0$, and define for $x \in Q$,
$v_{n,k}(x) := \frac{1}{\delta_k} u_n\bigl(x_0' + \delta_k Rx',x_3\bigr)$. 
Observe that 
$\ol{b}_{x_0}(x_0'+Rx')= b_0(x')$. 
We also have $h_{n,k}:=\frac{h_n}{\delta_k}\to 0$, $\eps_{n,k} := \frac{\eps_n}{\delta_k} \to 0$, $h_{n,k}/\eps_{n,k}\to \gamma$, and 
$(v_{n,k},\frac{1}{h_{n,k}} \partial_3 v_{n,k})  \to (0, b_k)$ in $[L^1(Q;\R^3)]^2$ where $b_k(x') := b\bigl( x_0'+\delta_k R'x'\bigr)$.
Changing variables, we derive from assumption $(H_5)$ that
$$
F_{\eps_n}^{h_n}\bigl(u_n,Q'_{\nu_0}(x_0',\delta_k) \times I \bigr)= \delta_k F^{h_{n,k}}_{\eps_{n,k}}(v_{n,k},Q)\,.
$$
Then it follows from the definition of $\mathcal{G}_\gamma$ that 
\begin{equation*}
\frac{d \mu_0}{d \mathcal{H}^{1}\restr{}\partial^\ast E\cap\omega}(x_0')
  = \lim_{k\to\infty} \lim_{n\to\infty} F^{h_{n,k}}_{\eps_{n,k}}\bigl(v_{n,k},Q\bigr)
  \geq \liminf_{k\to\infty}\, {\cal G}_\gamma(b_k)\,.
\end{equation*}
On the other hand, by \eqref{sideslimits}  we have $b_k \to b_0$ in $L^1(Q';\R^3)$. In view of Lemma \ref{lem:lsc}, we deduce that
$$
\frac{d \mu_0}{d \mathcal{H}^{1}\restr{}\partial^\ast E\cap\omega}(x_0')\geq \liminf_{k\to\infty} \,{\cal G}_\gamma(b_k)
  \geq {\cal G}_\gamma( b_0) = K^{\star}_\gamma\,,
$$
which completes the proof.
\prbox

%
%

\subsection[]{Lower bound on $K^\star_\gamma$}

In order to compare the constant $K^\star_\gamma$ with $K_\gamma$, we prove in this subsection that under the additional assumptions $(H_3)-(H_4)$, 
sequences realizing $K^\star_\gamma$ can be  prescribed near  
the two sides $\{x_1=\pm \frac{1}{2}\}$, and chosen to be independent of the $x_2$-variable. 
This is the object of Proposition~\ref{rempot} below. 
First we state some useful facts on  potentials $W$ satisfying assumptions $(H_1)-(H_3)$ that we shall use throughout the paper. 
The proof of Lemma~\ref{rempot} is elementary and we omit it.

\begin{lemma}\label{rempot}
Assume that $(H_1)$ holds.  Then $W$ satisfies $(H_2)-(H_3)$ if and only if there exists a constant $C_*\geq 1$ such that for every $\xi\in\R^{3\times 3}$, 
\begin{equation}\label{equivpot}
\frac{1}{C_*}\min\left(|\xi-A|^p,|\xi-B|^p\right) \leq W(\xi)\leq C_*\min\left(|\xi-A|^p,|\xi-B|^p\right)\,.
\end{equation}
In particular, if $(H_1)-(H_3)$ hold, then  
$$W(\xi)\leq C^2_{*} 2^{p-1}\big(W(\bar \xi)+|\xi -\bar \xi|^p\big)\qquad \forall\xi,\bar \xi \in\R^{3\times 3} \,. $$
\end{lemma}
\vskip5pt

We now state the pinning condition described above. It parallels \cite[Proposition~6.2]{CFL} in the context of dimension reduction. 

\begin{proposition}\label{prop:matching1}
Assume that $(H_1)-(H_4)$ and  \eqref{orientationwells} hold. Then there exist sequences $\eps_n\to 0$, $\{c_n\} \subset \R^3$, 
$\{g_n\} \subset C^{2}(Q;\R^3)$ such that  $g_n$ is independent of $x_2$ (i.e., $g_n(x)=:\hat g_n(x_1,x_3)$), $c_n \to 0$, 
$g_n \to u_0$ in $W^{1,p}(Q;\R^3)$,  $\frac{1}{\gamma\eps_n}\partial_3 g_n \to  b_0$ in~$L^p(Q;\R^3)$,  
$$ g_n =  u_0 + \gamma\eps_n x_3  b_0\;\text{ in }Q\cap\{x_1 >1/4\}\,,\quad g_n = u_0 + \gamma\eps_n x_3  b_0 + c_n\; \text{ in }Q\cap\{x_1 < -1/4\}\,,$$ 
and $\lim_{n} F^{\gamma\eps_n}_{\eps_n}(g_n,Q) = K_\gamma^{\star}$. 
\end{proposition}

\begin{proof}{\it Step 1.} 
Let us consider sequences $h_n\to 0$, $\tilde \eps_n\to0$ and $\{\tilde u_n\} \subset H^{2}(Q;\R^3)$ 
such that $h_n/\tilde \eps_n \to \gamma$, 
$(\tilde u_n,\frac{1}{h_n}\partial_3 \tilde u_n) \to  (u_0, b_0)$ in $[L^1(Q;\R^3)]^2$, and
$\lim_{n} F^{h_n}_{\tilde \eps_n}(\tilde u_n,Q) = K^{\star}_\gamma$. 
Applying standard regularization techniques if necessary, we may assume  that $\tilde u_n\in C^{2}(Q;\R^3)$.
By Theorem \ref{thm:compactness}, we have $\tilde u_n \to  u_0$ in $W^{1,p}(Q;\R^3)$, and $\frac{1}{h_n}\partial_3 \tilde u_n \to  b_0$ in $L^p(Q;\R^3)$. 
Setting $\eps_n:= h_n/\gamma$, we claim that 
$\lim_{n} F^{h_n}_{\eps_n}(\tilde u_n,Q) = K^{\star}_\gamma$. 
Indeed, it suffices to notice that 
\begin{equation}\label{changeeps}
\left|F^{h_n}_{\eps_n}(\tilde u_n,Q)-F^{h_n}_{\tilde \eps_n}(\tilde u_n,Q)\right| \leq \left(\big|1-\frac{\gamma\tilde\eps_n}{h_n}\big| 
+ \big|1-\frac{h_n}{\gamma\tilde\eps_n}\big| \right) F^{h_n}_{\tilde \eps_n}(\tilde u_n,Q)\mathop{\longrightarrow}\limits_{n\to\infty} 0\,.
\end{equation}

Extracting a further subsequence if necessary, we can find an exceptional set $Z\subset I$ of 
vanishing $\mathcal{H}^1$-measure such that  for every $x_2\in I\setminus Z$, the slices $\tilde u_n(\cdot,x_2, \cdot)$ 
and $\frac{1}{h_n}\partial_3 \tilde u_n(\cdot,x_2, \cdot)$ converge 
to $u_0$ in $W^{1,p}(I\times\{x_2\}\times I;\R^3)$ and $b_0$ in $L^{p}(I\times\{x_2\}\times I;\R^3)$ respectively. 
We select a level $s_n\in I\setminus Z$ satisfying
\begin{equation}\label{eq:sn}
\int_{I\times\{s_n\} \times I} \frac{1}{\eps_n}W(\nabla_{h_n}\tilde u_n)+\eps_n|\nabla_{h_n}^2\tilde u_n|^2\,d\mathcal{H}^2
 \leq  F^{h_n}_{\eps_n}(\tilde u_n,Q)\,.
 \end{equation} 
From now on we write 
$$u_n(x):= \tilde u_n(x_1,s_n,x_3)\quad\text{and}\quad \hat u_n(x_1,x_3):=\tilde u_n(x_1,s_n,x_3)\,.$$
By our choice of $s_n$, we have $u_n \to  u_0$ in 
$W^{1,p}(Q;\R^3)$, and $\frac{1}{h_n}\partial_3  u_n \to  b_0$ in $L^p(Q;\R^3)$. Since $\partial_2 u_n\equiv 0$, 
assumption $(H_4)$ together with \eqref{eq:sn} yields
$$\limsup_{n\to\infty} F^{h_n}_{\eps_n}(u_n,Q) \leq  \limsup_{n\to\infty} \int_{I\times\{s_n\} \times I} \frac{1}{\eps_n}W(\nabla_{h_n}\tilde u_n)+\eps_n|\nabla_{h_n}^2\tilde u_n|^2\,d\mathcal{H}^2 \leq  K^{\star}_\gamma\,.$$
On the other hand,  $\liminf_n  F^{h_n}_{\eps_n}(u_n,Q)\geq K^{\star}_\gamma$ by definition of $K^{\star}_\gamma$. Hence
$\lim_{n}  F^{h_n}_{\eps_n}(u_n,Q) = K^{\star}_\gamma$. 
\vskip5pt

\noindent{\it Step 2 (first matching).} 
We start partitioning $\bigl(\frac{1}{12},\frac{1}{6}\bigr) \times Q'$ into $M_n := \bigl[ \frac{1}{\eps_n}\bigr]$  layers of width $\frac{1}{12M_n}$ ($[\cdot]$~denotes the integer part). 
By Corollary \ref{cor:concentration}, the energy concentrates near the interface $\{x_1=0\}$. Therefore we can find a layer 
$L_n := \bigl(\theta_n - \frac{1}{12M_n},\theta_n\bigr) \times Q' \subset \bigl(\frac{1}{12},\frac{1}{6}\bigr) \times Q'$ for which
\begin{multline}\label{eq:layer}
M_n \left( \int_{L_n} |u_n- u_0|^p + \big|\frac{1}{h_n}\partial_3 u_n- b_0\big|^p + |\grad u_n -\grad  u_0|^p\, dx + F^{h_n}_{\eps_n}(u_n,L_n)\right)\,\leq  \\
   \int_{\bigl(\frac{1}{12},\frac{1}{6}\bigr) \times Q'}  |u_n- u_0|^p + \big|\frac{1}{h_n}\partial_3 u_n- b_0\big|^p+ |\grad u_n -\grad  u_0|^p\, dx
   + F^{h_n}_{\eps_n}(u_n,\bigl({\ts \frac{1}{12},\frac{1}{6}}\bigr) \times Q')  =: \alpha_n \to 0\,.
\end{multline}
Then  select a level $t_n \in \bigl(\theta_n - \frac{1}{12M_n},\theta_n\bigr)$ satisfying
\begin{multline}\label{eq:tn}
\int_{\{t_n\} \times Q'}  |u_n- u_0|^p + \big|\frac{1}{h_n}\partial_3 u_n- b_0\big|^p + |\grad u_n -\grad  u_0|^p\, d\mathcal{H}^2\\
+\int_{\{t_n\} \times Q'} \frac{1}{\eps_n}W(\nabla_{h_n} u_n)+\eps_n|\nabla_{h_n}^2 u_n|^2\,d\mathcal{H}^2\leq 12\alpha_n\,.
\end{multline}
Let $\varphi_n\in C^\infty(\R)$ be  a cut-off function 
satisfying 
\begin{equation}\label{glutest}
\begin{cases}
0\leq \varphi_n\leq 1\,,\\
\varphi_n(t) = 1  \text{ for } t \leq \theta_n - \frac{1}{12M_n}\,,\\
\varphi_n(t) = 0   \text{ for } t \geq \theta_n\,, \\
\eps_n| \varphi_n'|+\eps_n^2 | \varphi_n''|\leq C\,,
\end{cases}
\end{equation}
for a constant $C$ independent of $n$. For $x \in L_n$, we  set  
$$
v_n(x) := \bigl( 1-\varphi_n(x_1)\bigr) \bigl(u_0(x)+h_n x_3 b_0(x) + \bar u_n(x_3)\bigr) + \varphi_n(x_1) u_n(x) \,,
$$
with 
$$\bar u_n(x_3) := \hat u_n(t_n,x_3) - \bar u_0(t_n)-x_3\int_{I}\partial_3 \hat u_n(t_n,s)\, ds\,.$$
We claim that
\begin{align}
\label{grplim1} &\int_{L_n} |v_n -  u_0|^p\,  dx \to 0\,,\\ 
\label{grplim2} & \frac{1}{\eps_n} \int_{L_n} \big|\frac{1}{h_n} \partial_3 v_n -  b_0\big|^p \,dx \to 0\,,\\
\label{grplim3} & \frac{1}{\eps_n} \int_{L_n} |\grad' v_n - \grad'  u_0|^p\, dx \to 0\,,\\
\label{grplim4} & \frac{1}{\eps_n} \int_{L_n} W(\nabla_{h_n} v_n) \,dx \to 0\,,\\
\label{grplim5}  & \eps_n \int_{L_n}  \left|\nabla_{h_n}^2 v_n \right|^2\, dx \to 0\,.
\end{align}
Applying Jensen's inequality,  we derive from \eqref{eq:tn} that 
\begin{equation}\label{vanlpbaru}
\int_{I}|\bar{u}_n|^p\,dx_3\leq C\int_{\{t_n\}\times Q'}|u_n- u_0|^p+|\partial_3u_n|^p\,d\mathcal{H}^2 \leq C\alpha_n\to 0\,. 
\end{equation}
Then \eqref{grplim1} easily follows from \eqref{eq:layer}, \eqref{eq:tn}, and \eqref{vanlpbaru}. 
To prove \eqref{grplim2}, we first estimate for $x_3\in I$, 
\begin{equation}\label{cossslice}
 |\bar{u}_n^\prime(x_3)|
\leq \int_I |\partial_{33}^2\hat u_n(t_n,s)|\,ds \leq  \left(\int_{\{t_n\}\times Q'}|\partial_{33}^2u_n|^2\,d\mathcal{H}^2\right)^{1/2}\leq C\alpha_n^{1/2}\eps_n^{3/2}\,,
\end{equation}
where we have used Poincar\'e's inequality, H\"older's inequality, and \eqref{eq:tn}.
We may now infer that 
$$\frac{1}{\eps_n}\int_{L_n} \big|\frac{1}{h_n} \partial_3 v_n -  b_0\big|^p\, dx \leq C\left( \frac{1}{\eps_n}\int_{L_n} \big|\frac{1}{h_n} \partial_3 u_n -  b_0\big|^p\,  dx+
\frac{1}{\eps_n^p}\int_{I} |\bar{u}^\prime_n|^p\,dx_3\right)\leq C\alpha_n\to 0\,,$$
thanks to \eqref{eq:layer} and \eqref{cossslice}. 
Observing that 
$$u_n -  u_0-h_nx_3 b_0 - \ol{u}_n=x_3\int_{\{t_n\}\times Q'}\big(\partial_3 u_n-h_n b_0\big)\,d\mathcal{H}^2 \quad \text{on $\{t_n\}\times Q'$}\,, $$
we can apply Poincar\'e's inequality to obtain
\begin{multline}\label{eq:Poincare}
\int_{L_n} |u_n -  u_0 -h_nx_3  b_0- \bar{u}_n |^p \,dx \leq \frac{C}{M_n}\int_{\{t_n\}\times Q'}\big|\partial_3 u_n-h_n b_0\big|^p\,d\mathcal{H}^2\\
+ C \left(\frac{1}{M_n}\right)^p \int_{L_n} | \partial_1 u_n - \partial_1  u_0|^p\,  dx\leq C\alpha_n \eps^{p+1}_n\,.
\end{multline}
Then, using \eqref{eq:layer},  \eqref{eq:tn}, \eqref{glutest}, and \eqref{eq:Poincare} we derive 
$$\frac{1}{\eps_n} \int_{L_n} |\grad' v_n - \grad' u_0|^p\,  dx    \leq  \frac{C}{\eps_n} \int_{L_n} \bigg(\frac{1}{\eps_n^p}| u_n -  u_0-h_nx_3  b_0 - \bar{u}_n|^p +  
         |\grad' u_n - \grad'  u_0|^p \bigg)\, dx  \leq C \alpha_n \to 0\,,$$
and \eqref{grplim3} is proved. 
In view of \eqref{equivpot}, estimate \eqref{grplim4} follows from \eqref{grplim2} and  \eqref{grplim3}, {\it i.e.}, 
\begin{align}
\nonumber \frac{1}{\eps_n} \int_{L_n} W(\nabla_{h_n} v_n) \,dx & \leq  \frac{C_\star}{\eps_n} \int_{L_n} \min\left(|\nabla_{h_n} v_n-A|^p,|\nabla_{h_n} v_n-B|^p\right) \,dx \\
\label{estimatchdist} & \leq  \frac{C_\star}{\eps_n} \int_{L_n} \left|\nabla_{h_n} v_n-(\nabla' u_0, b_0)\right|^p dx \to 0\,.
\end{align}
We prove  \eqref{grplim5} in separate parts. In view of \eqref{glutest} we have 
\begin{multline*}
  \eps_n \int_{L_n} |(\grad')^2 v_n|^2 \, dx
    \leq C\bigg(\eps_n \int_{L_n}  |(\grad')^2 u_n|^2 \,dx +\frac{1}{\eps_n} \int_{L_n} |\grad' u_n - \grad'  u_0 |^2 \,dx \\
          +\frac{1}{\eps^3_n} \int_{L_n}  |u_n - u_0 -h_nx_3  b_0- \bar{u}_n|^2\, dx \bigg)\,,
\end{multline*}
and we shall estimate each term separately. The first  term on the right-hand-side of the inequality converges to 0 by \eqref{eq:layer}.
For the last two terms, we use \eqref{eq:layer} and \eqref{eq:Poincare} together with H\"older's inequality to obtain
$$\frac{1}{\eps_n} \int_{L_n}|\nabla' u_n - \grad' u_0 |^2 \,dx\leq \frac{|L_n|^{\frac{p-2}{p}}}{\eps_n}\left(\int_{L_n} |\grad' u_n - \grad' u_0 |^p \,dx\right)^{2/p} 
\leq C\alpha_n^{2/p}\to 0\,,$$
and
\begin{multline*}
\frac{1}{\eps^3_n} \int_{L_n} |u_n - u_0 -h_nx_3 b_0 - \bar{u}_n|^2  \,dx
   \leq \frac{|L_n|^{\frac{p-2}{p}}}{\eps_n^3}  \left(\int_{L_n} |u_n - u_0 -h_nx_3 b_0- \bar{u}_n|^p \, dx\right)^{2/p} 
   \leq C \alpha_n^{2/p} \to 0\,,
\end{multline*}
and we conclude that $ \eps_n \int_{L_n} |(\grad')^2 v_n|^2 \, dx\to 0$. 
Finally we estimate
\begin{multline*}
\eps_n \int_{L_n} \big| \nabla' \big(\frac{1}{h_n} \partial_3 v_n\big) \big|^2\, dx 
  \leq C\bigg(\eps_n \int_{L_n} \big| \nabla' \big(\frac{1}{h_n} \partial_3 u_n\big) \big|^2 \, dx \\
  + \frac{1}{\eps_n} \int_{L_n} \big| \frac{1}{h_n}\partial_3 u_n- b_0 \big|^2\, dx
+ \frac{1}{\eps_n^2}\int_{I}|\bar{u}'_n|^2\,dx_3\bigg)\leq C\bigg(\alpha_n+\alpha_n^{2/p}\bigg)\to 0\,,
\end{multline*}
where we have  used again \eqref{glutest}, H\"older's inequality, \eqref{eq:layer}, and \eqref{cossslice}. 
Since $\bar{u}^{\prime\prime}_n(x_3)=\partial_{33}^2\hat u_n(t_n,x_3)$,  we infer from  \eqref{eq:layer}, \eqref{eq:tn}, and \eqref{glutest} that 
$$\eps_n  \int_{L_n}\frac{1}{h^4_n}\left| \partial^2_{33} v_n \right|^2  \, dx 
  \leq C\bigg( \frac{\eps_n}{h^4_n}\int_{L_n}\left| \partial^2_{33} u_n \right|^2 \, dx 
  + \frac{\eps^2_n}{h^4_n} \int_{\{t_n\}\times Q'} \left|\partial^2_{33} u_n \right|^2\,  d\mathcal{H}^2 \bigg)\leq C\eps_n\alpha_n\to 0\,,$$
which ends the proof of  \eqref{grplim5}. 
\vskip5pt

\noindent{\it Step 3 (second matching).} Let $\psi_n\in C^\infty(\R)$ be  
such that $0\leq \psi_n\leq 1$, $\psi_n(t) = 1$ if $t \leq \theta_n$, $\psi_n(t) = 0$ if $t \geq 1/4$, and
$| \psi_n'| +| \psi_n''| \leq C$ for a constant $C$ independent of $n$. 
For $x \in \bigl(\theta_n,\frac{1}{4}\bigr) \times Q'$, we set  
$$
w_n(x) := u_0(x) +h_n x_3b_0(x)+ c^+_n + \psi_n(x_1) \bigl( \ol{u}_n(x_3) - c^+_n\bigr)\,,
$$
where $c^+_n := \int_{I} \bar{u}_n(x_3) \, dx_3\to 0$ thanks to \eqref{vanlpbaru}.  We claim that
\begin{align}
\label{grplim1bis} &\int_{(\theta_n,\frac{1}{4}) \times Q'} |w_n -  u_0|^p\, dx \to 0\,,\\
\label{grplim2bis} & \frac{1}{\eps_n} \int_{(\theta_n,\frac{1}{4}) \times Q'} \big|\frac{1}{h_n} \partial_3 w_n -  b_0\big|^p\, dx \to 0\,,\\
\label{grplim4bis} & \frac{1}{\eps_n} \int_{(\theta_n,\frac{1}{4}) \times Q'} W\left(\grad_{h_n} w_n\right) \, dx \to 0\,,\\
\label{grplim5bis}  & \eps_n \int_{(\theta_n,\frac{1}{4}) \times Q'} \left|\nabla_{h_n}^2 w_n \right|^2\, dx \to 0\,.
\end{align}
First, \eqref{grplim1bis} and \eqref{grplim2bis} are easy consequences of \eqref{eq:tn} and \eqref{cossslice} respectively.  
Next we apply Poincar\'e's inequality and  \eqref{cossslice} to derive that 
\begin{equation}\label{grplim3bis} 
\frac{1}{\eps_n}\int_{\bigl(\theta_n,\frac{1}{4}\bigr) \times Q'}|\nabla'w_n-\nabla' u_0|^p\,dx\leq \frac{C}{\eps_n}\int_{I}|\bar{u}_n-c^+_n|^p\,dx_3
 \leq \frac{C}{\eps_n}\int_{I}|\bar{u}'_n |^p\,dx_3 \leq C\eps_n^{\frac{3p-2}{2}}\alpha^{\frac{p}{2}}_n\to 0\,.
\end{equation}
Then, to prove \eqref{grplim4bis} we argue exactly as in \eqref{estimatchdist} using  \eqref{grplim2bis} and \eqref{grplim3bis}. 
We finally obtain in a similar way that
\begin{multline*}
\eps_n \int_{\bigl(\theta_n,\frac{1}{4}\bigr) \times Q'} \left|\nabla_{h_n}^2 w_n \right|^2 \,dx\leq\\
C\bigg(\eps_n\int_{\bigl(\theta_n,\frac{1}{4}\bigr) \times Q'} |\bar{u}_n-c^+_n|^2\,dx
+\frac{1}{\eps_n}\int_{I}|\bar{u}'_n|^2\,dx_3+\frac{\eps_n}{h_n^4}\int_{\{t_n\}\times Q'}|\partial_{33}^2u_n|^2\,d\mathcal{H}^2\bigg)\leq C\alpha_n\to 0\,,
\end{multline*}
and \eqref{grplim5bis} is proved.
\vskip5pt

\noindent{\it Step 4.} To conclude the proof, we first set for $x\in Q$, 
\begin{equation}\label{defgn+}
g^+_n(x):=\begin{cases}
u_n(x) & \text{for $x_1<\theta_n - \frac{1}{12M_n}$}\,,\\[5pt]
v_n(x) & \text{for $\theta_n - \frac{1}{12M_n}\leq x_1< \theta_n$}\,, \\[5pt]
w_n(x) & \text{for $\theta_n \leq x_1< \frac{1}{4}$}\,, \\[5pt]
u_0(x) +h_nx_3b_0(x) +c^+_n  & \text{for $\frac{1}{4}\leq x_1\leq \frac{1}{2}$}\,.
\end{cases}
\end{equation}
Recalling that $h_n=\gamma\eps_n$, it follows from the previous steps and Corollary \ref{cor:concentration} that $g_n^+\in C^{2}(Q;\R^3)$,  
$g_n^+\to u_0$ in $W^{1,p}(Q;\R^3)$, $\frac{1}{\gamma\eps_n}\partial_3g_n^+\to b_0$ in $L^p(Q;\R^3)$, and 
$\lim_{n} F_{\eps_n}^{\gamma\eps_n}(g_n^+,Q)= \lim_{n} F_{\eps_n}^{\gamma\eps_n}(u_n,Q)=K_\gamma^\star$. 
The sequence $\{g_n^+\}$ satisfies the pinning condition $g^+_n= u_0+\gamma\eps_nx_3 b_0+c_n^+$ in $Q\cap \{x_1>1/4\}$. 
Then we repeat construction to modify $g^+_n$ in $(-\frac{1}{2},0)\times Q'$ in order build a new field $g_n^-\in C^2(Q;\R^3)$ satisfying 
$g^-_n= u_0+\gamma\eps_nx_3 b_0+c_n^-$ 
in $Q\cap \{x_1<-1/4\}$ for 
some constants $c_n^-\to 0$. Now it suffices to set $g_n:=g_n^--c_n^+$ and $c_n:=c_n^--c_n^+$. By construction $g_n$ does not dependent on $x_2$, that is 
$g_n(x)=:\hat g_n(x_1,x_3)$. 
\end{proof}

\begin{corollary}\label{prop:charK_per}
Assume that $(H_1)-(H_4)$ and \eqref{orientationwells} hold. 
Then $K^\star_\gamma\geq K_\gamma\,$. 
\end{corollary}

\begin{proof} 
We consider the sequences $\{\eps_n\}$ and $\{g_n\}$ given by Proposition \ref{prop:matching1}. 
Remind that $g_n(x)=\hat g_n(x_1,x_3)$. We set $\ell_n:= 1/\eps_n$, and for $y=(y_1,y_2)\in \ell_n  I\times \gamma I$, 
$v_n(y):=\frac{1}{\eps_n}\hat g_n\big(\eps_ny_1,y_2/\gamma \big)$. 
Then straightforward computations yield $\nabla v_n(y)=(\bar u'_0(y_1),\bar b_0(y_1))$ nearby $\{|y_1|=\ell_n/2\}$, 
and 
$$K_\gamma\leq\frac{1}{\gamma}\int_{\ell_n I\times \gamma I}\mathcal{W}(\nabla v_n)+|\nabla^2 v_n|^2\,dy= F_{\eps_n}^{\gamma\eps_n}(g_n,Q)\,.$$ 
By construction of $\{g_n\}$, the conclusion follows letting $n\to \infty$. 
\end{proof}

%
%

\subsection[]{The $\Gamma$-$\limsup$ inequality}

We conclude this section with the construction of a recovery sequence. Then Theorem \ref{thm:RS_A'neqB'} together with Corollary \ref{prop:charK_per} and Theorem \ref{thm:gammaliminf} concludes 
the proof of Theorem \ref{thm:gammalim_gamma1}. 

\begin{theorem}\label{thm:RS_A'neqB'}
Assume that $(H_1)-(H_5)$ and  \eqref{orientationwells} hold. Let $\eps_n\to 0$ and $h_n\to 0$ be arbitrary sequences such that $h_n/\eps_n \to \gamma$. Then, for every $(u,b) \in \mathscr{C}$,  
there exists a sequence $\{u_n\} \subset H^{2}(\Omega;\R^3)$ such that $u_n\to u$ in $W^{1,p}(\Omega;\R^3)$, $\frac{1}{h_n}\partial_3 u_n \to b$ in $L^p(\Omega;\R^3)$ and 
\begin{equation}\label{gamlimsupcrit}
\lim_{n\to\infty} F^{h_n}_{\eps_n}(u_n) = K_{\gamma} \,\Per_{\omega}(E) \,, 
\end{equation}
where $(\nabla' u,b) (x)= \bigl(1-\chi_E(x')\bigr)A + \chi_E(x') B$.
\end{theorem}

\begin{proof}{\it Step 1.}  We first assume that $A'\not=B'$, so that $\partial^\ast E\cap \omega$ is of the form~\eqref{structredbdAdiffB} by  Theorem~\ref{thm:BJ}. We also assume 
that it is made by finitely many interfaces, {\it i.e.}, $\mathscr{I}=\{1,\ldots,m\}$ in  \eqref{structredbdAdiffB}.  
In this case, by  Theorem~\ref{thm:BJ}, we have   
$u(x)=\bar u(x_1) $ and $b(x)=\bar b(x_1)$ where $\bar u \in W^{1,\infty}((\alpha_{\rm min},\alpha_{\rm max});\R^3)$,  
$(\bar u',\bar b)\in BV((\alpha_{\rm min},\alpha_{\rm max});\{(a,A_3),(-a,-A_3)\})$, and $(\alpha_{\rm min},\alpha_{\rm max})$ is  defined by \eqref{defalpha}.  
Without loss of generality we may assume that $\bar u'(x_1)=-a$ for $x_1<\alpha_1$. 
Then $\bar u'(t)=a$ for $\alpha_i<t<\alpha_{i+1}$ if $i$ 
is odd, $\bar u'(t)=-a$ for $\alpha_i<t<\alpha_{i+1}$ if $i$ 
is even, and $\bar u'(t)=a$ or $\bar u'(t)=-a$  for $t>\alpha_{m}$ if $m$ 
is odd or even respectively. 

Let us  consider for each $k\in\NN$, some $\ell_k>0$ and $v_k\in C^{2}(\ell_k  I\times \gamma I;\R^3)$ such that  $\nabla v_k(y)=(\bar u'_0,\bar b_0)(y_1)$ 
nearby $\{|y_1|=\ell_k/2\}$, and 
\begin{equation}\label{envkcritA=B}
\frac{1}{\gamma}\int_{\ell_k I\times \gamma I}\mathcal{W}(\nabla v_k)+|\nabla^2v_k|^2\,dy\leq K_\gamma + 2^{-k}\,. 
\end{equation}
Subtracting a constant to $v_k$ if necessary, we may assume that 
\begin{equation}\label{fixbdvk}
v_k(y)=\begin{cases}
\ds ay_1+A_3y_2+c_k & \text{nearby $\{y_1=\ell_k/2\}$}\,,\\[5pt]
\ds -ay_1+B_3y_2-c_k &\text{ nearby $\{y_1=-\ell_k/2\}$}\,,  
\end{cases}
\end{equation}
for some  $c_k\in\R^3$.

Let $h_n\to 0$ be an arbitrary sequence, and without loss of generality we can choose  $\eps_n:=h_n/\gamma$ (see~\eqref{changeeps}). 
We  fix for each $i=1,\ldots,m$, a  bounded open interval $J'_i\subset \R$ such that 
\begin{equation}\label{intervlimsup}
J_i\subset\!\subset J'_i \quad\text{and}\quad \mathcal{H}^1(J'_i\setminus J_i)\leq 2^{-k}\,,
\end{equation}
and we shall consider integers $n$ large enough in such a way  that $\alpha_i+\ell_k\eps_n/2<\alpha_{i+1}-\ell_k\eps_n/2$ for every $i=1,\ldots,m-1$. We write  for each $i=1,\ldots,m$, 
\begin{equation}\label{notgamsup}
\alpha^n_{i-}:=\alpha_i-\frac{\ell_k\eps_n}{2}\quad\text{and}\quad\alpha^n_{i+}:=\alpha_i+\frac{\ell_k\eps_n}{2}\,. 
\end{equation}
Note that by convexity of $\omega$, 
\begin{equation}\label{bandeconvex}
\left((\alpha^n_{i-},\alpha^n_{i+})\times\R\right) \cap \omega \subset (\alpha^n_{i-},\alpha^n_{i+})\times J'_i 
\end{equation}
whenever $n$ is sufficiently large.  

We define the transition layer near each interface as follows: for each $i=1,\ldots,m$, we set for $x\in (\alpha^n_{i-},\alpha^n_{i+})\times J'_i\times I$, 
$$w^i_{n,k}(x):=(-1)^{i+1}v_k\left((-1)^{i+1}\frac{x_1-\alpha_{i}}{\eps_n}\,, (-1)^{i+1}\gamma x_3\right)+\left(1+(-1)^{i}\right)\left(a\frac{\ell_k}{2}+c_k\right)\,.$$
Observe that \eqref{fixbdvk} yields 
\begin{equation}\label{bdlayer}
w_{k,n}^i(\alpha^n_{i-},x_2,x_3)=
 v_k\left((-1)^i\,\frac{\ell_k}{2},\gamma x_3\right) \,,
 \end{equation}
and
\begin{equation}\label{bdlayer+}
w_{k,n}^i(\alpha^n_{i+},x_2,x_3)= v_k\left((-1)^{i+1}\frac{\ell_k}{2},\gamma x_3\right)+2\left(1+(-1)^i\right)c_k \,.
\end{equation}
Setting
\begin{equation}\label{notgamsupbis}
\beta^n_i:=\sum_{j=1}^i\bar u(\alpha^n_{j+})-\bar u(\alpha^n_{j-})\quad\text{and}\quad\kappa_i:=2\sum_{j=1}^i\left(1+(-1)^j\right)\,,
\end{equation}
with $\beta_0^n:=0$, $\kappa_0:=0$, we finally define for $n$ large enough and $x\in\Omega$, 
$$
u_{n,k}(x):= \begin{cases}
\ds \bar u(x_1) +\eps_nv_k\left(-\frac{\ell_k}{2},\gamma x_3\right) & \text{for } x_1\leq \alpha^n_{1-}\,,\\[10pt]
\ds  
\bar u(\alpha^n_{i-})-\beta^n_{i-1}+\eps_nw^i_{k,n}(x)+\eps_n\kappa_{i-1}c_k 
& \text{for } \alpha^n_{i-}< x_1< \alpha^n_{i+}\,,\\[10pt]
\ds \bar u(x_1)-\beta^n_{i} +\eps_nv_k\left((-1)^{i+1}\frac{\ell_k}{2},\gamma x_3\right)+\eps_n\kappa_{i}c_k&\text{for } \alpha^n_{i+}\leq x_1\leq \alpha^n_{(i+1)-}\,,\\[10pt]
\ds \bar u(x_1)-\beta^n_{m} +\eps_nv_k\left((-1)^{m+1}\frac{\ell_k}{2},\gamma x_3\right)+\eps_n\kappa_{m}c_k&\text{for } x_1\geq \alpha^n_{m+}\,.
\end{cases}
$$
In view of \eqref{bdlayer}-\eqref{bdlayer+} we have $u_{n,k}\in H^2(\Omega;\R^3)$. Moreover,  
$u_{n,k}$ does not depend on the $x_2$-variable, and 
\begin{equation}\label{structgradscaled}
\left(\partial_1 u_{n,k},\frac{1}{h_n}\partial_3 u_{n,k}\right)(x)=\begin{cases}
\ds \nabla v_k \left((-1)^{i+1}\frac{x_1-\alpha_{i}}{\eps_n}\,, (-1)^{i+1}\gamma x_3\right) & \text{if }\alpha^n_{i-}< x_1< \alpha^n_{i+}\,,\\[8pt]
(\bar u',\bar b)(x_1) & \text{otherwise}\,.
\end{cases}
\end{equation}
Since $\bar u$ is Lipschitz continuous, we have $|\beta_i^n|\leq C\eps_n$ for a constant $C$ independent of $n$.  In addition, 
$v_k$ and $\nabla v_k$ are bounded, and we infer that   
$u_{n,k}\to u$ in $W^{1,p}(\Omega;\R^3)$ and $\frac{1}{h_n}\partial_3u_{n,k}\to b$ in $L^p(\Omega;\R^3)$ as $n\to\infty$.  
Using \eqref{bandeconvex}, \eqref{structgradscaled}, and changing variables, we estimate
\begin{align*} 
F^{h_n}_{\eps_n}(u_{n,k})
& \leq \sum_{i=1}^mF^{h_n}_{\eps_n}\big(\eps_n w^i_{n,k},(\alpha^n_{i-},\alpha^n_{i+})\times J'_i\times I\big)  \\
& \leq \sum_{i=1}^m \frac{\mathcal{H}^{1}(J'_i)}{\gamma}\int_{\ell_k I\times\gamma I}\mathcal{W}(\nabla v_k)+|\nabla^2v_k|^2\,dy \\
& \leq K_\gamma \,\Per_{\omega}(E) +C_02^{-k}\,,
\end{align*}
for a constant $C_0$ which only depends on $m$ and $\Per_{\omega}(E)$. 
 
For each $k\in\NN$, we can now  find  $N_k\in\NN$ such that 
$$\|u_{n,k}-u\|_{W^{1,p}(\Omega)}\leq 2^{-k}\,,\quad\|\frac{1}{h_n}\partial_3 u_{n,k}-b\|_{L^p(\Omega)}\leq 2^{-k}\,,\quad F^{h_n}_{\eps_n}(u_{n,k})
\leq K_\gamma \,\Per_{\omega}(E) +C_02^{-k}$$ 
for every $n\geq N_k$. Moreover we can assume that  the resulting sequence $\{N_k\}$ satisfies $N_k<N_{k+1}$ for every $k\in \NN$. Then for every  $n\geq N_0$, 
there exists a unique $k_n$ such that $N_{k_n}\leq n<N_{k_n+1}$, and  $k_n\to+\infty$ as $n\to+\infty$. We define $u_n:=u_{n,k_n}$
and it follows that $u_n\to u$ in $W^{1,p}(\Omega;\R^3)$, $\frac{1}{h_n}\partial_3 u_n\to b$ in $L^p(\Omega;\R^3)$, 
\begin{equation} \label{limsupapprox}
\limsup_{n\to\infty} \, F^{h_n}_{\eps_n}(u_{n})\leq K_\gamma \,\Per_{\omega}(E)\,.
\end{equation}
Finally \eqref{gamlimsupcrit} holds by \eqref{limsupapprox}, Theorem~\ref{thm:gammaliminf}, and Corollary \ref{prop:charK_per}.  
\vskip5pt

\noindent{\it Step 2.}  We now consider the case where $A'\not=B'$ and $\partial^\ast E\cap \omega$ is made by infinitely many interfaces, {\it i.e.}, $\partial^\ast E\cap \omega$ 
is as in \eqref{structredbdAdiffB} with $\mathscr{I}\subset\ZZ$ infinite. We may assume for simplicity that $\mathscr{I}=\NN$. The general case can be recovered 
from the discussion below with the obvious modifications. 

By Theorem \ref{thm:BJ}, we have 
$\lim_{k}\sum_{i=0}^k\mathcal{H}^1(J_i) = \sum_{i\in \NN}\mathcal{H}^1( J_i)=\Per_{\omega}(E)$, 
and $\alpha_{k}$ converges to $\alpha_{\rm max}$.  
For $k\in\NN$ large enough, we define some $u_k\in W^{1,\infty}(\Omega;\R^3)$ in the following way:  we set $u_k(x):=u(x)$  for $x\in\Omega\cap\{x_1<\alpha_{k+1}\}$, 
and we extend $u_k$ to be affine in the remaining of $\Omega$  in such a way that $u_k$ and $\nabla u_k$ are continuous across the interface $\{x_1=\alpha_{k+1}\}$.  
Similarly we define for $x\in\Omega\cap\{x_1<\alpha_{k+1}\}$, $b_k(x):=b(x)$, and we extend $b_k$ by a suitable constant in the remaining of $\Omega$  so that 
it remains continuous across $\{x_1=\alpha_{k+1}\}$. Then one may check that $(u_k,b_k)\in\mathscr{C}$, and that $(\nabla'u_k,b_k)=
\bigl(1-\chi_{E_k}(x')\bigr)A + \chi_{E_k}(x') B$ with $\partial^*E_k\cap\omega=\bigcup_{i=0}^k \{\alpha_i\} \times J_i$. 
Moreover, using the fact that $\alpha_{k}\to \alpha_{\rm max}$, we derive that 
$u_k\to u$ in~$W^{1,p}(\Omega;\R^3)$ and $b_k\to b$ in $L^p(\Omega;\R^3)$. 

Let $h_n\to 0$ and $\eps_n\to 0$ be arbitrary sequences such that $h_n/\eps_n\to\gamma$. 
Since $\partial^*E_k\cap\omega$ is made by finitely many interfaces, by Step 1, we can find $\{u_{n,k}\}\subset H^2(\Omega;\R^3)$ such that $u_{n,k}\to u_k $ 
in$W^{1,p}(\Omega\R^3)$, $\frac{1}{h_n}\partial_3u_{n,k}\to b_k$ in $L^p(\Omega;\R^3)$, and 
$\lim_{n}F_{\eps_n}^{h_n}(u_{n,k})=K_\gamma \sum_{i=0}^k\mathcal{H}^1(J_i)$. 
Then the conclusion follows for a suitable diagonal sequence $u_n:=u_{n,k_n}$ as already pursued in Step 1. 
\vskip5pt

\noindent{\it Step 3.} We finally treat the case $A'=B'$ ($=0$ by \eqref{orientationwells}). Without loss of generality we may assume   that $u= 0$. 
According to Theorem \ref{thm:BJ}, we have $b(x)=(1-\chi_E(x'))A_3+\chi_E(x') B_3$ where $E\subset\omega$  is a set of finite perimeter in $\omega$. 
By Lemma~4.3 in \cite{AFMT}, we can find a sequence $\{E_k\}$ of bounded open sets in $\R^2$ with smooth 
boundary such that $\chi_{E_k}\to\chi_E$ in $L^1(\omega)$, and $\lim_{k}\mathcal{H}^1(\partial E_k\cap \ol{\omega})= \Per_{\omega}(E)$. 
We define for $x\in\Omega$, $b_k(x):= (1-\chi_{E_k}(x'))A_3+\chi_{E_k}(x') B_3$, so that $b_k\to b$ in $L^p(\Omega;\R^3)$. 
Since $\mathcal{M}^k:=\partial  E_k$ is a smooth submanifold of $\R^2$, for every $k\in\NN$ we can find  $\delta_k>0$ such that the nearest 
point projection onto $\mathcal{M}^k$ is well defined and smooth in the tubular $\delta_k$-neighborhood  
$$U_k:=\{x'\in\R^2 : {\rm dist}(x',\mathcal{M}^k)<\delta_k\}\,.$$ 
We define the signed distance to $\mathcal{M}^k$ as the function $d_k:\R^2\to [0,+\infty)$ given~by 
\begin{equation}\label{signdist}
d_k(x'):=\begin{cases}
-{\rm dist}(x',\mathcal{M}^k) & \text{if $x\in E_k$}\,,\\
{\rm dist}(x',\mathcal{M}^k) & \text{otherwise}\,.
\end{cases}
\end{equation}
Then $d_k$ is smooth in $U_k$, 
the level sets $\{d_k=t\}=:\mathcal{M}^k_t$ are smooth for all $t\in(-\delta_k,\delta_k)$, and the function $t\in(-\delta_k,\delta_k)\mapsto \mathcal{H}^1(\mathcal{M}^k_t\cap\overline\omega)$ 
is upper semicontinuous (see {\it e.g.} \cite[Proposition 1.62]{AFP}). In particular,  
\begin{equation}\label{limlenglevs}
\limsup_{t\to 0}\,\mathcal{H}^1(\mathcal{M}^k_t\cap\overline\omega)\leq\mathcal{H}^1(\mathcal{M}^k\cap\ol{\omega})\,.
\end{equation}
Next we consider for each $k\in\NN$, some $\ell_k>0$ and $v_k\in 
C^{2}(\ell_k  I\times \gamma  I;\R^3)$ satisfying  $\nabla v_k(y)=(0,\bar b_0(y_1))$ nearby $\{|y_1|=\ell_k/2\}$, and \eqref{envkcritA=B}. 

Let $h_n\to0$ be an arbitrary sequence. Here again we can choose $\eps_n:=h_n/\gamma$. 
For each $k\in\NN$ and $n\in\NN$ such that  $\eps_n\ell_k<\delta_k$, we define 
for $x\in \Omega$, 
$$u_{n,k}(x):=\begin{cases}
\ds \eps_n v_k\bigg(\frac{d_k(x')}{\eps_n},\gamma x_3\bigg) &\ds  \text{if $|d_k(x')|<\frac{\ell_k\eps_n}{2}$}\,,\\[10pt]
\ds \eps_n v_k\left(\frac{\ell_k}{2},\gamma x_3\right) & \ds \text{if $d_k(x')\geq \frac{\ell_k\eps_n}{2}$}\,,\\[10pt]
\ds \eps_n v_k\left(-\frac{\ell_k}{2},\gamma x_3\right) & \ds \text{if $d_k(x')\leq- \frac{\ell_k\eps_n}{2}$}\,.
\end{cases}$$
Then $u_{n,k}\in H^2(\Omega;\R^3)$, and 
\begin{equation}\label{structgradApr=Bprcrit}
\nabla_{h_n} u_n(x)=\begin{cases}
\ds \left(\partial_1v_k\bigg(\frac{d_k(x')}{\eps_n},\gamma x_3\bigg)\otimes\nabla d_k(x'), \partial_2 v_k\bigg(\frac{d_k(x')}{\eps_n},\gamma x_3\bigg) \right) 
&\ds \text{if $|d_k(x')|<\frac{\ell_k\eps_n}{2}$}\,,\\[10pt]
(0,b_k(x)) & \text{otherwise} \,.
\end{cases}
\end{equation}
From the boundedness of $v_k$ and  $\nabla v_k$  together with the smoothness of $d_k$ in $U_k$, we infer that $u_{n,k}\to 0$ in $W^{1,p}(\Omega;\R^3)$ and 
$\frac{1}{h_n}\partial_3 u_{n,k}\to b_k $ in $L^p(\Omega;\R^3)$ as $n\to\infty$. 
Now it remains to estimate $F_{\eps_n}^{h_n}(u_{n,k})$. First of all, \eqref{structgradApr=Bprcrit}  yields 
$F_{\eps_n}^{h_n}\big(u_{n,k}, \Omega\setminus\{|d_k(x')|<\ell_k\eps_n/2\}\big)=0$. 
Using the fact that $|\nabla d_k|=1$ $\mathcal{L}^2$-a.e. in $\R^2$, we infer from $(H_5)$ that for $x\in \Omega\cap \{|d_k(x')|<\ell_k\eps_n/2\}$, 
$$W\big(\nabla_{h_n}u_{n,k}(x)\big)=V\big(\big|\partial_1 v_k(d_k(x')/\eps_n,\gamma x_3)\big|, \partial_2 v_k(d_k(x')/\eps_n,\gamma x_3)\big)
=\mathcal{W}\big(\nabla v_k(d_{k}(x')/\eps_n,\gamma x_3)\big)\,. $$
Next we compute for $x\in \Omega\cap \{|d_k(x')|<\ell_k\eps_n/2\}$,
\begin{multline*}
\big|\nabla_{h_n}^2u_{n,k}(x)\big|^2=\frac{1}{\eps_n^2}\big|\nabla^2v_k(d_k(x')/\eps_n,\gamma x_3)\big|^2 +
\big|\partial_1 v_k(d_k(x')/\eps_n,\gamma x_3)\big|^2|\nabla^2d_k(x')|^2  \\
+\frac{2}{\eps_n}\bigg(
\partial_1v_k(d_k(x')/\eps_n,\gamma x_3)\cdot \partial^2_1v_k(d_k(x')/\eps_n,\gamma x_3)\bigg)
\bigg(\nabla^2d_k(x')\cdot \big(\nabla d_k(x')\otimes\nabla d_k(x')\big)\bigg)\,,
\end{multline*}
which yields 
$$\left|\nabla_{h_n}^2u_{n,k}(x)\right|^2\leq \frac{1+\eps_n}{\eps_n^2} \big|\nabla^2v_k(d_k(x')/\eps_n,\gamma x_3)\big|^2 +\frac{C_k}{\eps_n}\big|\partial_1 v_k(d_k(x')/\eps_n,\gamma x_3)\big|^2\,,$$
for $x\in \Omega\cap \{|d_k(x')|<\ell_k\eps_n/2\}$ and some constant $C_k$ independent of $n$. 
Therefore, 
\begin{equation}\label{estisupFunk}
F_{\eps_n}^{h_n}(u_{n,k})=F_{\eps_n}^{h_n}\big(u_{n,k},\Omega\cap \{|d_k(x')|<\ell_k\eps_n/2\}\big) 
 \leq I^k_n+II^k_n\,,
 \end{equation}
 with 
$$
I^k_n:= \frac{1}{\eps_n}\int_{\Omega\cap \{|d_k|<\ell_k\eps_n/2\}} \mathcal{W}\big(\nabla v_k(d_k(x')/\eps_n,\gamma x_3)\big)+
\left|\nabla^2v_k\big(d_k(x')/\eps_n,\gamma x_3\big)\right|^2dx\,,$$
 and 
 $$II^k_n:= \int_{\Omega\cap \{|d_k|<\ell_k\eps_n/2\}}  \left|\nabla^2v_k\big(d_k(x')/\eps_n,\gamma x_3\big)\right|^2+C_k \left|\partial_1 v_k\big(d_k(x')/\eps_n,\gamma x_3\big)\right|^2\,dx\,.$$
Using Fubini's theorem, the Coarea Formula, the fact that $|\nabla d_k|=1$, and changing variables  we estimate 
\begin{align*}
I^k_n=&\frac{1}{\eps_n}\int_I\bigg(\int_{\omega\cap\{|d_k|<\ell_k\eps_n/2\} }\mathcal{W}\big(\nabla v_k(d_k(x')/\eps_n,\gamma x_3)\big) +\big|\nabla^2v_k(d_k(x')/\eps_n,\gamma x_3)\big|^2 dx'\bigg)dx_3\\
=&\frac{1}{\eps_n}\int_I \bigg(\int_{\ell_k\eps_n I}\big(\mathcal{W}\big(\nabla v_k(t/\eps_n,\gamma x_3)\big)
+ \big|\nabla^2v_k(t/\eps_n,\gamma x_3)\big|^2\big)\mathcal{H}^1(\mathcal{M}^k_t\cap\omega)\,dt \bigg) dx_3\\
=&\frac{1}{\eps_n}\int_{\ell_k\eps_nI\times I}\left(\mathcal{W}\big(\nabla v_k(t/\eps_n,\gamma x_3)\big)+
\big|\nabla^2v_k(t/\eps_n,\gamma x_3)\big|^2\right)\mathcal{H}^1(\mathcal{M}^k_t\cap\omega) \,dtdx_3\\
\leq&\frac{1}{\gamma}\int_{\ell_k I \times \gamma I}\left(\mathcal{W}\big(\nabla v_k(y)\big)+
\big|\nabla^2v_k(y)\big|^2\right)\mathcal{H}^1(\mathcal{M}^k_{\eps_n y_1}\cap \overline\omega)\,dy\,.
\end{align*}
Then  Fatou's lemma, \eqref{limlenglevs}, and \eqref{envkcritA=B} yield 
\begin{equation}\label{estilimsupImk}
\limsup_{n\to\infty} I_n^k\leq (K_\gamma+2^{-k})\mathcal{H}^1(\mathcal{M}^k\cap\ol{\omega}) \,. 
\end{equation}
Arguing in the same way we infer that  
\begin{equation}\label{estilimsupIImk}
\lim_{n\to\infty} II^k_n= \lim_{n\to+\infty}\frac{\eps_n}{\gamma}\int_{\ell_k I \times \gamma I} \big( \big|\nabla^2v_k(y)\big|^2
+C_k \big|\partial_1 v_k(y)\big|^2\big)\mathcal{H}^1(\mathcal{M}_{\eps_n y_1}\cap\overline\omega )\,dy=0\,.
\end{equation}
Gathering \eqref{estisupFunk},  \eqref{estilimsupImk} and \eqref{estilimsupIImk},  we derive 
$$\limsup_{k\to\infty} \limsup_{n\to\infty}\, F_{\eps_n}^{h_n}(u_{n,k})\leq K_\gamma \Per_{\omega}(E)\,.$$
Since $\lim_k\lim_n\|u_{n,k}\|_{W^{1,p}(\Omega)}=0$, and $\lim_k\lim_n\|\frac{1}{h_n}\partial_3u_{n,k}-b\|_{L^{p}(\Omega)}=\lim_k\|b_k-b\|_{L^p(\Omega)}=0$, 
the conclusion follows for a suitable diagonal sequence $u_n:=u_{n,k_n}$ as in Step 1. 
\end{proof}


%
%


\section{$\Gamma$-convergence in the subcritical regime}\label{Sectsub}

This section is devoted to the proof of Theorem \ref{thm:gammalim_gammasub}. The $\Gamma$-liminf inequality is obtained 
through a slicing argument, and by 
establishing a lower asymptotic inequality for a reduced 2D functional (see Proposition~\ref{gliminfrefsubcrit}) much in the spirit 
of Section~\ref{gamliminfcrit}. 
The $\Gamma$-$\liminf$ and $\Gamma$-$\limsup$  inequalities  are stated in 
Theorem~\ref{gamliminfsubcrit} and Theorem~\ref{thm:RSsubcrit'} respectively, and Corollary~\ref{prop:charKsubcrit} shows 
that lower and upper inequalities~agree.

%
%

\subsection[]{The $\Gamma$-$\liminf$ inequality}

For a bounded open set $A\subset \R^2$ and $\eps>0$, we introduce the localized functional 
$F^{0}_\eps(\cdot,\cdot,A)$  defined 
for a pair $(u,b)\in H^{2}(A;\R^3)\times H^1(A;\R^3)$ by 
\begin{equation}\label{defF0}
F^{0}_\eps(u,b,A)
 : =  \int_{A}  \frac{1}{\eps}\,W(\nabla'u,b)+\eps \big( |(\nabla')^2 u|^2+2|\nabla' b|^2\big)\,dx'\,.
\end{equation}
Then we consider  the constant
\begin{multline}\label{eq:defKstar2D}
K_0^\star:= \inf \bigg\{\liminf_{n\to\infty}\, F^{0}_{\eps_n}(u_n,b_n,Q') \;:\; \eps_n \to 0^+\,,\, \{(u_n,b_n)\} \subset H^{2}(Q';\R^3)\times H^1(Q';\R^3)\,,\\
  (u_n,b_n) \to (u_0,b_0) \text{ in } [L^1(Q';\R^3)]^2  \bigg\}\,.
\end{multline}
Here again the  constant $K^\star_0$ is finite, as one may 
check by considering an admissible sequence $\{(u_n,b_n)\}$ made of suitable (standard) regularizations of $u_0$ and $b_0$. 
As in the previous section, we first  provide a lower bound in terms  of $K^\star_0$ for the lower $\Gamma$-limit of the family $\{F^{0}_\eps\}$ 
in case of an elementary  jump~set.

\begin{proposition}\label{gliminfrefsubcrit}
Assume that assumptions $(H_1)$, $(H_2)$ and \eqref{orientationwells} hold. Let  $\eps_n\to0^+$ be an arbitrary sequence.   
Let $\rho>0$ and $\alpha \in\R$, let $J\subset\R$ be a bounded open interval, and consider the cylinder $U':=(\alpha-\rho,\alpha+\rho)\times J$.  
Let $(u,b)\in W^{1,\infty}(U';\R^3)\times L^\infty(U';\R^3)$ satisfying  \eqref{eq:u01}.  
Then for any sequence $\{(u_n,b_n)\}\subset H^2(U';\R^3)\times  H^1(U';\R^3)$ such that $(u_n,b_n)\to (u,b)$ in $[L^1(U';\R^3)]^2$, we have
$$\liminf_{n\to \infty}\,F^{0}_{\eps_n}(u_n,b_n,U')\geq K^\star_0\mathcal{H}^1(J)\,. $$ 
\end{proposition}

\begin{proof} Here the proof  closely follows the one of Proposition \ref{gliminfref}. The arguments are essentially the same with the obvious modifications 
once we consider 
\begin{multline*}
{\cal E}_{0}(J,\rho):=\inf \bigg\{\liminf_{n\to\infty}\, F^{0}_{\eps_n}(u_n,b_n,J'_\rho) \;:\; \eps_n \to 0^+\,,\, \{(u_n, b_n)\} \subset H^{2}(J'_\rho;\R^3)\times H^1(J'_\rho;\R^3)\,, \\
 (u_n , b_n) \to (u_0, b_0) \text{ in } [L^1(J'_\rho;\R^3)]^2 \bigg\}
\end{multline*}
in place of ${\cal E}_{\gamma}(J,\rho)$ with $J\subset \R$ a bounded open set, $\rho>0$, and $J'_{\rho} := \rho I \times J $.  Then 
one  proves the analogue of Lemma~\ref{lem:scales}, in particular that ${\cal E}_{0}(J,\rho)=K_0^\star  \mathcal{H}^{1}(J) $. 
We omit any further details. 
\end{proof}

\begin{remark}\label{cor:concentrationsubcrit}
As in Corollary \ref{cor:concentration}, the energy of optimal sequences for 
${\cal E}_0(J,\rho)$ is concentrated near the limiting interface, {\it i.e.}, given $0<\delta<\rho$, for any sequences  
$\eps_n\to 0^+$ and $\{(u_n,b_n)\}\subset H^{2}(J'_\rho;\R^3)\times H^{1}(J'_\rho;\R^3)$ such that   
 $(u_n,b_n) \to (u_0,b_0)$ in $[L^1(J'_\rho;\R^3)]^2$ and $\lim_{n} F^{0}_{\eps_n}(u_n,b_n,J'_\rho)={\cal E}_0(J,\rho)$, we have 
$\lim_{n} F^{0}_{\eps_n}(u_n,b_n,J'_\rho\setminus J'_{\delta}) = 0$. 
\end{remark}

We now prove the lower inequality for the $\Gamma$-$\liminf$ of $\{F^{h}_{\eps}\}$ essentially as in Theorem~\ref{thm:gammaliminf} together with a slicing argument 
involving the functionals $\{F^{0}_{\eps}\}$.

\begin{theorem}\label{gamliminfsubcrit}
Assume that assumptions $(H_1)-(H_2)$, $(H_5)$ and \eqref{orientationwells} hold.   
Let $h_n\to0^+$ and $\eps_n\to0^+$ be arbitrary sequences such that $h_n/\eps_n \to 0$.  
Then, for any $(u,b)\in \mathscr{C}$ and any sequence $\{u_n\}\subset H^2(\Omega;\R^3)$ such that $(u_n,\frac{1}{h_n}\partial_3 u_n) \to (u,b)$ in 
$[L^1(\Omega;\R^3)]^2$, we have  
$$
\liminf_{n\to\infty}\, F^{h_n}_{\eps_n}(u_n) \geq K_0^{\star}\, {\rm Per}_{\omega}(E)\,,
$$
where $(\nabla' u,b) (x)= \bigl(1-\chi_E(x')\bigr)A + \chi_E(x') B\,$.
\end{theorem}

\begin{proof}{\it Step 1.} First we may assume that 
$\liminf_{n} F^{h_n}_{\eps_n}(u_n) =\lim_{n} F^{h_n}_{\eps_n}(u_n) <\infty$. We set $b_n:=\frac{1}{h_n}\partial_3 u_n \in H^1(\Omega;\R^3)$. 
It is well known that for $\mathcal{L}^1$-a.e. $x_3\in I$ the slices $u_n^{x_3}(x'):=u_n(x',x_3)$ and $b^{x_3}_n(x'):=b_n(x',x_3)$ belong 
to $H^2(\omega;\mathbb{R}^3)$ and $H^1(\omega;\mathbb{R}^3)$ respectively, and horizontal weak derivatives coincide $\mathcal{L}^3$-a.e. in $\Omega$ 
(see {\it e.g.} \cite[p. 204]{AFP}). 
Moreover, up to a subsequence, $(u_n^{x_3},b_n^{x_3})\to (u,b)$ in~$[L^1(\omega;\R^3)]^2$ for $\mathcal{L}^1$-a.e. $x_3\in I$.  
Hence, using Fubini's theorem we can estimate
$$F_{\eps_n}^{h_n}(u_n)=
\int_I \left( \int_{\omega\times\{x_3\}} \frac{1}{\eps_n}\,W(\nabla_{h_n} u_n) + \eps_n |\nabla_{h_n}^2 u_n|^2 \, d\mathcal{H}^2\right)\,dx_3
\geq \int_{I} F^0_{\eps_n}\big(u^{x_3}_n,b_n^{x_3},\omega\big)\,dx_3\,, $$
and then infer from Fatou's lemma that 
$$\lim_{n\to\infty} F_{\eps_n}^{h_n}(u_n) \geq \int_{I} \liminf_{n\to+\infty} F^0_{\eps_n}\big(u_n^{x_3},b_n^{x_3},\omega\big)\,dx_3\,.$$
Now it remains to prove that for $\mathcal{L}^1$-a.e. $x_3\in I$,  
\begin{equation}\label{liminffatou}
\liminf_{n\to\infty} F^0_{\eps_n}\big(u_n^{x_3},b_n^{x_3},\omega\big)\geq K_0^\star \, {\rm Per}_{\omega}(E)\,.
\end{equation}
The next steps are devoted to the proof of \eqref{liminffatou}. 
\vskip5pt

\noindent{\it Step 2.} First assume that $A^\prime\not= B^\prime$. We obtain estimate \eqref{liminffatou} by applying Proposition \ref{gliminfrefsubcrit} together with the covering argument used in the proof of Theorem \ref{thm:gammaliminf}, Step 1. 
Further details are left to the~reader.  
\vskip5pt

\noindent{\it Step 3.} We now consider the case $A^\prime=B^\prime$ ($=0$ by \eqref{orientationwells}), and  
we may assume  that   $u\equiv 0$. 
Then consider an arbitrary sequence $\{(u_n,b_n)\} \subset H^{2}(\omega;\R^3) \times H^{1}(\omega;\R^3)$ satisfying $(u_n,b_n) \to (0,b)$ in~$[L^1(\omega;\R^3)]^2$.
We may also assume that $\liminf_{n} F^{0}_{\eps_n}(u_n,b_n)  = \lim_{n} F^{0}_{\eps_n}(u_n,b_n)< \infty$. By Theorem~\ref{thm:BJ} we have  
 $b(x')=\bigl(1-\chi_E(x')\bigr)A_3 + \chi_E(x') B_3$ for a set $E\subset \omega$ of finite perimeter in $\omega$. We prove the announced result 
following the blow-up argument in the proof of Theorem~\ref{thm:gammaliminf}, Step 3. We introduce 
 the finite nonnegative Radon measure $\mu_n$ on $\omega$ given by 
$$
\mu_n := \left( \frac{1}{\eps_n} W\left(\nabla' u_n,  b_n\right)
+ \eps_n \left(\left|(\nabla')^2 u_n \right|^2 + 2\left|\nabla' b_n \right|^2\right) \right) \Leb{2} \restr{} \omega\,. 
$$
Then $\mu_n(\omega)=F^{0}_{\eps_n}(u_n,b_n) $,  $\sup_n \mu_n(\omega) < \infty$, and there is a subsequence (not relabeled) such that
$\mu_n \rightharpoonup \mu $ weakly* in the sense of measures for some  finite nonnegative Radon measure $\mu$ on $\omega$. 
By lower semicontinuity we have $\mu(\omega) \leq  \lim_{n} F^{0}_{\eps_n}(u_n,b_n)$,  
and we have to prove that $\mu(\omega) \geq K^{\star}_0 \mathcal{H}^{1}(\partial^\ast E\cap\omega)$. This estimate can be achieved as  
in the proof of Theorem~\ref{thm:gammaliminf}, Step 3, with minor modifications. 
\end{proof}

\begin{remark}\label{liminf2D}
Let $\eps_n\to 0^+$ be an arbitrary sequence. By the arguments above, for any $(u,b)\in\mathscr{C}$ and any sequence  $\{(u_n,b_n)\}
\subset H^2(\omega;\R^3)\times H^1(\omega;\R^3)$ satisfying $(u_n,b_n)\to (u,b)$ in $[L^1(\omega;\R^3)]^2$, we have 
$\liminf_n F^0_{\eps_n}(u_n,b_n,\omega)\geq  K_0^{\star}\, {\rm Per}_{\omega}(E)$ where $(\nabla' u,b)= \bigl(1-\chi_E\bigr)A + \chi_E B\,$.
\end{remark}

%
%

\subsection{Lower bound on $K^\star_0$}

As in Proposition~\ref{rempot}, 
we now prove that 
sequences realizing $K^\star_0$ can be  prescribed near  
the two sides $\{x_1=\pm \frac{1}{2}\}$, and chosen to be independent of the $x_2$-variable.

\begin{proposition}\label{prop:matchingsubcrit}
Assume that $(H_1)-(H_4)$ and \eqref{orientationwells} hold. Then 
there exist sequences $\eps_n\to 0^+$, $\{c_n\} \subset \R^3$, and $\{(g_n,d_n)\} \subset C^{2}(Q';\R^3)\times C^{1}(Q';\R^3)$ 
such that  $(g_n,d_n)$ is independent of~$x_2$ (i.e., $g_n(x')=:\bar g_n(x_1)$ and $d_n(x')=:\bar d_n(x_1)$), $c_n \to 0$, $g_n \to u_0$ in $W^{1,p}(Q';\R^3)$, $d_n\to b_0$ in $L^p(Q';\R^3)$,  
$$(g_n,d_n)= (u_0,b_0) \;\text{ in } Q'\cap\{x_1 >  1/4\}\,,\quad (g_n,d_n)=(u_0+c_n,b_0) \;\text{ in } Q'\cap\{x_1 < -1/4\}\,,$$  
and $\lim_{n} F^{0}_{\eps_n}(g_n,d_n,Q') = K_0^{\star}$. 
\end{proposition}

\begin{proof}{\it Step 1.} 
Consider sequences $\eps_n\to0^+$ and 
$\{(u_n,b_n)\} \subset H^{2}(Q';\R^3)\times H^1(Q';\R^3)$ such that 
$(u_n,b_n) \to  (u_0,  b_0)$ in $[L^1(Q';\R^3)]^2$, and
$\lim_{n} F^{0}_{\eps_n}(u_n,b_n,Q') = K^{\star}_0$.  
Arguing as in the proof of Proposition \ref{prop:matching1}, we may assume 
 that $(u_n,b_n)\in C^{2}(Q';\R^3)\times C^{1}(Q';\R^3)$, and that $(u_n,b_n)$ is independent of 
$x_2$, {\it i.e.}, $(u_n,b_n)(x)=:(\bar u_n(x_1),\bar b_n(x_1))$. 
Moreover, the arguments used in the proof of Theorem~\ref{thm:compactness} (with minor modifications) yield   
$u_n \to  u_0$ in $W^{1,p}(Q';\R^3)$, and $b_n \to  b_0$ in $L^p(Q';\R^3)$. 
\vskip5pt

\noindent{\it Step 2.} 
Here again we consider  a partition of $\bigl(\frac{1}{12},\frac{1}{6}\bigr)$ into $M_n := \bigl[ \frac{1}{\eps_n}\bigr]$ intervals of length $\frac{1}{12M_n}$.   
By Remark~\ref{cor:concentrationsubcrit}, the energy concentrates near the interface $\{x_1=0\}$, and we can find a suitable interval 
$I_n := \bigl(\theta_n - \frac{1}{12M_n},\theta_n\bigr)  \subset \bigl(\frac{1}{12},\frac{1}{6}\bigr) $ for which
\begin{multline}\label{eq:layersub}
M_n \left( \int_{I_n} |\bar u_n- \bar u_0|^p + |\bar b_n- \bar b_0|^p + |\bar u'_n -\bar  u'_0|^p dx_1 + F^{0}_{\eps_n}(u_n,b_n,I_n\times I)\right)  \\
  \leq \int_{\bigl(\frac{1}{12},\frac{1}{6}\bigr)}  |\bar u_n- \bar u_0|^p + |\bar b_n- \bar b_0|^p+ |\bar u'_n -\bar u'_0|^p dx_1
   + F^{0}_{\eps_n}(u_n,b_n,\bigl({\ts \frac{1}{12},\frac{1}{6}}\bigr) \times I)  =: \alpha_n \to 0\,.
\end{multline}
We  select a level $t_n \in \bigl(\theta_n - \frac{1}{12M_n},\theta_n\bigr)$ satisfying
\begin{multline}\label{eq:tnsub}
 |\bar u_n(t_n)- \bar u_0(t_n)|^p + |\bar b_n(t_n)- \bar b_0(t_n)|^p + |\bar u'_n(t_n) -\bar u'_0(t_n)|^p \\
+ \frac{1}{\eps_n}W\big(\bar u'_n(t_n),0,\bar b_n(t_n)\big)+\eps_n|\bar u^{\prime\prime}_n(t_n)|^2+2\eps_n|\bar b'_n(t_n)|^2\leq 12\alpha_n\,.
\end{multline}
Let $\varphi_n\in C^\infty(\R)$ be  a cut-off function as in \eqref{glutest}. 
For $x_1 \in I_n$ we  set  
$$
v_n(x_1) := \bigl( 1-\varphi_n(x_1)\bigr) \bigl(\bar u_0(x_1)+ c^+_n\bigr) + \varphi_n(x_1) \bar u_n(x_1) \,,
$$
with $c^+_n:= \bar u_n(t_n) - \bar u_0(t_n)\to 0$, and 
$$\zeta_n(x_1):= \bigl( 1-\varphi_n(x_1)\bigr) \bar b_0(x_1) + \varphi_n(x_1) \bar b_n(x_1) \,,$$
We claim that
\begin{align}
\label{grplim1sub} &\int_{I_n} |v_n -  \bar u_0|^p  dx_1 \to 0\,,\\
\label{grplim2sub} & \frac{1}{\eps_n} \int_{I_n} \big|\zeta_n -  \bar b_0\big|^p dx_1 \to 0\,,\\
\label{grplim3sub} & \frac{1}{\eps_n} \int_{I_n} |v'_n - \bar u'_0|^p dx_1 \to 0\,,\\
\label{grplim4sub} & \frac{1}{\eps_n} \int_{I_n} W(v'_n,0,\zeta_n) \,dx_1 \to 0\,,\\
\label{grplim5sub}  & \eps_n \int_{I_n}  \left|v^{\prime\prime}_n \right|^2+2|\zeta'_n|^2 \,dx_1 \to 0\,.
\end{align}
Estimates \eqref{grplim1sub} and \eqref{grplim2sub}  come straightforward from  \eqref{eq:layersub}.
We  apply Poincar\'e's inequality to obtain
\begin{equation}\label{eq:Poincaresub}
\int_{I_n} |\bar u_n - \bar u_0 -c^+_n |^p dx_1 \leq C \left(\frac{1}{M_n}\right)^p \int_{I_n} | \bar u'_n - \bar u'_0|^p\,  dx_1\leq C\alpha_n \eps^{p+1}_n\,, 
\end{equation}
and using \eqref{eq:layersub},  \eqref{glutest}, and \eqref{eq:Poincaresub}, we derive 
\begin{align*}
\frac{1}{\eps_n} \int_{I_n} |v'_n - \bar u'_0|^p\,  dx_1    \leq  \frac{C}{\eps_n} \int_{I_n} \bigg(\frac{1}{\eps_n^p}| \bar u_n - \bar u_0-c^+_n|^p +  
         |\bar u'_n - \bar  u'_0|^p \bigg)\, dx_1  \leq C \alpha_n \to 0\,,
  \end{align*}
so that \eqref{grplim3sub} is proved. Now \eqref{grplim4sub} follows from \eqref{grplim2sub} and  \eqref{grplim3sub} exactly as \eqref{estimatchdist}. 
Finally we obtain \eqref{grplim5sub} arguing as in the proof of Proposition \ref{prop:matching1} with minor modifications. We omit further details. 
\vskip5pt
 
\noindent{\it Step 3.}  We conclude as in  the proof of Proposition~\ref{prop:matching1}, Step 4. We first define a sequence $(g_n^+,d_n^+)$ by setting for $x'\in Q'$, 
\begin{equation}\label{defgn+sub}
(g^+_n,d_n^+)(x'):=\begin{cases}
\big(\bar u_n(x_1),\bar b_n(x_1)\big) & \text{for $x_1<\theta_n - \frac{1}{12M_n}$}\,,\\[5pt]
\big(v_n(x_1),\zeta_n(x_1)\big) & \text{for $\theta_n - \frac{1}{12M_n}\leq x_1< \theta_n$}\,, \\[5pt]
\big(\bar u_0(x_1) +c^+_n,\bar b_0(x_1)\big)  & \text{for $\theta_n\leq x_1\leq \frac{1}{2}$}\,.
\end{cases}
\end{equation}
Then we repeat the procedure above to modify $g_n^+$ in $(-\frac{1}{2},0)\times I$. Again, we omit further details. 

\end{proof}

\begin{corollary}\label{prop:charKsubcrit}
Assume that $(H_1)-(H_4)$  and \eqref{orientationwells} hold.
Then $K^\star_0 \geq K_0$. 
\end{corollary}

\begin{proof}
Consider the sequences $\{\eps_n\}$ and $\{(g_n,d_n)\}$ given by Proposition \ref{prop:matchingsubcrit}. 
We set $\ell_n:= \eps_n/2$, and for $t\in[- \ell_n,\ell_n]$, $\phi_n(t):=(\phi_{1,n},\phi_{2,n})(t):=\big(\bar g'_n(t/\eps_n),\bar d_n(t/\eps_n)\big)$. 
Then straightforward computations yield $\phi_n=(\bar u'_0,\bar b_0)$ nearby $\{|t|=\ell_n\}$, 
and 
$$\int_{-\ell_n}^{\ell_n}\mathcal{W}(\phi_{1,n}(t),\phi_{2,n}(t))+|\phi'_{1,n}(t)|^2+2|\phi'_{2,n}(t)|^2\,dt= F_{\eps_n}^{0}(g_n,d_n,Q')\,.$$ 
By definition of $K_0$ and the construction of $\{(g_n,d_n)\}$ we have $K_0\leq  
F_{\eps_n}^{0}(g_n,d_n,Q')\to K_0^\star$ as $n\to\infty$, 
and the proof is complete. 
\end{proof}

%
%

\subsection[]{The $\Gamma$-$\limsup$ inequality}

We now complete the proof of Theorem \ref{thm:gammalim_gammasub} with the construction of recovery sequences. 

\begin{theorem}\label{thm:RSsubcrit'}
Assume that $(H_1)-(H_5)$ and \eqref{orientationwells} hold. Let $\eps_n\to 0^+$ and $h_n\to 0^+$ be arbitrary sequences such that $h_n/\eps_n \to 0$. Then, for every $(u,b) \in \mathscr{C}$,  
there exists a sequence $\{u_n\} \subset H^{2}(\Omega;\R^3)$ such that $u_n \to u$ in $W^{1,p}(\Omega;\R^3)$, $\frac{1}{h_n}\partial_3 u_n \to b$ in $L^p(\Omega;\R^3)$,  
and 
$$
\lim_{n\to\infty} F^{h_n}_{\eps_n}(u_n) = K_{0} \,\Per_{\omega}(E) \,, 
$$
where $(\nabla' u,b) (x)= \bigl(1-\chi_E(x')\bigr)A + \chi_E(x') B\,$.
\end{theorem}

\begin{proof} The proof parallels the one of Theorem \ref{thm:RS_A'neqB'}, and we shall refer to it for the notation. 
\vskip3pt

\noindent{\it Step 1.} We first assume that $A'\not=B'$, and that $\partial^\ast E\cap \omega$ is made by finitely many interfaces. 
We also assume that the pair $(u,b)$ is given by $(u,b)(x)=(\bar u,\bar b)(x_1) $ as  in the proof  of 
Theorem~\ref{thm:RS_A'neqB'}, Step 1. 

For $k\in\NN$ arbitrary, we choose some $\ell_k>0$ and $(\phi_{1,k},\phi_{2,k}):\R\to\R^{3\times 2}$ of class $C^1$ such that 
 $(\phi_{1,k},\phi_{2,k})=(\bar u'_0,\bar b_0)$ in $\{|t|\geq\ell_k/2\}$, and 
\begin{equation}\label{def_K0_sub}
\int_{-\ell_k/2}^{\ell_k/2}\mathcal{W}(\phi_{1,k}(t),\phi_{2,k}(t))+|\phi_{1,k}'(t)|^2+2|\phi_{2,k}'(t)|^2\,dt\leq K_0 + 2^{-k}\,. 
\end{equation} 
Without loss of generality, we may assume that $\phi_{2,k}'$ is Lipschitz continuous. 
In the remaining of this step we shall drop the subscript $k$ for simplicity. 
For each $i=1,\ldots,m$ we fix some bounded open interval $J'_i\subset \R$ satisfying \eqref{intervlimsup}, and we consider for $n$ large enough the 
coefficients $\{\alpha^n_{i\pm}\}$ as in~\eqref{notgamsup}, and such that \eqref{bandeconvex} holds.

Let $h_n\to 0^+$ and $\eps_n\to 0^+$ be arbitrary sequences such that $h_n/\eps_n\to 0$,  and define  
\begin{equation}\label{Phi_def}
\Phi(t):=\int_{-\ell/2}^t\phi_{1}(s)\,ds-c\qquad\text{with }\;c:=\frac{1}{2}\int_{-\ell/2}^{\ell/2}\phi_{1}(s)\,ds\,.
\end{equation}
We  set  for $i=1,\ldots,m$
and  $x\in (\alpha^n_{i-},\alpha^n_{i+})\times\R\times I$, 
$$w^i_{n}(x):=(-1)^{i+1}\Phi\left((-1)^{i+1}\frac{x_1-\alpha_{i}}{\eps_n}\right)+
\frac{h_n}{\eps_n}x_3 \phi_{2}\left((-1)^{i+1}\frac{x_1-\alpha_{i}}{\eps_n}\right)
+\left(1+(-1)^{i}\right)c\,.$$
Then we have
\begin{equation}\label{bdcondlayersub1}
w_{n}^i(\alpha^n_{i-},x_2,x_3)=\begin{cases}
\ds \Phi(-\ell/2) +\frac{h_n}{\eps_n}x_3\phi_{2}(-\ell/2) & \text{if $i$ is odd}\,,\\[6pt]
\ds \Phi(\ell/2)+\frac{h_n}{\eps_n}x_3\phi_{2}(\ell/2)  & \text{if $i$ is even}\,,
\end{cases}
\end{equation}
and
\begin{equation}\label{bdcondlayersub2}
w_{n}^i(\alpha^n_{i+},x_2,x_3)=\begin{cases}
\ds\Phi(\ell/2)+\frac{h_n}{\eps_n}x_3\phi_{2}(\ell/2) & \text{if $i$ is odd}\,,\\[6pt]
\ds\Phi(-\ell/2)+\frac{h_n}{\eps_n}x_3\phi_{2}(-\ell/2)+4c & \text{if $i$ is even}\,.
\end{cases}
\end{equation}
Setting $\beta_i^n$ and $\kappa_i$ as in \eqref{notgamsupbis}, we define for $x\in \Omega$, 
$$u_{n}(x):= \begin{cases}
\ol u(x_1) +\eps_n\Phi(\frac{-\ell}{2})+h_nx_3\phi_{2}\big(\frac{-\ell}{2}\big) & \text{for } x_1\leq \alpha^n_{1-}\,,\\[8pt]  
\ol u(\alpha^n_{i-})-\beta^n_{i-1}+\eps_nw^i_{n}(x)+\eps_n\kappa_{i-1}c 
& \text{for } \alpha^n_{i-}< x_1< \alpha^n_{i+}\,,\\[8pt]
 \ol u(x_1)-\beta^n_{i} +\eps_n\Phi\big(\frac{(-1)^{i+1}\ell}{2}\big)+h_nx_3\phi_{2}\big(\frac{(-1)^{i+1}\ell}{2}\big)+\eps_n\kappa_{i}c&\text{for } \alpha^n_{i+}\leq x_1\leq \alpha^n_{(i+1)-}\,,\\[8pt]
\ol u(x_1)-\beta^n_{m} +\eps_n\Phi\big(\frac{(-1)^{m+1}\ell}{2}\big) +h_n x_3\phi_{2}\big(\frac{(-1)^{m+1}\ell}{2}\big)+\eps_n\kappa_{m}c&\text{for } x_1\geq \alpha^n_{m+}\,.
\end{cases}$$
In view of \eqref{bdcondlayersub1}-\eqref{bdcondlayersub2}, and since $\phi_2'(\pm\ell/2)=0$, we have $u_{n}\in H^2(\Omega;\R^3)$. Moreover $u_n$ does not depend 
on $x_2$, 
$$\partial_1 u_{n}(x)=\begin{cases}
\ds \phi_{1} \left((-1)^{i+1}\frac{x_1-\alpha_{i}}{\eps_n}\right)+ (-1)^{i+1}\frac{h_n}{\eps_n}x_3\phi'_{2}\left((-1)^{i+1}\frac{x_1-\alpha_{i}}{\eps_n}\right) & \text{for }\alpha^n_{i-}< x_1< \alpha^n_{i+}\,,\\[8pt]
\bar u'(x_1) & \text{otherwise}\,,
\end{cases}$$
and
$$\frac{1}{h_n}\partial_3 u_{n}(x)=\begin{cases}
\ds \phi_{2}\left((-1)^{i+1}\frac{x_1-\alpha_{i}}{\eps_n}\right) & \text{for }\alpha^n_{i-}< x_1< \alpha^n_{i+}\,,\, i=1,\ldots,m\,,\\[8pt]
\bar b(x_1) & \text{otherwise}\,.
\end{cases}$$
Arguing as in the proof of Theorem~\ref{thm:RS_A'neqB'}, Step 1, 
we derive that $u_n\to u$ in $W^{1,p}(\Omega;\R^3)$, and $\frac{1}{h_n}\partial_3u_n\to b$ in $L^p(\Omega;\R^3)$. Then we estimate 
$$F_{\varepsilon_n}^{h_n}(u_n)\leq \sum_{i=1}^m F_{\varepsilon_n}^{h_n}\big(\eps_nw^i_{n},(\alpha^n_{i-},\alpha^n_{i+})\times J'_i\times I\big)\,.$$
Changing variables and using Fubini's theorem, we obtain  
\begin{multline*}
F_{\varepsilon_n}^{h_n}(\eps_nw^i_{n},(\alpha^n_{i-},\alpha^n_{i+})\times J'_i\times I) \\
=\mathcal{H}^1(J'_i)\int_{-\ell/2}^{\ell/2}\bigg(\int_{I}\mathcal{W}\big(\phi_{1}(t)
+(-1)^{i+1}\frac{h_n}{\varepsilon_n}x_3\phi'_{2}(t),\phi_{2}(t) \big)\,dx_3+|\phi'_1(t)|^2+ 2|\phi^{\prime}_{2}(t)|^2\bigg)\,dt \\ 
+\frac{\mathcal{H}^1(J'_i)h_n^2}{12\eps_n^2}\int_{-\ell/2}^{\ell/2} |\phi^{\prime\prime}_{2}(t)|^2\,dt \,.
\end{multline*}
Since $\mathcal{W}$ is continuous and $h_n/\eps_n\to 0$, we infer that 
$$\lim_{n\to\infty} F_{\varepsilon_n}^{h_n}(\varepsilon_n w^i_{n},(\alpha^n_{i-},\alpha^n_{i+})\times J'_i\times I)=\mathcal{H}^1(J'_i) 
\int_{-\ell/2}^{\ell/2}\mathcal{W}\left(\phi_{1}(t),\phi_{2}(t) \right)+|\phi'_1(t)|^2+ 2|\phi^{\prime}_{2}(t)|^2\,dt\,,$$
which leads to 
$$\limsup_{n\to\infty}\,F_{\varepsilon_n}^{h_n}(u_n)\leq  K_{0} \,\Per_{\omega}(E) +C_02^{-k}\,, $$
for a constant $C_0$ which only depends on $m$ and $\Per_{\omega}(E)$. Then the conclusion follows 
for a suitable diagonal sequence as already pursued in the proof of Theorem~\ref{thm:RS_A'neqB'}. 
\vskip5pt

\noindent{\it Step 2.}  In the case where $A'\not=B'$ and $\partial^\ast E\cap \omega$ is made by infinitely many interfaces, the proof follows from the previous step through a diagonalization argument as in  the proof of Theorem \ref{thm:RS_A'neqB'}. 
\vskip5pt

\noindent{\it Step 3.} We now consider the case $A'=B'\,(= 0)$, and  we proceed as in the proof of Theorem~\ref{thm:RS_A'neqB'} (we refer to it for the notation). 
We may assume that $u = 0$, and 
$b(x) = (1-\chi_E(x'))A_3 + \chi_E(x')B_3$ where $E \subset \omega$ has finite perimeter in $\omega$. We consider a sequence $\{E_k\}$ of smooth 
bounded subset of $\R^2$ such that $\chi_{E_k}\to\chi_E$ in $L^1(\omega)$, and $\lim_k \mathcal{H}^1(\partial E_k\cap \ol{\omega})=\Per_{\omega}(E)$. 
We define $b_k := (1-\chi_{E_k})A_3 + \chi_{E_k}B_3$, 
and the signed distance $d_k$ to $\mathcal{M}^k:=\partial E_k$ as in \eqref{signdist}. Here again we shall drop the subscript $k$ for simplicity.

For $k\in\NN$ arbitrary, we choose $\ell>0$ and $(\phi_{1},\phi_{2}):\R\to\R^{3\times 2}$ of class $C^1$ satisfying 
 $(\phi_{1},\phi_{2})(t)=(0,\bar b_0(t))$ nearby $\{|t|=\ell/2\}$ and \eqref{def_K0_sub}. We may also assume $\phi^\prime_{2}$ to be Lipschitz continuous.  
 Defining $\Phi$ as in \eqref{Phi_def}, we set for $x \in \Omega$,
$$
u_{n}(x)=\begin{cases}
	\ds \eps_n \Phi\bigg(\frac{d(x')}{\eps_n}\bigg)+h_nx_3\phi_2\bigg(\frac{d(x')}{\eps_n}\bigg) 
			& \ds \text{if }|d(x')|<\frac{\ell\eps_n}{2}\,,\\[10pt]
\ds	\eps_n \Phi\left(\frac{\ell}{2}\right)+h_nx_3\phi_2\left(\frac{\ell}{2}\right) 
			& \ds \text{if }d(x')\geq \frac{\ell\eps_n}{2}\,,\\[10pt]
\ds	\eps_n \Phi\left(-\frac{\ell}{2}\right)+h_nx_3\phi_2\left(-\frac{\ell}{2}\right) 
			&\ds  \text{if }d(x')\leq -\frac{\ell\eps_n}{2}\,.
\end{cases}
$$
Then $u_{n}\in H^2(\Omega;\R^3)$, and we compute 
$$
\nabla'u_n(x)=\begin{cases}
\ds \left(\phi_1\bigg(\frac{d(x')}{\eps_n}\bigg)+\frac{h_n}{\eps_n}x_3\phi'_2\bigg(\frac{d(x')}{\eps_n}\bigg)\right) \otimes\nabla d(x') 
& \ds  \text{if } |d(x')|<\frac{\ell\eps_n}{2}\,,\\[8pt]
0 & \text{otherwise} \,,
\end{cases}
$$
and
$$
\frac{1}{h_n}\partial_3 u_n(x)=\begin{cases}
\ds\phi_2\bigg(\frac{d(x')}{\eps_n}\bigg)
& \ds  \text{if $|d(x')|<\frac{\ell\eps_n}{2}$}\,,\\[8pt]
b(x) & \text{otherwise} \,.
\end{cases}
$$
Since $|\nabla d|=1$ $\mathcal{L}^2$-a.e. in $\R^2$, in the set $\{|d| < \ell \eps_n/2\}$ we have
\begin{multline*}
 \left|\nabla_{h_n}^2u_n\right|^2
	= \frac{1}{\eps_n^2}\big|\phi'_1(d/\eps_n)+\frac{h_n}{\eps_n}x_3\phi^{\prime\prime}_2(d/\eps_n)\big|^2+ 2 \big| \phi_2'(d/\eps_n)\big|^2 \\
	+\frac{1}{\eps_n}\bigg(\phi'_1(d/\eps_n)+\frac{h_n}{\eps_n}x_3\phi^{\prime\prime}_2(d/\eps_n)\big)
	\cdot\big(\phi_1\big(d/\eps_n\big)+\frac{h_n}{\eps_n}x_3\phi^{\prime}_2\big(d/\eps_n\big)\bigg)\big(\nabla^2 d\cdot (\nabla d\otimes\nabla d)\big)\\
	+ \big|\phi_1\big(d/\eps_n\big)+\frac{h_n}{\eps_n}x_3\phi^{\prime}_2\big(d/\eps_n\big)\big|^2|\nabla^2d|^2\,.
\end{multline*}
As in the proof of Theorem \ref{thm:RS_A'neqB'}, Step 3, we derive that $u_n\to 0$ in $W^{1,p}(\Omega;\R^3)$, and 
$\frac{1}{h_n}\partial_3 u_n\to b_k$ in~$L^p(\Omega;\R^3)$.  
Then, using the fact that $|\nabla d|=1$ and assumption $(H_5)$, we estimate 
\begin{equation}\label{estisupFun_sub}
F^{h_n}_{\eps_n}(u_n)
	= F^{h_n}_{\eps_n}\bigl(u_n, \Omega \cap \{ |d(x')| < \ell \eps_n/2\} \bigr)
	\leq I_n + C(\eps_n+h^2_n/\eps^2_n)\,,
\end{equation}
with
\begin{multline*}
I_n := \frac{1}{\eps_n}\int_{\Omega \cap \{ |d| <\ell \eps_n/2\}} 
	\mathcal{W}\left( \phi_1\big(d(x')/\eps_n\big)+\frac{h_n}{\eps_n}x_3\phi'_2\big(d(x')/\eps_n\big),\phi_2\big(d(x')/\eps_n\big) \right) \,dx \\
	+ \frac{1}{\eps_n}\int_{\Omega \cap \{ |d| <\ell \eps_n/2\}}  \left|\phi'_1\big(d(x')/\eps_n\big)\right|^2
		+ 2 \left| \phi_2'\big(d(x')/\eps_n\big)\right|^2 \, dx\,,
\end{multline*}
and a constant $C$ independent of $n$. 
Using the Coarea Formula, we derive as in the proof of Theorem~\ref{thm:RS_A'neqB'}, Step 3, that 
$$I_n= \int_{-\ell/2}^{\ell/2} \bigg(
	\int_{I}\mathcal{W}\left( \phi_1(t)+{\ts \frac{h_n}{\eps_n}}x_3\phi'_2(t) , \phi_2(t) \right)dx_3 
	+  \left|\phi'_1(t)\right|^2
		+ 2 \left| \phi_2'(t)\right|^2\bigg) \mathcal{H}^1(\mathcal{M}^k_{\eps_n t} \cap \omega) \, dt\,.$$
Since $\mathcal{W}$ is continuous and $h_n/\eps_n \to 0$, we infer from Fatou's lemma, \eqref{limlenglevs}, \eqref{def_K0_sub} and \eqref{estisupFun_sub} that
$$\limsup_{n\to \infty} \, F^{h_n}_{\eps_n}(u_n) \leq (K_0 + 2^{-k}) \mathcal{H}^1(\mathcal{M}^k\cap \ol\omega)\,.$$ 
Then the conclusion follows 
for a suitable diagonal sequence as already pursued in the proof of Theorem~\ref{thm:RS_A'neqB'}, Step 3.  
\end{proof}

\begin{remark}\label{recseq2D}
Given $\eps_n\to 0^+$, a slight modification of the above arguments yields that for every $(u,b)\in \mathscr{C}$, there is a sequence $\{(u_n,b_n)\}\subset H^2(\omega:\R^3)\times H^1(\omega;\R^3)$ 
such that $(u_n,b_n)\to (u,b)$ in $W^{1,p}(\omega;\R^3)\times L^p(\omega;\R^3)$ and $\lim_n F^0_{\eps_n}(u_n,b_n,\omega)=K_{0} \,\Per_{\omega}(E)$ where 
$(\nabla' u,b)= \bigl(1-\chi_E\bigr)A + \chi_E B$, and $F^0_{\eps_n}$ is defined by \eqref{defF0}. 
\end{remark}

\begin{remark}\label{sepscales}
Let us consider an arbitrary sequence $h_n\to 0^+$ and $\eps>0$ fixed. It is well known (see~\cite{KJ}) that the functionals $\{\mathscr{F}^{h_n}_\eps\}$ $\Gamma$-converge 
for the strong $L^1$-topology to 
$$ \mathscr{F}^{0}_\eps(u,b):=\begin{cases}
F_\eps^0(u,b,\omega) & \text{if $(u,b)\in H^2(\Omega:\R^3)\times H^1(\Omega;\R^3)$ and $\partial_3 u=\partial_3 b=0$}\,,\\
+\infty & \text{otherwise}\,,
\end{cases}$$
where $F^0_{\eps_n}$ is defined by \eqref{defF0}, and we have identified functions $(u,b)$ satisfying $\partial_3 u=\partial_3 b=0$ with functions defined on the mid-surface $\omega$. 
Let us  now consider an arbitrary  sequence $\eps_n\to 0^+$. By Remark \ref{liminf2D} and Remark \ref{recseq2D}, the functionals $\{ \mathscr{F}^{0}_{\eps_n}\}$ in 
turn $\Gamma$-converge for the strong $L^1$-topology to~$\mathscr{F}_{0}$ (compactness follows as in Theorem \ref{thm:compactness} with minor modifications). 
\end{remark}


%
%


\section{$\Gamma$-convergence in the supercritical regime}

This section is essentially devoted to the proof of Theorem \ref{thm:gammalim_gammasup}. The  $\Gamma$-liminf inequality is 
a direct consequence of the results in Section~\ref{gamliminfcrit} once we have proved that under assumption  \eqref{orientationwellssc}, $K_\infty^\star<+\infty$.  
In contrast with the lower inequality, the estimate for the $\Gamma$-$\limsup$ requires a more sophisticated construction based on an homogenization procedure. 
The $\Gamma$-$\liminf$ and $\Gamma$-$\limsup$  inequalities  are stated in 
Theorem \ref{thm:gammaliminfsupcrit} and Theorem \ref{thm:RS_A'neqB'sc} respectively, and the conclusion follows from Lemma \ref{finitekgam}. For $p=2$, $\lambda=0$, 
and under the symmetry assumption on $W$, we obtain the $\Gamma$-convergence of the functionals through Corollary \ref{prop:charK_perscvert}.   
In a last subsection, we consider the situation where the wells $A$ and $B$ are not compatible, and we illustrate some 
rigidity phenomena in Theorems \ref{rigidrank} and \ref{rigidrank2}.

%
%

\subsection[]{The $\Gamma$-$\liminf$ inequality}

We define the constant $K_\infty^\star$ as in \eqref{eq:defKstar1} with $\gamma=+\infty$, {\it i.e.}, 
\begin{multline}\label{defKstarinf}
K_\infty^\star:=\inf \biggl\{\liminf_{n\to\infty}\, F^{h_n}_{\eps_n}(u_n,Q)\,:\, h_n\to0^+ \text{ and } \eps_n \to 0^+  \text{ with } h_n/\eps_n\to \infty,\\ 
\{u_n\} \subset H^{2}(Q;\R^3),\,
\text{$(u_n,\frac{1}{h_n}\partial_3 u_n) \to (u_0, b_0)$  in  $[L^1(Q;\R^3)]^2$} \biggr\}\,.
\end{multline}
We start by proving that if \eqref{orientationwellssc} holds, then $K_\infty^\star$ is finite and strictly positive. 

\begin{lemma}\label{finitekgam}
Assume that $(H_1)-(H_3)$ and \eqref{orientationwellssc} hold for some $\lambda\in\R$. Then
$0<K_\infty^\star<\infty$.  
\end{lemma}

\begin{proof}
Let us consider arbitrary sequences $\eps_n\to 0^+$, $h_n\to 0^+$ such that $h_n/\eps_n\to\infty$. Observe that under assumption \eqref{orientationwellssc}, we 
have $A=-B=(a,0,\lambda a)$ so that $A-B$ is rank-1 connected. By the results in \cite{CFL}, there exists a sequence $\{w_n\}\subset H^{2}\big((-1,1);\R^3\big)$ such that 
$w_n\to \bar u_0$ in $W^{1,p}\big((-1,1);\R^3\big)$, and 
$$\sup_{n\in\NN}\, \int_{-1}^1\frac{1}{\eps_n}\min\big\{|w_n^\prime-a|^p,|w_n^\prime+a|^p\big\}+\eps_n|w_n^{\prime\prime}|^2\,dt<\infty\,.$$
For $n$ large enough, we consider the sequence $\{u_n\}\subset H^{2}(Q;\R^3)$ defined by $u_n(x):=w_n(x_1+\lambda h_n x_3)$. Then one may check 
that $(u_n,\frac{1}{h_n}\partial_3 u_n) \to (u_0, b_0)$  in  $[L^1(Q;\R^3)]^2$. Using Lemma~\ref{rempot}, we estimate 
$$
F^{h_n}_{\eps_n}(u_n,Q)\leq C \int_{-1}^1\frac{(1+\lambda^2)^{p/2}}{\eps_n}\min\big\{|w_n^\prime-a|^p,|w_n^\prime+a|^p\big\}+\eps_n(1+\lambda^2)^2|w_n^{\prime\prime}|^2\,dt\,,
$$
which shows that $\sup_n F^{h_n}_{\eps_n}(u_n,Q)<\infty$, and thus $K^*_\infty<\infty$. On the other hand, we have 
\begin{multline*}
K_\infty^\star\geq \inf \biggl\{\liminf_{n\to\infty}\,  \int_{Q} \frac{1}{\eps_n} W(\nabla'u_n,b_n)+\eps_n |(\nabla')^2 u_n|^2+2\eps_n |\nabla' b_n|^2\,dx'\,:\, \eps_n \to 0^+ \,,\\ 
\{(u_n,b_n)\} \subset H^{2}(Q;\R^3)\times H^1(Q;\R^3)\,,\,
\text{$(u_n,b_n) \to (u_0, b_0)$  in  $[L^1(Q;\R^3)]^2$} \biggr\}\,.
\end{multline*}
In view of Lemma~\ref{rempot} and \cite{CFL}, we easily infer that 
\begin{align*}
K_\infty^\star& \geq 
 \begin{multlined}[t][14cm]
 \inf \biggl\{\liminf_{n\to\infty}\,  \int_{I} \frac{1}{C_*\eps_n} \min\big\{|v^\prime_n-a|^p, |v^\prime_n+a|^p\big\}+\eps_n |v^{\prime\prime}_n|^2\,dt\,:\, \eps_n \to 0^+ \,,\\ 
\{v_n\} \subset H^{2}(I;\R^3)\,,\,\text{$v_n \to \bar u_0$  in  $L^1(I;\R^3)$} \biggr\} 
\end{multlined} \\
&\geq  \begin{multlined}[t][14cm]
\inf \biggl\{\int_{-L}^L \frac{1}{C_*} \min\big\{|v(t)-a|^p, |v(t)+a|^p\big\}+ |v^{\prime}|^2\,dt\,:\, L>0 \,,\,
\text{$v$ piecewise $C^1$}\,,\\
v(L)=v(-L)=a \biggr\} >0\,,
\end{multlined} 
\end{align*}
and the proof is complete. 
\end{proof}

Thanks to Lemma \ref{finitekgam} and Remark \ref{validKinfty}, we can now reproduce the first step in the proof of Theorem~\ref{thm:gammaliminf} to obtain the following result. 

\begin{theorem}\label{thm:gammaliminfsupcrit}
Assume that $(H_1)-(H_3)$ and  \eqref{orientationwellssc} hold for some $\lambda\in\R$. 
Let $h_n\to0^+$ and $\eps_n\to0^+$ be arbitrary sequences such that $h_n/\eps_n \to \infty$.  
Then, for any $(u,b)\in \mathscr{C}$ and any sequences $\{u_n\}\subset H^2(\Omega;\R^3)$ such that $(u_n,\frac{1}{h_n}\partial_3 u_n)\to (u,b)$ in $[L^1(\Omega;\R^3)]^2$, 
we have  
$$
\liminf_{n\to\infty}\, F^{h_n}_{\eps_n}(u_n) \geq K_\infty^{\star}\, {\rm Per}_{\omega}(E)\,,
$$
where $(\nabla' u,b) (x)= \bigl(1-\chi_E(x')\bigr)A + \chi_E(x') B\,$.
\end{theorem}

%
%

\subsection[]{Lower bound on $K^\star_\infty$ in the case $\lambda=0$}

As a direct consequence of Lemma \ref{rempot}, we have the following elementary property in the case $\lambda=0$. 

\begin{lemma} \label{borninfW}
Assume that $(H_1)-(H_3)$ and \eqref{orientationwellssc} hold with $\lambda=0$. Then there is a constant $C_W>0$ such that 
$W(\xi)\geq C_W |\xi_3|^p$ for all $\xi=(\xi_1,\xi_2,\xi_3)\in\R^{3\times 3}$. 
\end{lemma}

In parallel with Proposition~\ref{prop:matching1}, the next propositions will establish that realizing sequences for~$K_\infty^{\star}$ can be first chosen with lateral boundary conditions, 
and then periodic in the vertical direction. 

\begin{proposition}\label{prop:matching1sc}
Assume that $(H_1)-(H_4)$ and \eqref{orientationwellssc}  hold with $\lambda=0$. Then there exist sequences $h_n\to 0^+$, $\eps_n\to 0$, $\{c_n\} \subset \R^3$, and 
$\{g_n\} \subset C^{2}(Q;\R^3)$ such that  $h_n/\var_n\to\infty$, $g_n$ is independent of~$x_2$ (i.e., $g_n(x)=:\hat g_n(x_1,x_3)$),  $c_n \to 0$, 
$g_n \to u_0$ in $W^{1,p}(Q;\R^3)$,  $\frac{1}{h_n}\partial_3 g_n \to  0$ in $L^p(Q;\R^3)$,  
$$ g_n =  u_0 \; \text{ in } Q\cap\{x_1 < -1/4\}\,,\quad  g_n = u_0 + c_n\; \text{ in } Q\cap\{x_1 < -1/4\}\,,$$  
and $\lim_{n} F^{h_n}_{\eps_n}(g_n,Q) = K^\star_\infty$.  
\end{proposition}

\begin{proof}{\it Step 1.} 
Since $\lambda=0$ we have $b_0=0$, and in view of Lemma~\ref{finitekgam},  there exist sequences $h_n\to 0^+$, $\eps_n\to0^+$ and 
$\{u_n\} \subset H^{2}(Q;\R^3)$ such that $h_n/\eps_n \to \infty$, 
$(u_n,\frac{1}{h_n}\partial_3 u_n) \to  (u_0,  0)$ in $[L^1(Q;\R^3)]^2$, and
$\lim_{n} F^{h_n}_{\eps_n}(u_n,Q) = K^{\star}_\infty<\infty$. 
Arguing as in the proof of Proposition \ref{prop:matching1}, we may assume 
 that $u_n\in C^{2}(Q;\R^3)$, and that $u_n$ is independent of 
$x_2$, {\it i.e.}, $u_n(x)=:\hat u_n(x_1,x_3)$. 
By Theorem~\ref{thm:compactness},   $u_n \to  u_0$ in~$W^{1,p}(Q;\R^3)$, and $\frac{1}{h_n}\partial_3 u_n \to  0$ in $L^p(Q;\R^3)$. 
\vskip5pt

\noindent{\it Step 2 (first matching).} 
As in the proof of  Proposition~\ref{prop:matching1} we consider a partition of  
$(\frac{1}{12},\frac{1}{6}) \times Q'$ into $M_n := \bigl[ \frac{1}{\eps_n}\bigr]$ layers along the $x_1$-direction. 
By Lemma \ref{finitekgam} and Remark \ref{validKinfty}, we can find such a layer  
$L_n := (\theta_n - \frac{1}{12M_n},\theta_n)\times Q' \subset (\frac{1}{12},\frac{1}{6})\times Q'$
 such that \eqref{eq:layer} holds (with $b_0= 0$). 
Then  select a level $t_n \in \bigl(\theta_n - \frac{1}{12M_n},\theta_n\bigr)$ for which \eqref{eq:tn} holds.  
We consider a cut-off function $\varphi_n\in C^\infty(\R)$ satisfying \eqref{glutest}, and  we  set  for $x \in L_n$,
$$
v_n(x) := \bigl( 1-\varphi_n(x_1)\bigr) \bigl(\bar u_{0}(x_1)+ \bar u_n(x_3)\bigr) + \varphi_n(x_1) u_n(x) \,,
$$
with $\bar u_n(x_3) := \hat u_n(t_n,x_3) - \bar u_0(t_n)$. 

We claim that estimates \eqref{grplim1}, \eqref{grplim2}, \eqref{grplim3}, \eqref{grplim4}, and \eqref{grplim5} still hold (with $b_0=0$). 
First note that \eqref{grplim1} is an easy consequence of  \eqref{eq:layer} and \eqref{eq:tn}. 
In view of Lemma \ref{borninfW}, we infer from  \eqref{eq:tn}~that 
\begin{equation}\label{energsliceubscvert}
\frac{1}{\eps_n}\int_{I}\frac{1}{h_n^{p}}| \bar u'_n(x_3)|^p dx_3+\varepsilon_n\int_{I}\frac{1}{h_n^{4}}| \bar u^{\prime\prime}_n(x_3)|^2 dx_3\leq C\alpha_n\,.
\end{equation}
Combining \eqref{eq:layer} and \eqref{energsliceubscvert} yields \eqref{grplim2}. 
By construction $u_n - u_{0}- \bar u_n=0 $ on  $\{x_1=t_n\}\cap Q $,  
and applying Poincar\'e's inequality we deduce from \eqref{eq:layer},
\begin{equation}\label{eq:Poincarescvert}
\int_{L_n} |u_n- u_{0}- \bar u_n|^p \, dx \leq  C \left(\frac{1}{M_n}\right)^p \int_{L_n} | \partial_1 u_n - \partial_1  u_0|^p dx
\leq C\alpha_n \eps^{p+1}_n\,.
\end{equation}
Using \eqref{glutest}, we may now infer that 
\begin{equation}\label{eq:Poincarescvertbis}
\frac{1}{\eps_n}\int_{L_n} |\nabla' v_n - \nabla' u_0|^p  dx 
\leq \frac{C}{\eps_n}\int_{L_n} |\partial_1 u_n -  \partial_1u_0|^p
+\frac{1}{\eps_n^p}|u_n-u_{0}-\bar u_n|^p dx\leq C\alpha_n\to 0\,.
\end{equation}
Estimates \eqref{grplim2}  and  \eqref{grplim3} being proved, \eqref{grplim4} now follows exactly as in \eqref{estimatchdist}. 

Using again \eqref{glutest}, we estimate 
\begin{multline*}
  \eps_n \int_{L_n} |\nabla_{h_n}^2 v_n|^2 \, dx
    \leq C\bigg(\eps_n \int_{L_n} |\nabla_{h_n}^2 u_n|^2dx 
    +\frac{1}{\eps_n} \int_{L_n}  |\grad' u_n - \grad'  u_0 |^2dx \\
  +  \frac{1}{\eps^3_n} \int_{L_n}  |u_n -u_{0}- \bar u_n|^2 dx
    +\frac{1}{\eps_n} \int_{L_n} |\frac{1}{h_n}\partial_3 u_n|^2dx+\int_{I}\frac{1}{h_n^{2}}| \bar u^{\prime}_n(x_3)|^2 dx_3
    +\varepsilon^2_n\int_{I}\frac{1}{h_n^{4}}| \bar u^{\prime\prime}_n(x_3)|^2 dx_3
           \bigg)\,,
\end{multline*}
Arguing as in the proof of Proposition~\ref{prop:matching1}, Step 2,  \eqref{grplim5} now follows from 
\eqref{eq:layer}, \eqref{energsliceubscvert}, \eqref{eq:Poincarescvert} and \eqref{eq:Poincarescvertbis}  together with H\"older's inequality. 
\vskip5pt

\noindent{\it Step 3 (second matching).} 
Let $\psi_n\in C^\infty(\R)$ be a cut-off function  such that $0\leq \psi_n\leq 1$, $\psi_n(t) = 1$ if $t \leq \theta_n$, $\psi_n(t) = 0$ if $t \geq 1/4$, and satisfying 
$| \psi_n'| +| \psi_n''| \leq C$ for a constant $C$ independent of $n$. 
For $x \in \{\theta_n<x_1<\frac{1}{4}\}\cap Q$, we set  
$$
w_n(x) := u_{0}(x)+ c^+_n + \psi_n(x_1) \bigl( \bar u_n(x_3) - c^+_n\bigr)\,,
$$
where $c^+_n :=\int_{I} \bar u_n\, dx_3\to 0$,  
thanks to \eqref{eq:tn}.  We claim that \eqref{grplim1bis}, \eqref{grplim2bis}, \eqref{grplim4bis}, and \eqref{grplim5bis} hold. 

First \eqref{grplim1bis}  and \eqref{grplim2bis}  are direct consequences of  \eqref{eq:tn} and \eqref{energsliceubscvert} respectively. 
Next we apply \eqref{energsliceubscvert} and Poincar\'e's inequality to derive that 
\begin{align}\label{grplim3bissc} 
 \frac{1}{\eps_n}\int_{\{\theta_n<x_1<\frac{1}{6}\}\cap Q}|\nabla'w_n-\nabla' u_0|^pdx & \leq 
\frac{C}{\eps_n}\int_{I}|\bar u_n(x_3)-c^+_n|^pdx_3 \leq C h_n^p\alpha_n\to 0\,.
\end{align}
To prove \eqref{grplim4bis}, we can argue exactly as  in \eqref{estimatchdist} using  \eqref{grplim2bis} and \eqref{grplim3bissc}. 

We finally obtain in much similar ways that
$$
\eps_n \int_{\bigl(\theta_n,\frac{1}{4}\bigr) \times Q'}  \left|\nabla_{h_n}^2 w_n \right|^2dx
\leq C\eps_n\int_{I} |\ol{u}_n-c^+_n|^2+ \frac{1}{h_n^2}|\bar u'_n|^2
+ \frac{1}{h_n^4}|\bar u^{\prime\prime}_n|^2\,dx_3 
\leq C\alpha_n\to 0\,,
$$
and \eqref{grplim5bis} is proved.
\vskip5pt

\noindent{\it Step 4.}  We conclude the proof as in Proposition~\ref{prop:matching1}, Step 4. We first define $g_n^+$ as in \eqref{defgn+} (with $b_0= 0$), 
and then we repeat the procedure to modify  $g_n^+$ in $(-\frac{1}{2},0)\times Q'$. We omit further details. 
\end{proof}

We now prove that, in the case where $p=2$, $\lambda=0$, and $W$ is symmetric in $\xi_3$,  optimal sequences for $K_\infty^{\star}$ can be modified into $1$-periodic functions  in the 
$x_3$-variable without increasing the energy.

\begin{proposition}[vertical periodicity]\label{transperiodvert}
Assume that $(H_1)-(H_4)$ and \eqref{orientationwellssc} hold with $p=2$, $\lambda=0$, and that $W(\xi',\xi_3)=W(\xi',-\xi_3)$ for every $(\xi',\xi_3)\in\mathbb{R}^{3\times 2}\times\R^3$. 
Then there exist sequences $h_n\to 0^+$, $\eps_n\to 0^+$,  and $\{f_n\} \subset C^{2}(\R^3;\R^3)$ such that  $h_n/\eps_n\to \infty$,  $f_n$ is independent of $x_2$ 
(i.e., $f_n(x)=\hat f_n(x_1,x_3)$),  $f_n \to  u_0$ in $H^1(Q;\R^3)$,  $\frac{1}{h_n}\partial_3 f_n \to  0$ in $L^2(Q;\R^3)$, 
$f_n$ is 1-periodic in the $x_3$-variable,  $ \nabla f_n=\nabla u_{0}$ in $\{|x_1|>  1/4\}$, and 
$$\lim_{n\to\infty} F^{h_n}_{\eps_n}(f_n,Q) = K_\infty^{\star}\,.$$
\end{proposition}

\begin{proof}
{\it Step 1.} We claim that it suffices to find sequences  $h_n\to 0^+$, $\eps_n\to 0^+$,  and 
$\{g^\sharp_n\} \subset C^{2}(\R^3;\R^3)$ such that  $h_n/\eps_n\to \infty$,  $g^\sharp_n(x)=:\hat g^\sharp_n(x_1,x_3)$, 
$\nabla g^\sharp_n=\nabla u_{0}$ in  $\{|x_1|>  1/4\}$, $g^\sharp_n$ is 2-periodic in $x_3$,  
$g^\sharp_n \to  u_0$ in $H^{1}(Q;\R^3)$, $\frac{1}{h_n}\partial_3g^\sharp_n \to  0$ in $L^2(Q;\R^3)$, and 
$
\limsup_{n} F^{h_n}_{\eps_n}(g^\sharp_n,2Q) \leq 4 K_\infty^{\star}$.  
Indeed, if the claim holds we set $f_n(x):= \frac{1}{2}g^\sharp_n(2x)$ for $x\in\R^3$. 
Then $f_n \to  u_0$ in $H^{1}(Q;\R^3)$,  $\frac{1}{h_n}\partial_3 f_n \to  0$ in~$L^2(Q;\R^3)$. By definition of $K^\star_\infty$, a change of variables yields 
 $$K_\infty^{\star} \leq\liminf_{n\to\infty} F^{h_n}_{\frac{\eps_n}{2}}(f_n,Q)\leq \limsup_{n\to\infty} F^{h_n}_{\frac{\eps_n}{2}}(f_n,Q)= 
\limsup_{n\to\infty} \frac{1}{4} F^{h_n}_{\eps_n}(g^\sharp_n,2Q) \leq  K_\infty^{\star}\,,
$$
and thus $\{f_n\}$ satisfies the requirements (with $\eps_n/2$ instead of $\eps_n$). 
\vskip5pt

\noindent{\it Step 2.} Let $h_n\to 0^+$ and $\eps_n\to0^+$ satisfying $h_n/\eps_n \to \infty$.  Consider an arbitrary sequence $\{u_n\} \subset H^{2}(Q;\R^3)$ such that 
$(u_n, \frac{1}{h_n}\partial_3 u_n) \to  (u_0,b_0)$ in $[L^1(Q;\R^3)]^2$, and $\lim_{n} F^{h_n}_{\eps_n}(u_n,Q) = K^{\star}_\infty$. 
We claim that for any $0<\delta<1/2$, we have 
$$\limsup_{n\to\infty} F^{h_n}_{\eps_n}\left(u_n,Q'\times\big((1/2-\delta,1/2)\cup(-1/2,-1/2+\delta)\big)\right) \leq 2\delta K^{\star}_\infty\,.$$
This is of course equivalent to the following inequality,
\begin{equation}\label{equivineq}
\liminf_{n\to\infty}F^{h_n}_{\eps_n}\left(u_n,Q'\times (-1/2+\delta,1/2-\delta)\right)\geq (1-2\delta)K_\infty^{\star} \,,
\end{equation}
that we prove by rescaling. For $x\in Q$, we set $v_n(x):=u_n(x',(1-2\delta)x_3)$ and $\tilde h_n:=(1-2\delta)h_n$. Then $\tilde h_n/\eps_n\to\infty$ 
and $(v_n, \frac{1}{\tilde h_n}\partial_3 v_n) \to  (u_0,b_0)$ in $[L^1(Q;\R^3)]^2$. Therefore, 
$$K_\infty^{\star} \leq \liminf_{n\to\infty}F^{h_n}_{\eps_n}(v_n,Q) = \liminf_{n\to\infty}\frac{1}{1-2\delta}F^{h_n}_{\eps_n}\left(u_n,Q'\times (-1/2+\delta,1/2-\delta)\right)\,,$$
and \eqref{equivineq} follows. 

\vskip5pt
\noindent{\it Step 3.} Consider the sequences $\{h_n\}$, $\{\eps_n\}$, and $\{g_n\}\subset C^{2}(Q;\R^3)$ given by Proposition~\ref{prop:matching1sc}, and 
let us fix  $m\in\NN$ arbitrarily  large. 
We infer from Step 2 (with $\delta=1/m$) that
\begin{equation}\label{enprestrip}
\limsup_{n\to\infty} F^{h_n}_{\eps_n}\left(g_n, Q\cap\left\{\frac{1}{2}-\frac{1}{m}<|x_3|<\frac{1}{2}\right\} \right) \leq \frac{2}{m}K_\infty^{\star}\,.
\end{equation}
Next we divide $Q'\times (\frac{1}{2}-\frac{1}{m},\frac{1}{2})$ into $[\frac{h_n}{\eps_n}]$ thin horizontal strips $R^+_{m,n,i}$ of width $\frac{1}{m}[\frac{h_n}{\eps_n}]^{-1}$, {\it i.e.}, 
$$R^{+}_{m,n,i}:= Q' \times\left(\frac{1}{2}-\frac{i}{m}\left[\frac{h_n}{\eps_n}\right]^{-1},\frac{1}{2}-\frac{i-1}{m}\left[\frac{h_n}{\eps_n}\right]^{-1}\right)$$
for $i=1,\ldots,[\frac{h_n}{\eps_n}]$. We proceed symmetrically in the set $Q' \times (-\frac{1}{2},
-\frac{1}{2}+\frac{1}{m})$, and we denote by $R^-_{m,n,i}$ the resulting strips.  
Applying Lemma~\ref{borninfW}, we infer from \eqref{enprestrip} that for $n$ large enough,
\begin{multline*}
\sum_{i=1}^{[\frac{h_n}{\eps_n}]}\int_{R^-_{m,n,i}\cup R^+_{m,n,i}} \bigg(\frac{1}{\eps_n}W(\nabla_{h_n}g_n)+\eps_n\left|\nabla_{h_n}^2g_n\right|^2
+|\nabla'g_n-\nabla' u_0|^2 \\
+\frac{C_W}{\eps_n}\big|\frac{1}{h_n}\partial_3g_n\big|^2+|g_n- u_0|^2\bigg)\,dx\leq \frac{4}{m}K_\infty^{\star}\,.
\end{multline*}
where we also have used the fact that $\|g_n- u_0\|_{H^{1}(Q)}\to 0$. Now consider a pair of 
strips $(R^-_{m,n,i_0},R^+_{m,n,i_0})$ with $i_0=i_0(m,n)$ satisfying
\begin{multline}\label{energstripR}
\int_{R^-_{m,n,i_0}\cup R^+_{m,n,i_0}} \bigg(\frac{1}{\eps_n}W(\nabla_{h_n}g_n)+\eps_n\left|\nabla_{h_n}^2g_n\right|^2
+|\nabla'g_n-\nabla' u_0|^2 \\
+\frac{C_W}{\eps_n}\big|\frac{1}{h_n}\partial_3g_n\big|^2+|g_n- u_0|^2\bigg)\,dx\leq \frac{4}{m}\left[\frac{h_n}{\eps_n}\right]^{-1}K_\infty^{\star}\,, 
\end{multline}
and we shall write for simplicity $R^\pm_{m,n}:=R^\pm_{m,n,i_0}$ (respectively).     
 Then we choose a level 
$$t_{m,n}\in \left(\frac{1}{2}-\frac{i_0-1/2}{m}\left[\frac{h_n}{\eps_n}\right]^{-1},\frac{1}{2}-\frac{i_0-1}{m}\left[\frac{h_n}{\eps_n}\right]^{-1}\right)$$
for which
\begin{multline}\label{ennivtmn}
\int_{Q\cap \{|x_3|=t_{m,n}\}} \bigg(\frac{1}{\eps_n}W(\nabla_{h_n}g_n)+\eps_n\left|\nabla_{h_n}^2g_n\right|^2
+|\nabla'g_n-\nabla' u_0|^2 \\
+\frac{C_W}{\eps_n}\big|\frac{1}{h_n}\partial_3g_n\big|^2+|g_n- u_0|^2\bigg)\,d\mathcal{H}^2
\leq 8 K_\infty^{\star}\,.
\end{multline}

Let $\varphi_{m,n}:\R\to [0,1]$ be a smooth cut-off function such that 
$\varphi_{m,n}(t)=0$  for $t>t_{m,n}$,  $\varphi_{m,n}(t)=1$ for $t<t_{m,n}-\frac{1}{2m}\left[\frac{h_n}{\eps_n}\right]^{-1}$, and
\begin{equation}\label{esticutoffmn}
\frac{\eps_n}{m h_n}|\varphi'_{m,n}|+\frac{\eps_n^2}{m^2h_n^2}|\varphi''_{m,n}|\leq C\,, 
\end{equation}
for a constant $C$ independent of $m$ and $n$. We define for $x\in Q$, 
$$w_{m,n}(x):= \varphi_{m,n}(x_3) g_n(x)+\big(1-\varphi_{m,n}(x_3)\big)\hat g_n(x_1,t_{m,n})\,.$$
We shall prove in Step 4 below that 
\begin{equation}\label{convnormwnm}
\limsup_{m\to\infty}\,\limsup_{n\to\infty}\,\|w_{m,n}- u_0\|_{H^{1}(Q)}+\big\|\frac{1}{h_n}\partial_3w_{m,n}\big\|_{L^2(Q)} = 0\,,
\end{equation}
and
\begin{equation}\label{limsupenwnm}
\limsup_{m\to\infty}\,\limsup_{n\to\infty} \, F_{\eps_n}^{h_n}(w_{m,n},Q)\leq K^\star_\infty\,.
\end{equation}
Assuming for the moment that \eqref{convnormwnm} and \eqref{limsupenwnm} hold, 
we find a diagonal sequence $n_m\to+\infty$ such that setting $\eps_m:=\eps_{n_m}$, $h_m:=h_{n_m}$, and $w_m:=w_{m,n_m}$, we have 
$w_{m}\to u_0$ in $H^{1}(Q;\R^3)$, $\frac{1}{h_m}\partial_3w_{m}\to 0$ in $L^2(Q;\R^3)$, and 
$\limsup_{m} F_{\eps_m}^{h_m}(w_{m},Q)\leq K^\star_\infty$.  We now repeat this construction in the strip $R^-_{m,n}$, and we write $\tilde w_m$ the resulting function.  

Since $\tilde w_m$ is independent of $x_3$ in a neighborhood of $\{|x_3|=1/2\}\cap Q$, we may first reflect $\tilde w_m$ across the hyperplane $\{x_3=1/2\}$ 
setting for $\frac{1}{2}\leq x_3 \leq\frac{3}{2}$,  $\tilde w_m(x',x_3):=w_m(x',1-x_3)$, and then we extend $\tilde w_m$ by periodicity to all values of $x_3$. The resulting 
function $\tilde w_m$ belongs to $C^2(Q'\times\R;\R^3)$. Since  $\nabla \tilde w_m=\nabla u_0$ in $\{|x_1|>1/4\}$, we can   extend linearly $\tilde w_m$ in $x_1$, 
and constantly in~$x_2$. We finally set for $x\in\R^3$, $g_m^\sharp(x):=\tilde  w_m\big(x', x_3-\frac{1}{2}\big)$. Since $W(\xi',-\xi_3)=W(\xi',\xi_3)$ for all $\xi\in\R^{3\times 3}$, we find that 
$$F^{h_m}_{\eps_m}(g^\sharp_m,2Q) = 4 F^{h_m}_{\eps_m}(\tilde w_m,Q)\,,$$
so that  the function $g_m^\sharp$ satisfies all the requirements of Step 1. 
\vskip5pt 

\noindent{\it Step 4.} We now complete the proof by showing that \eqref{convnormwnm} and \eqref{limsupenwnm} do hold. To this purpose we shall write 
$$L^+_{m,n}:=Q'\times \left(\frac{1}{2}-\frac{i_0}{m}\left[\frac{h_n}{\eps_n}\right]^{-1},\frac{1}{2}\right)\,. $$
We first estimate 
\begin{align}
\nonumber F_{\eps_n}^{h_n}(w_{m,n},Q)&=
\begin{multlined}[t][11cm]
F_{\eps_n}^{h_n}(g_{n},Q\setminus L^+_{m,n})+ F_{\eps_n}^{h_n}(w_{m,n},R^+_{m,n})\\
+\int_{L^+_{m,n}\setminus R^+_{m,n}}\frac{1}{\eps_n}W\big(\partial_1\hat g_n(x_1,t_{m,n}),0,0\big)+ \eps_n\big|\partial^2_1\hat g_n(x_1,t_{m,n})\big|^2\,dx
\end{multlined} \\
\label{step41}&\leq 
\begin{multlined}[t][11cm]
F_{\eps_n}^{h_n}(g_{n},Q)+ F_{\eps_n}^{h_n}(w_{m,n},R^+_{m,n}) \\
+\frac{1}{m}\int_{-1/2}^{1/2} \frac{1}{\eps_n}W\big(\partial_1\hat g_n(x_1,t_{m,n}),0,0\big)+ \eps_n\big|\partial^2_1\hat g_n(x_1,t_{m,n})\big|^2\,dx_1\,.
\end{multlined} 
\end{align}
By Lemma~\ref{rempot} and Lemma~\ref{borninfW}, we have 
$$W\big(\partial_1\hat g_n(x_1,t_{m,n}),0,0\big)\leq C\left( W\big(\nabla_{h_n}g_n(x',t_{m,n})\big) +\big|\frac{1}{h_n}\partial_3 g_n(x',t_{m,n})\big|^2\right)\leq C W\big(\nabla_{h_n}g_n(x',t_{m,n})\big)\,, $$
so that \eqref{ennivtmn} yields
\begin{multline}\label{step42}
 \frac{1}{m}\int_{-1/2}^{1/2} \frac{1}{\eps_n}W\big(\partial_1\hat g_n(x_1,t_{m,n}),0,0\big)+ \eps_n\big|\partial^2_1\hat g_n(x_1,t_{m,n})\big|^2\,dx_1 \\
 \leq  \frac{C}{m}\int_{Q\cap \{x_3=t_{m,n}\}} \frac{1}{\eps_n}W(\nabla_{h_n}g_n)+\eps_n\left|\nabla_{h_n}^2g_n\right|^2\,d\mathcal{H}^2 
  \leq \frac{C}{m}\,.
 \end{multline}
 Similarly, we infer from \eqref{ennivtmn} that 
 \begin{align}
\nonumber \int_{Q} |w_{m,n}- & u_0|^2+\big|\frac{1}{h_n}\partial_3 w_{m,n}\big|^2\,dx \\
\nonumber &\leq  
\begin{multlined}[t][11cm]
\int_{Q} |g_{n}-u_0|^2+\big|\frac{1}{h_n}\partial_3 g_{n}\big|^2\,dx 
+\int_{R^+_{m,n}} |w_{m,n}-u_0|^2+\big|\frac{1}{h_n}\partial_3 w_{m,n}\big|^2\,dx\\ 
+\frac{1}{m}\int_{-1/2}^{1/2}  |\hat g_{n}(x_1,t_{m,n})-\bar u_0(x_1)|^2\,dx_1
\end{multlined} \\
\label{step43} &\leq \int_{Q} |g_{n}-u_0|^2+\big|\frac{1}{h_n}\partial_3 g_{n}\big|^2\,dx 
+\int_{R^+_{m,n}} |w_{m,n}-u_0|^2+\big|\frac{1}{h_n}\partial_3 w_{m,n}\big|^2\,dx +\frac{C}{m}\,.
\end{align}
In view of \eqref{step41}, \eqref{step42}, \eqref{step43}, and Theorem~\ref{thm:compactness}, to prove \eqref{convnormwnm} and \eqref{limsupenwnm} it suffices to show that 
for every $m\in\NN$ large enough,  
\begin{equation}\label{step4vannorm}
\lim_{n\to\infty}\int_{R^+_{m,n}} |w_{m,n}-u_0|^2+\big|\frac{1}{h_n}\partial_3 w_{m,n}\big|^2\,dx=0\,,
\end{equation}
and
\begin{equation}\label{step4vanen}
\lim_{n\to\infty} F_{\eps_n}^{h_n}(w_{m,n},R^+_{m,n}) =0\,.
\end{equation}

We start with the proof of \eqref{step4vanen}. Writing for $x\in R^+_{m,n}$, 
$$\partial_1 w_{m,n}(x)= \partial_1 g_n(x)  +\big(1-\varphi_{m,n}(x_3)\big)\big(\partial_1\hat g_n(x_1,t_{m,n}) -\partial_1 g_n(x)\big)\,,$$
and 
$$ \frac{1}{h_n}\partial_3 w_{m,n}(x)= \frac{1}{h_n}\partial_3 g_n(x)  -\big(1-\varphi_{m,n}(x_3)\big) \frac{1}{h_n}\partial_3 g_n(x) 
+\frac{\varphi^\prime_{m,n}(x_3)}{h_n}\big(g_n(x)-\hat g_n(x_1,t_{m,n})\big)\,, $$
we derive from Lemma~\ref{rempot}, Lemma~\ref{borninfW}, and \eqref{esticutoffmn} that 
\begin{multline}\label{pointestiW}
W\big(\nabla_{h_n}w_{m,n}(x)\big)\leq C\bigg( W\big(\nabla_{h_n} g(x)\big) \\
+|\partial_1 g_n(x)-\partial_1\hat g_n(x_1,t_{m,n})|^2 
+\frac{m^2}{\eps_n^2}\big|g_n(x)-\hat g_n(x_1,t_{m,n})\big|^2\bigg) \,.
\end{multline}
Using Poincar\'e's inequality and \eqref{energstripR}, we estimate
\begin{equation}\label{step4poinc1}
\frac{m^2}{\eps_n^3}\int_{R^+_{m,n}} \big|g_n(x)-\hat  g_n(x_1,t_{m,n})\big|^2\,dx\leq C\frac{1}{\eps_n}\int_{R^+_{m,n}} \big|\frac{1}{h_n}\partial_3 g_n(x)\big|^2\,dx\leq \frac{C}{m} 
\frac{\eps_n}{h_n}\,,
\end{equation}
and 
\begin{equation}\label{step4poinc2}
\frac{1}{\eps_n}\int_{R^+_{m,n}} |\partial_1 g(x)-\partial_1\hat g_n(x_1,t_{m,n})|^2\, dx\leq C\frac{\eps_n}{m^2h_n^2}\int_{R^+_{m,n}}|\partial^2_{13} g(x)|^2\,dx\leq \frac{C}{m^3} \frac{\eps_n}{h_n}  \,.
\end{equation}
In view of \eqref{pointestiW}, we have thus obtained 
$$\frac{1}{\eps_n}\int_{R^+_{m,n}}W(\nabla_{h_n}w_{m,n})\,dx \leq C\left(\frac{1}{\eps_n}\int_{R^+_{m,n}}W(\nabla_{h_n}g_{n})\,dx + 
\frac{\eps_n}{h_n}\right)\leq C
\frac{\eps_n}{h_n}\,\mathop{\longrightarrow}\limits_{n\to\infty}0\,.$$
Then, straightforward computations using \eqref{esticutoffmn} yield
\begin{multline*}
\big|\nabla^2_{h_n}w_{m,n}(x)\big|^2\leq C\bigg(\big|\nabla^2_{h_n}g_{n}(x)\big|^2+
\big|\partial^2_1\hat g_{n}(x_1,t_{m,n})\big|^2+\frac{m^2}{\eps_n^2} \big|\partial_1 g_n(x)-\partial_1\hat g_n(x_1,t_{m,n}) \big|^2\\
+\frac{m^2}{\eps_n^2}\big|\frac{1}{h_n}\partial_{3} g_{n}(x)\big|^2+\frac{m^4}{\eps^4_n}\big|g_n(x)-\hat g_n(x_1,t_{m,n})\big|^2\bigg)\,.
\end{multline*}
Combining \eqref{energstripR}, \eqref{ennivtmn}, \eqref{step4poinc1}, and \eqref{step4poinc2}, we deduce that 
$$\eps_n\int_{R^+_{m,n}}\big|\nabla^2_{h_n}w_{m,n}\big|^2\,dx\leq C m \frac{\eps_n}{h_n}\,\mathop{\longrightarrow}\limits_{n\to\infty}0\,,$$
which completes the proof of \eqref{step4vanen}. 

Using   \eqref{energstripR}, \eqref{ennivtmn}, and \eqref{step4poinc1}, we finally estimate 
\begin{multline*}
\int_{R^+_{m,n}} |w_{m,n}-u_0|^2+\big|\frac{1}{h_n}\partial_3 w_{m,n}\big|^2\,dx\leq C\bigg(\int_{R^+_{m,n}}
 |g_{n}-u_0|^2+\big|\frac{1}{h_n}\partial_3 g_{n}\big|^2 \,dx \\
+ \frac{m^2}{\eps_n^2} \int_{R^+_{m,n}}\big|g_n(x)-\hat  g_n(x_1,t_{m,n})\big|^2 \,dx 
 +\frac{\eps_n}{h_n} \int_{Q\cap\{x_3=t_{m,n}\}} |g_n-u_0|^2\,d\mathcal{H}^2 \bigg) \leq C\frac{\eps_n}{h_n}\,,
 \end{multline*}
and \eqref{step4vannorm} is proved. 
\end{proof}

\begin{corollary}\label{prop:charK_perscvert}
Assume that $(H_1)-(H_4)$ and \eqref{orientationwellssc} hold with $p=2$, $\lambda=0$, and that $W(\xi',\xi_3)=W(\xi',-\xi_3)$ for every $\xi=(\xi',\xi_3)\in\mathbb{R}^{3\times 3}$. 
Then $K^\star_\infty\geq K_\infty\,$. 
\end{corollary}

\begin{proof} 
We consider the sequences $h_n\to 0^+$, $\eps_n\to 0^+$, and $\{f_n\} \subset C^{2}(\R^3;\R^3)$ given by Proposition~\ref{transperiodvert}. We define 
$ N_n:=[\frac{1}{h_n}]$, $\rho_n:=\frac{1}{N_nh_n}$, and $\ell_n:=\frac{1}{\rho_n\eps_n}$ ($[\cdot]$ still denotes the integer part). 
Recalling that $f_n(x)=\hat f_n(x_1,x_3)$, we define for $y\in\R^2$, 
$$ v_n(y):=\rho_n\hat f_n\left(\frac{y_1}{\rho_n},\frac{y_2}{\rho_n h_n}\right)\,.$$
Then $v_n$ is $1/N_n$-periodic in the $y_2$-variable, and  
$\nabla v_n(y)=(\bar u_0(y_1),0)$ in  $\{|y_1| > \frac{\rho_n}{4}\}$. 
Since $v_n$ is $1/N_n$-periodic in $y_2$, and $N_n$ being an integer, we deduce that $v_n$ is also $1$-periodic in $y_2$. Moreover,  
since $\rho_n\to 1$, we have for $n$ large enough 
\begin{equation}\label{cpadutou}
\nabla v_n(y)=(\bar u_0(y_1),0)\text{ in }\{|y_1|>1/3\}\,.
\end{equation}
Hence, 
$$\int_{Q'} \ell_n\,\mathcal{W}(\nabla v_n)+\frac{1}{\ell_n}|\nabla^2 v_n|^2\,dy \geq K_\infty\,.$$
Changing variables, using \eqref{cpadutou} and the $1$-periodicity in $x_3$ of $f_n$, we compute
$$\int_{Q'} \ell_n\,\mathcal{W}(\nabla v_n)+\frac{1}{\ell_n}|\nabla^2 v_n|^2\,dy=\rho_n h_nF^{h_n}_{\eps_n}(f_n,Q'\times N_n I)
= F^{h_n}_{\eps_n}(f_n,Q)\mathop{\longrightarrow}\limits_{n\to\infty} K_\infty^\star\, ,$$
which completes the proof. 
\end{proof}

%
%

\subsection[]{The $\Gamma$-$\limsup$ inequality}

The next theorem provides the announced upper bound for the $\Gamma-\limsup$ of the functionals $\{F^{h}_{\eps}\}$ when $\eps\ll h$, 
and thus completing the proof of Theorem \ref{thm:gammalim_gammasup}.

\begin{theorem}\label{thm:RS_A'neqB'sc}
Assume that $(H_1)-(H_4)$ and \eqref{orientationwellssc} hold for some $\lambda\in\R$.  Let $\eps_n\to 0^+$ and $h_n\to 0^+$ be arbitrary 
sequences such that $h_n/\eps_n \to \infty$. Then, for every $(u,b) \in \mathscr{C}$,  
there exists a sequence $\{u_n\} \subset H^{2}(\Omega;\R^3)$ such that $u_n\to u$ in $W^{1,p}(\Omega;\R^3)$, $\frac{1}{h_n}\partial_3 u_n \to b$ in $L^p(\Omega;\R^3)$, and 
\begin{equation}\label{gamlimsupsupcrit}
\limsup_{n\to\infty}\, F^{h_n}_{\eps_n}(u_n) \leq  K_{\infty} \,\Per_{\omega}(E) \,, 
\end{equation}
where $(\nabla' u,b) (x)= \bigl(1-\chi_E(x')\bigr)A + \chi_E(x') B$.
\end{theorem}

\begin{proof}
We first introduce some useful notation. For a given a sequence $h_n\to0^+$, we define
$$\nu_n:=\frac{e_1+\lambda h_n e_3}{\sqrt{1+\lambda^2h_n^2}}\in\mathbb{S}^2 \quad\text{and}\quad \nu^\perp_n:=\frac{-\lambda h_n e_1 +e_3}{\sqrt{1+\lambda^2h_n^2}}\in\mathbb{S}^2\,. $$
We recall that $Q^\prime_{\lambda}$ denotes the unit cube of $\R^2$ centered at the origin 
with two faces orthogonal to the unit vector $\nu_\lambda=\frac{1}{\sqrt{1+\lambda^2}}(1,\lambda)$. 
\vskip3pt

By Theorem \ref{thm:BJ} and \eqref{orientationwellssc},  $\partial^\ast E\cap \omega$ is of the form \eqref{structredbdAdiffB}. We assume 
that $\partial^\ast E\cap \omega$ is made by finitely many interfaces, {\it i.e.}, $\mathscr{I}=\{1,\ldots,m\}$ in \eqref{structredbdAdiffB}. 
The proof for infinitely many interfaces follows from a diagonalization argument as in  the proof of Theorem \ref{thm:RS_A'neqB'}. 
Then 
$u(x)=\bar u(x_1) $ for some function $\bar u$ that we may assume to be as in the proof of Theorem~\ref{thm:RS_A'neqB'}, Step 1 (we refer to it for the notation). Then   
\eqref{orientationwellssc} yields~$b(x)=\lambda \bar u'(x_1)$.

Let us now consider for each $k\in\NN$, some $\ell_k>0$ and some function $v_k\in C^{2}(\R^2;\R^3)$ 1-periodic in the direction $\nu_\lambda^\perp:=\frac{1}{\sqrt{1+\lambda^2}}(-\lambda,1)$, satisfying 
  $\nabla v_k(y)=\pm(a,\lambda a)$ nearby $\{y\cdot\nu_\lambda=\pm1/2\}$ respectively, and such that 
$$\int_{Q'_{\lambda}}\ell_k\mathcal{W}(\nabla v_k)+\frac{1}{\ell_k}|\nabla^2v_k|^2\,dy\leq \frac{K_\infty + 2^{-k}}{\sqrt{1+\lambda^2}}\,. $$ 
Without loss of generality we may assume that 
\begin{equation}\label{fixbdvksc}
v_k(y)=\begin{cases}
\ds \sqrt{1+\lambda^2}\, (y\cdot\nu_\lambda)a+c_k & \text{nearby $\{y\cdot\nu_\lambda=1/2\}$}\,,\\[5pt]
\ds -\sqrt{1+\lambda^2}\, (y\cdot\nu_\lambda)a-c_k &\text{ nearby $\{y\cdot\nu_\lambda=-1/2\}$}\,,  
\end{cases}
\end{equation}
for some constant $c_k\in\R^3$. From now on we drop the subscript $k$ for simplicity.

Let $\eps_n\to0^+$ and $h_n\to 0^+$ be  arbitrary sequences such that $h_n/\eps_n\to+\infty$. 
Again we  choose for each index $i=1,\ldots,m$, an  bounded open interval $J'_i\subset \R$ such that $J_i\subset\!\subset J'_i$ and $\mathcal{H}^1(J'_i\setminus J_i)\leq 2^{-k}$. 
We write 
$$\alpha^n_{i\pm}:=\frac{1}{\sqrt{1+\lambda^2h_n^2}}\left(\alpha_i\pm\frac{\ell \eps_n\sqrt{1+\lambda^2}}{2}\right)\,, $$
and we consider integers $n$ large enough in such a way that $\alpha^n_{i+}<\alpha^n_{(i+1)-}$ for every $i$, and  for which~\eqref{bandeconvex} holds.  
We define the transition layers as follows: 
for $i=1,\ldots,m$ and for $x\in \R^3$, we set 
$$w^i_{n}(x):=(-1)^{i+1}v\left((-1)^{i+1}\frac{x_1-\alpha_{i}}{\ell\eps_n}\,, (-1)^{i+1}\frac{h_n x_3}{\ell \eps_n}\right)
+\left(1+(-1)^{i}\right)\left(\frac{1}{2}\sqrt{1+\lambda^2}\,a+c \right)\,.$$
Then \eqref{fixbdvksc} yields 
\begin{equation}\label{bdlayersc}
w_{n}^i(x)=
\frac{1}{2}\sqrt{1+\lambda^2}\,a -(-1)^{i+1} c \quad\text{ on } \left\{x\cdot\nu_n=\alpha^n_{i-}\right\}\,,
\end{equation}
and
\begin{equation}\label{bdlayer+sc}
w_{n}^i(x)=\left(\frac{1}{2}\sqrt{1+\lambda^2}\,a +(-1)^{i+1} c \right)+2(1+(-1)^i)c \quad\text{ on } \left\{x\cdot\nu_n=\alpha^n_{i+}\right\}\,.
\end{equation}
Setting
$$\beta^n_i:=\sum_{j=1}^i\bar u\left(\alpha^n_{j+}\sqrt{1+\lambda^2h_n^2}\right)-\bar u\left(\alpha^n_{j-}\sqrt{1+\lambda^2h_n^2}\right)\,,$$
with $\beta_0^n:=0$ and $\kappa_i$ as in \eqref{notgamsupbis}, we define for $n$ large enough and $x\in\Omega$, 
$$u_{n}(x):= \begin{cases}
\bar u(x_1+\lambda h_nx_3) +\ell \eps_n\big(\frac{a}{2}\sqrt{1+\lambda^2}-c\big) & \text{for } x\cdot\nu_n\leq \alpha^n_{1-}\,,\\[8pt]
\ds  \bar u\left(\alpha^n_{i-}\sqrt{1+\lambda^2h_n^2}\right)-\beta^n_{i-1}+\ell \eps_n\big(w^i_{n}(x)+\kappa_{i-1}c \big)
& \text{for } \alpha^n_{i-}< x\cdot\nu_n< \alpha^n_{i+}\,,\\[8pt]
 \bar u(x_1+\lambda h_nx_3)-\beta^n_{i} +\ell \eps_n\big(\frac{a}{2}\sqrt{1+\lambda^2}+((-1)^{i+1}+\kappa_i)c \big)&\text{for } 
\alpha^n_{i+}\leq x\cdot\nu_n\leq \alpha^n_{(i+1)-}\,,\\[8pt]
 \bar u(x_1+\lambda h_n x_3)-\beta^n_{m} +\ell \eps_n\big(\frac{a}{2}\sqrt{1+\lambda^2}+((-1)^{m+1}+\kappa_m)c \big) &
 \text{for } x\cdot\nu_n\geq \alpha^n_{m+}\,.
\end{cases}$$
Using \eqref{bdlayersc}-\eqref{bdlayer+sc} one may check that $u_{n}$ and $\nabla u_{n}$ are continuous across  
each interface $\{x\cdot\nu_n= \alpha^n_{i\pm}\}$, and thus  $u_{n}\in H^2(\Omega;\R^3)$. In addition  $\partial_2 u_{n}\equiv 0$, and 
\begin{equation}\label{structgradresclsupcrit}
\left(\partial_1 u_{n},\frac{1}{h_n}\partial_3 u_{n}\right)(x)=\begin{cases}
\ds \nabla v \left((-1)^{i+1}\frac{x_1-\alpha_{i}}{\ell \eps_n}\,, (-1)^{i+1}\frac{h_nx_3}{\ell \eps_n}\right) &
 \text{for }\alpha^n_{i-}< x\cdot\nu_n< \alpha^n_{i+}\,,\\[8pt]
\ds \big(\bar u'(x_1+\lambda h_n x_3),\lambda \bar u'(x_1+\lambda h_nx_3)\big) & \text{otherwise}\,.
\end{cases}
\end{equation}
Then one observes that the maps $x\in\Omega\mapsto \bar u(x_1+\lambda h_nx_3) $ and $x\in\Omega\mapsto \bar u'(x_1+\lambda h_nx_3) $ converge to $u$ and $b$ in 
$W^{1,p}(\Omega;\R^3)$ and in $L^p(\Omega;\R^3)$ respectively as $n\to\infty$ (here we also use the fact that $b=\lambda \bar u'$). On the other hand,  $v$ and $\nabla v $ are 
bounded in $\{|y\cdot\nu_\lambda|\leq 1/2\}$ by periodicity in the direction~$\nu_{\lambda}^\perp$, and $|\beta_i^n|\leq C\eps_n$ for a constant $C$ independent of $i$ and $n$ by the Lipschitz continuity of~$\bar u$. 
Hence  $u_{n}\to u$ in $W^{1,p}(\Omega;\R^3)$ and $\frac{1}{h_n}\partial_3u_{n}\to b$ in $L^p(\Omega;\R^3)$.

By \eqref{bandeconvex} we have for $n$ large, 
$$\Omega\cap\left\{\alpha^n_{i-}< x\cdot\nu_n< \alpha^n_{i+}\right\}\subset \left\{x\in\R^3 : \alpha^n_{i-}< x\cdot\nu_n< \alpha^n_{i+}\,,\,|x_3|<1/2\,,\,x_2\in J'_i\right\}=:\Omega_i^n \,,$$
Using  \eqref{structgradresclsupcrit} we estimate for $n$ large enough, 
\begin{equation} \label{limsupapproxsc}
F^{h_n}_{\eps_n}(u_{n})
\leq \sum_{i=1}^mF^{h_n}_{\eps_n}\left(\ell \eps_n w^i_{n},\Omega_i^n\right)\,,
\end{equation}
and it remains to estimate each term of the sum in the right-hand side of \eqref{limsupapproxsc}. 

Changing variables, one obtains
\begin{equation}\label{changevarsc}
F^{h_n}_{\eps_n}\left(\ell \eps_n w^i_{n},\Omega_i^n\right)=\frac{\ell \eps_n}{h_n}\mathcal{H}^1(J'_i)\int_{\Theta_i^n}\ell \,\mathcal{W}(\nabla v)+\frac{1}{\ell }\left|\nabla^2v \right|^2\,dy\,,
\end{equation}
where $\Theta_i^n:=\{y\in\R^2: |y\cdot\nu_\lambda|<1/2\,,\,|y_2|<h_n/(2\ell \eps_n)\}$. Notice that for every $t\in(-\frac{1}{2},\frac{1}{2})$, we have
\begin{multline*}
\Theta_i^n\cap \{y\cdot\nu_\lambda=t\}=\{y\cdot\nu_\lambda=t\}\cap\left\{|y\cdot\nu^\perp_\lambda+\lambda t|<\frac{h_n\sqrt{1+\lambda^2}}{2\ell \eps_n}\right\} \\
\subset \{y\cdot\nu_\lambda=t\}\cap\left\{|y\cdot\nu^\perp_\lambda+\lambda t|<\frac{N_n}{2}\right\} \,,
\end{multline*}
with $N_n:=\left[\frac{h_n\sqrt{1+\lambda^2}}{\ell \eps_n}\right]+1$. 
Using Fubini's theorem and the periodicity of $v$, we estimate
\begin{align}
\nonumber \int_{\Theta_i^n}\ell \,\mathcal{W}(\nabla v )+\frac{1}{\ell}\left|\nabla^2v \right|^2\,dy & =\int_{-\frac{1}{2}}^{\frac{1}{2}}\left(\int_{\Theta_i^n\cap \{y\cdot\nu_\lambda=t\}}
\ell \,\mathcal{W}(\nabla v )+\frac{1}{\ell }\left|\nabla^2v \right|^2\,d\mathcal{H}^1\right)\,dt \\
\nonumber& \leq \int_{-\frac{1}{2}}^{\frac{1}{2}}\left(\int_{ \{y\cdot\nu_\lambda=t\}\cap\left\{|y\cdot\nu^\perp_\lambda+\lambda t|<\frac{N_n}{2}\right\}  }
\ell\, \mathcal{W}(\nabla v )+\frac{1}{\ell}\left|\nabla^2v\right|^2\,d\mathcal{H}^1\right)\,dt \\
\label{estimcubesup}& \leq N_n \int_{Q'_{\lambda}}\ell \,\mathcal{W}(\nabla v )+\frac{1}{\ell }\left|\nabla^2v \right|^2\,dy \,. 
\end{align}
Combining \eqref{changevarsc} with \eqref{estimcubesup} yields 
$$F^{h_n}_{\eps_n}\left(\ell \eps_n w^i_{n},\Omega_i^n\right)\leq \frac{\ell \eps_nN_n}{h_n\sqrt{1+\lambda^2}}\left(K_\infty+2^{-k}\right)\mathcal{H}^1(J'_i)\,.$$
Summing up over $i$ this last inequality, and passing to the limit $n\to+\infty$ in \eqref{limsupapproxsc} leads to 
$$\limsup_{n\to\infty}F^{h_n}_{\eps_n}(u_{n})\leq K_\infty  \,\Per_{\omega}(E) +C_02^{-k}\,,$$
for a constant $C_0$ independent of $k$. Then the conclusion follows 
for a suitable diagonal sequence as already pursued in the proof of Theorem~\ref{thm:RS_A'neqB'}, Step 3.  
\end{proof}

\begin{remark}\label{sepscalessc}
Let us consider an arbitrary sequence $\eps_n\to 0^+$ and $h>0$ fixed, and assume for simplicity that $W(\xi) =\dist\big(\xi,\{A,B\}\big)^p$. If \eqref{orientationwellssc} holds, we can apply the results in \cite{CFL} to infer that the functionals  $\{{\bf E}_{\eps_n}(\ \cdot\ ,\Omega_h)\}$ (defined in \eqref{mod1}) $\Gamma$-converge 
for the strong $L^1(\Omega_h)$-topology to 
$$
{\bf E}_{0}({\bf u},\Omega_h):=\begin{cases}
K_* {\rm Per}_{\Omega_h}(\{\grad {\bf u}=B\}) & \text{if ${\bf u}\in W^{1,1}(\Omega_h;\R^3)$, and $\nabla{\bf u}\in BV(\Omega_h;\{A,B\})$}\,,\\
+\infty & \text{otherwise}\,,
\end{cases}
$$
with $K_*=K_\infty/ \sqrt{1+\lambda^2} \,$. 
Let us now consider the rescaling $u(x)={\bf u}(x_1,x_2,x_3/h)$, and define the functional $F^h_0:L^1(\Omega;\mathbb{R}^3)\to [0,\infty]$ by 
$$F^h_0(u):=\frac{1}{h}{\bf E}_{0}({\bf u},\Omega_h) \,.$$
By \cite{BJ,CFL} (see also Theorem~\ref{thm:BJ}), if ${\bf u}$ has finite energy, then the set $F:=\{\grad {\bf u}=B\}$ is layered perpendicularly to the vector $\nu_\lambda$. Setting $E:=F\cap \omega$, we  easily obtain  
that ${\rm Per}_{\Omega_h}(F)=h\sqrt{1+\lambda^2}{\rm Per}_{\omega}(E)+o(h)$, and thus 
$$
F_0^h(u)= K_\infty {\rm Per}_{\omega}(E)+o(1)\,.
$$
It is then straightforward to show that the family $\{F_0^h\}$ $\Gamma$-converges for the strong $L^1$-topology to $\mathscr{F}_\infty$ as $h\to 0^+$.   

%
\end{remark}

%
%

\subsection[]{Some rigidity properties}

For $\eps\ll h$ we expect  the thin film to behave like a three dimensional sample by separation of scales, so that sequences with uniformly bounded energy 
should have trivial limits under suitable assumptions on $A$ and $B$. The first situation we consider is when $A'=B'$ (so $A$ and $B$ are rank-one connected). Indeed, in this case if we first perform the asymptotic $\eps\to0$
(see Remark~\ref{sepscalessc}), the 
limiting configurations $u$ with finite energy must satisfy $\nabla u=\chi_{K}(x_3)A+(1-\chi_{K}(x_3))B$ for some finite set $K\subset I$, and the $\Gamma$-limit  is proportional to $\frac{1}{h}{\rm Card}(K)\mathcal{L}^2(\omega)$, see \cite{CFL}. 
This latter energy can be bounded with respect to $h$ only if ${\rm Card}(K)=0$ for $h$ small, and it formally explain the expected rigidity effect. We have rigorously proved this fact only in the case where 
$\eps$ is sufficiently small relative to $h$ as stated in the following theorem.

\begin{theorem}\label{rigidrank}
Assume $(H_1)-(H_3)$ and \eqref{orientationwells} hold with $A'=B'$. Let $h_n \to 0^+$ and $\eps_n \to 0^+$ be arbitrary sequences such that $\sup_n \eps_n/h_n^p < \infty$. 
Then, for any  $\{u_n\} \subset H^{2}(\Omega;\R^3)$ such that $\sup_n F_{\eps_n}^{h_n}(u_n) < \infty$, 
there exist a subsequence (not relabeled) and $\xi_0 \in \{A,B\}$ such that
$\grad_{h_n} u_n \to \xi_0$ in $L^p(\Omega;\R^{3\times 3})$.
\end{theorem}

\begin{proof} By Theorem~\ref{thm:compactness}, we can find a subsequence such that $\nabla_{h_n}u\to (\nabla'u,b)$  in $L^p(\Omega;\R^{3\times 3})$ 
for some $(u,b)\in\mathscr{C}$. Since $A'=B'$, $\nabla'u$ is constant, and we only have to prove that $b$ is constant. 

By \eqref{orientationwells} we have $A'=B'=0$, and thus Lemma \ref{rempot} yields
$$
W(\xi) \geq \frac{1}{C_*} \big(|\xi'|^p + \min \big\{ |\xi_3 - A_3|^p,|\xi_3-B_3|^p\big\}\big)\,,\quad\forall \xi\in\R^{3\times3}\,,
$$
Setting $v_n := \frac{1}{h_n}u_n$, we deduce that
$$
\sup_{n\in\NN} \int_{\Omega} \frac{h_n^p}{\eps_n} |\grad' v_n|^p + \frac1{\eps_n}\min \big\{ |\partial_3 v_n - A_3|^p,|\partial_3 v_n-B_3|^p\big\} \, dx < \infty\,.
$$
Hence $\{\grad v_n\}$ is bounded in $L^p(\Omega;\R^3)$. By Poincar\'e's inequality, there exists a further subsequence (not relabeled) 
such that $ v_n-\,\med_{\Omega}v_n\, \rightharpoonup v$ weakly in $W^{1,p}(\Omega;\R^3)$ for some $v \in W^{1,p}(\Omega;\R^3)$. 
But since $\partial_3 v_n =\frac{1}{h_n}\partial u_n$,  we have 
$ \partial_3 v_n \to b$ strongly in $L^p(\Omega;\R^3)$. 
Hence $\partial_3\big(v(x) - b(x') x_3\big)=0$, and we can argue as in Theorem~\ref{thm:BJ}, Step 2, to prove that $v(x) = b(x') x_3 + w(x')$
for some function $w\in BV(\omega;\R^3)$. Integrating this equality in $x_3$ over the interval~$I$ yields $w(x')=\int_I v(x',x_3)\,dx_3$ a.e. in $\omega$. It obviously implies that 
$w \in W^{1,p}(\omega;\R^3)$. Since $b(x')x_3=v(x)-w(x')$, we conclude that $b \in W^{1,p}(\omega;\{A_3,B_3\})$, and thus $b$ must be constant.
\end{proof}

The other case where one can expect rigidity is when $A$ and $B$ are not rank-one connected, and thus not compatible in the bulk \cite{BJ}.  
We will show that rigidity occurs at least for some particular potentials $W$ as a consequence of a two-wells rigidity 
estimate due to Chaudhuri \& M\"uller~\cite{CMrig} (see~\cite{FJM} for single well rigidity). The class of double-well potentials we consider is as  follows. 
For simplicity we will assume that 
\begin{equation}\label{spewell}
A=I_d\,, \quad \text{and}\quad B={\rm diag}(\theta_1,1,\theta_2)\,,
\end{equation}
for some $\theta_1,\theta_2 \in\R$ 
satisfying 
\begin{equation}\label{Matos}
\theta_i>0\quad i=1,2\,,\quad \text{and}\quad (1-\theta_1)(1-\theta_2)>0\,.
\end{equation}
Here $I_d$ denotes the $3\times 3$ identity matrix.  
The second assumption in \eqref{Matos} corresponds to the strong incompatibility condition between $A$ and $B$ in the sense of  
Matos \cite{Mat} (see also \cite{CMrig,CM}). Noticing that $A'$ and $B'$ are rank-one connected,  we shall consider  continuous 
potentials $W:\R^{3\times 3}\to [0,\infty)$ such that  $(H_1)-(H_3)$ hold with $p=2$. 
\vskip5pt

Using the rigidity estimate of \cite{CMrig} and an argument similar to \cite{CM}, we have  obtained the following~result. 

\begin{theorem}\label{rigidrank2}
Assume $(H_1)-(H_3)$ hold with $p=2$, \eqref{spewell}, and \eqref{Matos}. Let $h_n \to 0^+$~and $\eps_n \to 0^+$ be~arbitrary sequences such that 
$ \eps_n/h_n\to \infty$. Then, for any  sequence $\{u_n\} \subset H^{2}(\Omega;\R^3)$ such that~$\sup_n F_{\eps_n}^{h_n}(u_n) < \infty$, 
there exist a subsequence (not relabeled) and $\xi_0 \in \{A,B\}$ such that
$\grad_{h_n} u_n \to \xi_0$ in $L^2(\Omega;\R^{3\times 3})$.
\end{theorem}

\begin{proof}
By Theorem  \ref{thm:compactness}, there is a subsequence such that $u_n - \med_{\Omega} u_n \, dx \to u$ in $H^{1}(\Omega;\R^3)$ and
 $\frac{1}{h_n} \partial_3 u_n \to b$ in $L^2(\Omega;\R^3)$ for some $(u,b)\in\mathscr{C}$. 
To prove the announced result, it suffices to prove that for an arbitrary open set $\mathcal{O}\subset \omega$,  $(\nabla'u,b)$ is constant in $\mathcal{O}\times I$. Without loss of 
generality we may assume that $\mathcal{O}=Q'$ the unit cube of $\R^2$. We proceed as follows. 
\vskip5pt

\noindent {\it Step 1.} First we infer from Lemma \ref{rempot} that 
$$W(\xi)\geq \frac{1}{C_*} \min\left\{ \min_{R\in SO(3)} |\xi-R|^2, \min_{R\in SO(3)} |\xi-RB|^2\right\} = \frac{1}{C_*}{\rm dist}^2\big(\xi,K\big)\qquad \forall \xi\in\R^{3\times 3}\,,$$
where $K:=SO(3)\cup SO(3) B$. From this estimate we deduce that  
\begin{equation}\label{energielevelSO3}
\frac{1}{h_n} \int_{Q}{\rm dist}^2\big(\nabla_{h_n} u_n,K\big)\,dx \leq C\frac{\eps_n}{h_n}\,.
\end{equation}
Setting $M_n:=[\frac{2}{h_n}]$, we now divide $Q'$ into $M_n^2$ squares $S_{a,n}$ of the form
$$S_{a,n}=a+M_n^{-1}Q'\quad\text{with}\quad a\in \mathscr{A}^n:= M_n^{-1}\ZZ^2\cap Q'\,,$$
so that  $Q'=\cup_{a\in \mathscr{A}^n} S_{a,n}$ up to a set of $\mathcal{L}^2$-measure zero. Then for each 
$a\in \mathscr{A}^n$, we define the rescaled map $v_n^a:M_n^{-1}Q\to\R^3$ by 
$v_n^a(y):=u_n\left(a+y',\frac{y_3}{h_n}\right)$. 
By \cite[Theorem 2]{CMrig}, there exists a universal constant $C_{\rm univ}$ such that  for each $a\in \mathscr{A}^n$ we can 
find $R^a_n\in K$ satisfying 
$$\int_{M_n^{-1} Q} |\nabla v^a_n-R^a_n|^2\,dy\leq C_{\rm univ} \int_{M_n^{-1} Q}{\rm dist}^2\big(\nabla v^a_n,K\big)\,dy\,.$$
Scaling back, we derive that 
\begin{equation}\label{ineqova}
\int_{S_{a,n}\times \frac{1}{2} I } |\nabla_{h_n} u_n-R^a_n|^2\,dx\leq C_{\rm univ} \int_{S_{a,n}\times I }{\rm dist}^2\big(\nabla_{h_n} u_n,K\big)\,dx
\qquad\forall a\in \mathscr{A}^n\,.
\end{equation}
Defining the piecewise constant map $R_n:Q'\to K$ by $R_n(x'):=R^a_n$ for $x'\in S_{a,n}$, and adding the previous inequalities in \eqref{ineqova} leads to 
$$\int_{Q'\times \frac{1}{2} I } |\nabla_{h_n} u_n-R_n(x')|^2\,dx\leq C_{\rm univ} \int_{Q}{\rm dist}^2\big(\nabla_{h_n} u_n,K\big)\,dx\leq C\eps_n\mathop{\longrightarrow}
\limits_{n\to\infty} 0\,,$$
thanks to \eqref{energielevelSO3}. Since $\nabla_{h_n} u_n \to (\nabla' u,b)$ in $L^2(\Omega;\R^{3\times 3})$, 
we conclude that $R_n\to  (\nabla' u,b)$ in $L^2(Q';\R^{3\times 3})$. 
\vskip5pt

\noindent{\it Step 2.} Let $\delta>0$ be a small parameter to be chosen. We divide $\mathscr{A}^n$ into the following classes, 
$$\mathscr{A}_0^n:=\left\{a\in \mathscr{A}^n\,:\, \int_{S_{a,n}\times I }{\rm dist}^2\big(\nabla_{h_n} u_n,K\big)\,dx\geq \delta h_n^2  \right\}\,, $$
$\mathscr{A}_1^n:=\big\{a\in \mathscr{A}^n\setminus\mathscr{A}_0^n \,:\, R^a_n \in SO(3)\big\}$, and  $\mathscr{A}_2^n:=\big\{a\in \mathscr{A}^n\setminus\mathscr{A}_0^n \,:\, R^a_n \in SO(3)B\big\}$. 
We observe that \eqref{energielevelSO3} yields 
\begin{equation}\label{card}
{\rm Card}(\mathscr{A}_0^n)\leq \frac{1}{\delta h_n^2}\int_{Q}{\rm dist}^2\big(\nabla_{h_n} u_n,K\big)\,dx=o(1/h_n)\,,
\end{equation}
where ${\rm Card}$ denotes the counting measure. Next we consider the sets 
$$G_0^n:=\bigcup_{a\in \mathscr{A}_0^n}S_{a,n}\,,\quad G_1^n:=\bigcup_{a\in \mathscr{A}_1^n}S_{a,n}\,,\quad G_2^n:=\bigcup_{a\in \mathscr{A}_2^n}S_{a,n} \,,$$
so that $Q'=G_0^n\cup G_1^n \cup G_2^n$ up to a set of $\mathcal{L}^2$-measure zero. 

Now, we shall enumerate the edges $\Gamma^j_a$ ($j=1,2,3,4$) of a square  $S_{a,n}$ according the counterclockwise sense, $\Gamma^1_a$ being the bottom edge.  
We observe that each boundary $\partial G_i^n$ is polyhedral and 
made by the edges $\Gamma^j_a$ (of length $M_n^{-1}$) of some squares $S_{a,n}$ with $a\in \mathscr{A}_i^n$, that we 
call {\it boundary squares}. For $i=0,1,2$ and $j=1,2,3,4$, we set 
$$\mathscr{B}_i^n:=\big\{a\in \mathscr{A}_i^n : S_{a,n} \text{ is a {\it boundary square}}\big\}\, \quad\text{and}\quad \mathscr{E}_{i,j}^n:=
\big\{a\in \mathscr{B}_i^n\,:\, \Gamma_a^j\subset \partial G_i^n \cap Q'\big\}\,.$$
We claim that if $\delta>0$ is chosen small enough, then for every $n\in \NN$ large enough, and for $i=1,2$, $j=1,2,3,4$,
\begin{equation}\label{contactedge}
\Gamma_a^j\subset \partial G_0^n \qquad \forall a\in  \mathscr{E}_{i,j}^n\,.
\end{equation}
We shall prove \eqref{contactedge} in the next 
step. Assuming that \eqref{contactedge} is true, we estimate for $i=1,2$, 
$$
\mathcal{H}^1(\partial G_i^n \cap Q')  = \sum_{j=1}^4 \sum_{a\in  \mathscr{E}_{i,j}^n}  \mathcal{H}^1(\Gamma_a^j)  \leq 4 h_n {\rm Card}(\mathscr{A}_0^n)
\mathop{\longrightarrow}\limits_{n\to\infty} 0\,,
$$
thanks to \eqref{card}. Therefore, we can extract a subsequence such that for $i=1,2$, either $\mathcal{L}^2(Q'\setminus G_i^n)\to 0$ or  
$\mathcal{L}^2(G_i^n)\to 0$. Since $\mathcal{L}^2(G_0^n)\to 0$ by \eqref{card}, and $Q'=G_0^n\cup G_1^n \cup G_2^n$, we must have 
 $\mathcal{L}^2(Q'\setminus G_1^n)\to 0$ or $\mathcal{L}^2(Q'\setminus G_2^n)\to 0$. Without loss of generality, we may assume that 
 $\mathcal{L}^2(Q'\setminus G_1^n)\to 0$. Then we estimate 
$$\int_{Q'}{\rm dist}^2\big((\nabla'u,b),SO(3)\big)\,dx' \leq \int_{Q'}\big|(\nabla'u,b)-R_n\big|^2\,dx +C\mathcal{L}^2(Q'\setminus G_1^n)\mathop{\longrightarrow}
\limits_{n\to\infty} 0\,,$$
which yields $(\nabla'u,b)(x')\in SO(3)\cap \{I_d,B\}$ for $\mathcal{L}^2$-a.e. $x'\in Q'$. Since $B\not\in SO(3)$, we finally conclude that~$(\nabla'u,b)\equiv I_d$ in $Q'$. 
\vskip5pt

\noindent{\it Step 3.} It remains to prove \eqref{contactedge}. We argue by contradiction. Without loss of 
generality, we may assume that there exists $a\in \mathscr{E}_{1,1}^n$ such that $\Gamma_a^1\not\subset \partial G_0^n$. Since 
$\Gamma_a^1\subset Q'$, we have $\tilde a:=a-(0,M_n^{-1})\in  M_n^{-1}\ZZ^2\cap Q'$, and $\tilde a\not \in \mathscr{A}_0^n\cup \mathscr{A}_1^n$.  
Thus $\tilde a\in \mathscr{A}_2^n$. As in Step 1, we can apply \cite[Theorem 2]{CMrig} to find $\tilde R^{a}_n\in K$ such that 
\begin{align*}
\frac{1}{h_n^2}\int_{(S_{a,n}\cup S_{\tilde a,n})\times \frac{1}{2}I}\big|\nabla_{h_n}u_n-\tilde R^{a}_n\big|^2\,dx & \leq 
\frac{\tilde C_{\rm univ}}{h_n^2}\int_{(S_{a,n}\cup S_{\tilde a,n})\times I} {\rm dist}^2\big(\nabla_{h_n} u_n,K\big)\,dx \\
&\leq 2 \max\{C_{\rm univ},\tilde C_{\rm univ}\} \delta\,,
\end{align*}
for some universal constant $\tilde C_{\rm univ}$. Then we have 
$$|R^a_n- \tilde R^{a}_n|^2\leq  \frac{16}{h_n^2}\int_{(S_{a,n})\times \frac{1}{2}I}\big|\nabla_{h_n}u_n-R^{a}_n\big|^2+ \big|\nabla_{h_n}u_n-\tilde R^{a}_n\big|^2\,dx 
\leq 32 \max\{C_{\rm univ},\tilde C_{\rm univ}\} \delta\,.$$
We proceed similarly to get $|R^{\tilde a}_n- \tilde R^{a}_n|^2\leq 32 \max\{C_{\rm univ},\tilde C_{\rm univ}\} \delta$, and we obtain a contradiction 
whenever $\delta < [32 \max\{C_{\rm univ},\tilde C_{\rm univ}\}]^{-1}{\rm dist}^2\big(SO(3),SO(3)B\big)$. 
\end{proof}


 \section*{}

{\it\bf Aknowledgements.} 
{\small The authors thank J.-F. Babadjian, I. Fonseca, and G. Leoni for  their advices   
during the preparation of this work. 
Part of this work was done while B. Galv\~ao-Sousa was postdoctoral fellow at McMaster University, and it was partially supported by {\it Funda\c{c}\~{a}o para a Ci\^{e}ncia e a Tecnologia} under grant PRAXIS XXI SFRH/BD/8582/2002, and by 
the {\it Center for Nonlinear Analysis (CNA)} under the National Science Fundation Grant No. 0405343. The  
research of V. Millot was partially supported by the Agence Nationale de la Recherche under Grant ANR-10-JCJC 0106.}



\end{document}